\documentclass{article}
\usepackage{inputenc}

\usepackage{amsmath,amsbsy,amssymb,amsthm,bbm,bm}
\newcommand{\bs}[1]{\boldsymbol{#1}}
\usepackage{multirow,multicol,tabularx,booktabs}
\usepackage{placeins,color,url}
\usepackage{bbm,bm}
\usepackage[ruled]{algorithm2e}
\usepackage[shortlabels]{enumitem}
\usepackage{wrapfig}
\usepackage{relsize}
\usepackage{subfigure}
\usepackage{caption}
\usepackage[draft]{fixme}
\fxsetup{layout=margin}
\fxusetheme{color}
\usepackage{xcolor}

\definecolor{jcgreen}{rgb}{0.09, 0.65, 0.35}

\usepackage{mathrsfs}
\usepackage{graphicx}
\usepackage[top=2.5cm,bottom=2.5cm,right=2.5cm,left=2.5cm]{geometry}
\usepackage[hidelinks]{hyperref}

\allowdisplaybreaks

\DeclareMathOperator*{\argmax}{arg\,max}
\DeclareMathOperator*{\argmin}{arg\,min}

\newtheorem{conj}{Conjecture}[section]
\newtheorem{thm}[conj]{\bf Theorem}
\newtheorem{defi}[conj]{\bf Definition}

\newtheorem{prop}[conj]{\bf Proposition}
\newtheorem{lemma}[conj]{\bf Lemma}

\newtheorem{innercustomgeneric}{\customgenericname}
\providecommand{\customgenericname}{}
\newcommand{\newcustomtheorem}[2]{%
  \newenvironment{#1}[1]
  {%
   \renewcommand\customgenericname{#2}%
   \renewcommand\theinnercustomgeneric{##1}%
   \innercustomgeneric
  }
  {\endinnercustomgeneric}
}

\newcustomtheorem{customthm}{Assumption}

\newcommand{\seven}{\bgroup
\sbox0{7}\usebox0\llap{\rule[.5\ht0]{.4\wd0}{.05\ht0}\rule{.24\wd0}{0pt}}
\egroup}

\newcommand{\PROB}{\mathbf{P}}

\newcommand{\Var}{\operatorname{Var}}

\newcommand{\sgn}{\operatorname{sgn}}

\def\bar{\overline}

\def\to{\rightarrow}

\def\id{\operatorname{id}}

\def\Var{\operatorname{Var}}

\def\supp{\operatorname{supp}}

\newcommand{\cro}[1]{\left[{#1}\right]}

\def\EE{ {\rm I} \kern-.15em {\rm E} }
\def\PP{ {\rm I} \kern-.15em {\rm P} }
\def\Var{{\mathbb{V}\text{ar}}}

\newcommand{\bp}{\mathbf{p}}

\newcommand{\bx}{\mathbf{x}}
\newcommand{\bX}{\mathbf{X}}
\newcommand{\by}{\mathbf{y}}

\newcommand{\mA}{\mathcal{A}}

\makeatletter
\newcommand{\ssymbol}[1]{^{\@fnsymbol{#1}}}

\def\1{\mathbbm{1}}

\def\wh{\widehat}

\def\ol{\overline}

\usepackage{import}
\usepackage{xifthen}
\usepackage{pdfpages}
\usepackage{transparent}

\title{A novel statistical approach to analyze image classification}
\author{Juntong Chen$\ssymbol{1}$, Sophie Langer$\ssymbol{3}$
and Johannes Schmidt-Hieber$\ssymbol{1}$ \\[0.15cm] 
\em {$\ssymbol{1}$Department of Applied Mathematics, University of Twente}\\ 
\em {$\ssymbol{3}$Faculty of Mathematics, Ruhr University Bochum}}

\date{}

\begin{document}

\listoffixmes
\clearpage

\maketitle

\begin{abstract}
The recent statistical theory of neural networks focuses on nonparametric denoising problems that treat randomness as additive noise. Variability in image classification datasets does, however, not originate from additive noise but from variation of the shape and other characteristics of the same object across different images. To address this problem, we introduce a tractable model for supervised image classification. While from the function estimation point of view, every pixel in an image is a variable, and large images lead to high-dimensional function recovery tasks suffering from the curse of dimensionality, increasing the number of pixels in the proposed image deformation model enhances the image resolution and makes the object classification problem easier. We introduce and theoretically analyze three approaches. Two methods combine image alignment with a one-nearest neighbor classifier. Under a separation condition, it is shown that perfect classification is possible. The third method fits a convolutional neural network (CNN) to the data. We derive a rate for the misclassification error that depends on the sample size and the complexity of the deformation class. An empirical study corroborates the theoretical findings.
\end{abstract}

\section{Introduction}
From a machine learning perspective, object recognition is typically framed as a high-dimensional classification problem, where each pixel is treated as an independent variable. The objective of the classification rule is to learn the functional relation between the pixel values of an input image and the corresponding conditional class probabilities or the labels.  However, for high-resolution images with many pixels, the domain of the function is a high-dimensional space, leading to slow convergence rates due to the curse of dimensionality. To align the strong empirical performance of convolutional neural networks (CNNs) with theoretical guarantees, one common approach is to assume that the true functional relationship between inputs and outputs has a latent low-dimensional structure, see, e.g., \cite{KKW20}. This assumption allows the convergence rate to depend only on this low-dimensional structure, potentially circumventing the curse of dimensionality. 

A functional data perspective is to treat images as highly structured objects that can be represented by a bivariate function, where each pixel value corresponds to a local average of the function over its location. From this viewpoint, variations of the same object in different images are interpreted as deformations of a template image, which introduces additional complexity for classification compared to the pixel-wise approach. This concept has been explored in foundational work on pattern recognition. Grenander and Mumford (\cite{G70, mumford_1996}) distinguish between \textit{pure images} and \textit{deformed images}, with the latter being generated from pure images through specific deformations. Since then, several generative models for object deformation on images have been proposed. For instance, \cite{5995635,M12,MR2957703} study a rich class of local deformations, while \cite{Mumford2000, MR2723182} extend these models to address more complex deformations such as noise, blur, multi-scale superposition and domain warping. Generative models are becoming increasingly important in fields such as medical image registration or computer vision \cite{6522524,ASHBURNER200795, perlay,pernet}. While existing work focuses on algorithms that can effectively handle image deformations, statistical modeling and theoretical generalization guarantees are, however, underexplored.

This paper aims to bridge this gap by introducing a tractable image deformation model that addresses a fundamental yet rich class of geometric transformations, including common variations in object positioning, scaling, brightness, and rotation. We focus on a binary image classification setting, where datasets consist of $n$ labeled images of two objects, such as the digits $0$ and $4$, with each image representing a random deformation of one of these objects. In the case of digits, these deformations can capture natural variations found in, for instance, individual handwriting. 

Our statistical analysis differs significantly from the wide range of well-understood classification problems that rely on local smoothing. In these settings, the source of randomness arises because the covariates (or inputs) do not fully determine the class label, requiring classifiers to aggregate training data with similar covariate values to effectively denoise. The resulting convergence rates for standard smoothness classes typically align with those seen in nonparametric regression and suffer from the curse of dimensionality for high-resolution images \cite{MR4406243}.

In the proposed statistical setting, the randomness occurs due to the different deformations that can arise on images within one class. The objective of the classification rule, therefore, is to remain invariant to these uninformative variations.

We approach this non-standard classification problem by first constructing classifiers exploiting the specific structure of the random object deformation model. These classifiers interpolate the data and can be interpreted as one-nearest neighbor classifiers in a transformed space. At low image resolutions, however, distinguishing between highly similar objects becomes impossible. We prove that if the two objects satisfy a separation condition, that depends on the image resolution, then, the classifiers can perfectly discriminate between the two classes on test data. Interestingly, the sample size $n$ plays a minor role in the analysis; it suffices to observe one training sample for each class. The imposed separation condition is also necessary in the sense that any smaller separation would result in non-identifiability of the classes, making accurate discrimination impossible (see Theorem \ref{thm27}).

A key contribution of this work are the misclassification error rates for CNN classifiers, showing that CNNs can adapt to various geometric deformations. As a first result, we prove that for a suitably chosen network architecture, specific parameter assignments in a CNN can effectively discriminate between the two classes. This shows that among all classifiers that are representable by a given CNN architecture, there exist classifiers that are (nearly) invariant with respect to the possible deformations of an object in an image. Based on this and statistical learning techniques, we derive misclassification bounds for CNN classifiers in Theorems \ref{thm4} and \ref{CE-bound}. For specific deformation classes, the obtained rates depend on the sample size and the number of pixels and the dependence on the input dimension is much more favorable than the curse of dimensionality observed in standard nonparametric convergence rates. The proposed setting has the potential to provide a more refined understanding of  phenomena such as overparametrization or the improved performance through data augmentation.

The article is structured as follows. In Section \ref{deformation-setting}, we introduce the image deformation model. As theoretical benchmarks, we introduce two classifiers for this deformation model in  Section~\ref{sec.image_process}. Section \ref{sec.CNN} analyzes a CNN-based classifier. The simulation study in Section~\ref{sec.num_sim} compares the three classifiers. A literature overview is provided in Section~\ref{defor-review}. We conclude in Section~\ref{s6} with a discussion of potential extensions and future research directions.

\textit{Notation:} For a real number $x$, $\lfloor x\rfloor$ represents the largest integer that is less than or equal to $x$, whereas $\lceil x\rceil$ represents the smallest integer that is greater than or equal to $x$. We denote vectors and matrices by bold letters, e.g., $ \mathbf{v}:=(v_1, \dots, v_d)$ and $ \mathbf{W} = (W_{i,j})_{i=1, \dots,m; j=1, \dots, n}$. As usual, $|\mathbf{v}|_{p} := (\sum_{i=1}^d |v_i|^p)^{1/p}$ and $| \mathbf{v}|_{\infty} := \max_i |v_i|.$ For a matrix $ \mathbf{W} = (W_{i,j})_{i=1, \dots,m; j=1, \dots, n},$ we define the maximum entry norm as $|\mathbf{W}|_{\infty} = \max_{i=1, \dots, m; j=1, \dots, n} |W_{i,j}|$. For two sequences $(a_n)_n$ and $(b_n)_n$, we write $a_n \lesssim b_n$ if there exists a constant $C$ such that $a_n \leq Cb_n$ for all $n$. For $m \geq 2$ and $a_1, a_2, \ldots, a_m$ we define $a_1 \vee a_2 \vee \dots \vee a_m = \max\{a_1,a_2, \dots, a_m\}$ and $a_1 \wedge a_2 \wedge \dots \wedge a_m = \min\{a_1, a_2, \dots, a_m\}$. For functions, $\|\cdot\|_{L^p(D)}$ denotes the $L^p$-norm on the domain $D.$ When $D=[0,1]^2,$ we also write $\|\cdot\|_p.$ 
For a function $A=(a_{1},a_{2}):\mathbb{R}^2\rightarrow\mathbb{R}^2$, we define $\|A\|_{L^{\infty}(D)}:=\max_{i=1,2}\sup_{ \mathbf{z}\in D}|a_i(\mathbf{z})|$ and set $\|A\|_{\infty}:=\|A\|_{L^{\infty}([0,1]^2)}$. For $B$ a set, the indicator function is denoted by $\1(x \in B)$. It takes the value $1$ if $x \in B$ and $0$ otherwise. Since we frequently work with bivariate functions, we write $f(\cdot,\cdot)$ for a function $(x,y)\mapsto f(x,y)$.

\section{Image deformation models}\label{deformation-setting}
We first discuss a specific case and then introduce the full image deformation model. For any integers $j,\ell\in\mathbb{Z}$, define \begin{equation*}
I_{j,\ell}=\Big[\frac{j-1}d,\frac jd\Big)\times\Big[\frac{\ell-1}d,\frac \ell d\Big),
\end{equation*}
representing a square with side length $1/d.$ A $d\times d$ image with $d^2$ pixels, as illustrated in Figure \ref{fig:foobar-1}, can be expressed as a bivariate function $$f:\mathbb{R}^2\to [0,\infty),$$ where the grayscale value of the $(j,\ell)$-th pixel is given by
\begin{equation}\label{ave-intensity}
\bar{f}_{j,\ell}=d^{2}\int_{I_{j,\ell}}f\left(u,v\right)dudv, \quad j,\ell\in\{1,\ldots,d\},
\end{equation}
representing the average intensity of $f$ on $I_{j,\ell}.$ The pixel value decodes the grayscale with smaller function values corresponding to darker pixels. To deal with image deformations, it is more convenient to define $f$ on $\mathbb{R}^2$ instead of $[0,1]^2.$

\begin{wrapfigure}{r}{.5\textwidth}
\centering
\vspace{-10pt}\includegraphics[width=0.4\textwidth]{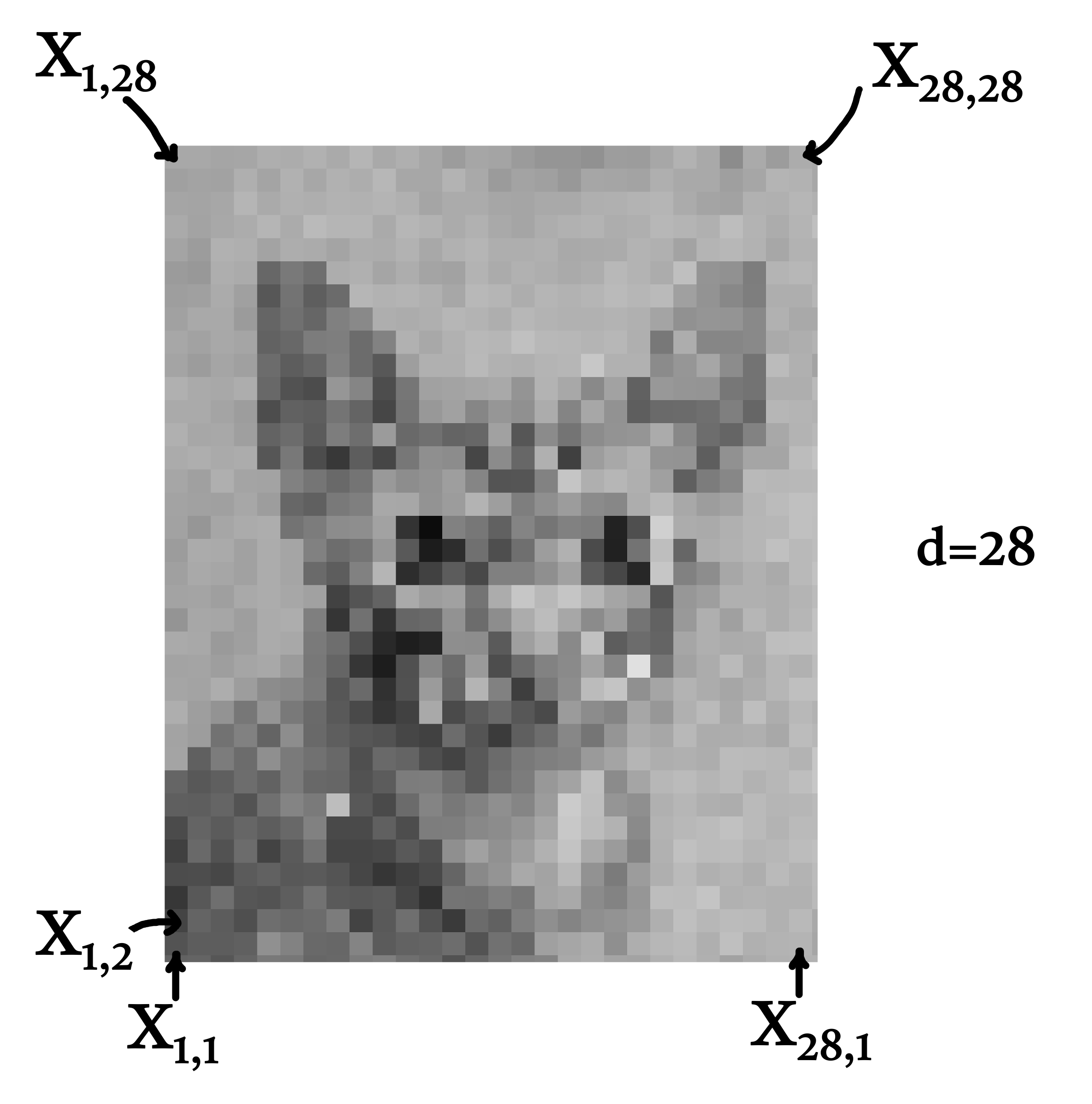}
\caption{Image represented by pixels}
\label{fig:foobar-1}
\end{wrapfigure}

The support of a function $g$ is defined as the set of all $x$ for which the function value $g(x)$ is non-zero. For a function $f$ representing an image, we refer to $f$ restricted to its support as the \textit{object}. The \textit{background} is defined as the complement of the support, that is, the set of $x$ with $f(x)=0$. Assuming that the images have zero background, all positive pixels are considered as part of the object itself. 

Next, we explore how simple transformations such as scaling, shifting, and brightness affect the function $f$, and consequently, the image. Multiplying the function values by a factor $\eta >1$ brightens the image,  making the pixel values appear whiter, while multiplying by $0<\eta<1$ darkens the image. Shifting the object within the image, either horizontally or vertically, corresponds to translating the function $f$ by a vector $(\tau,\tau'),$ changing the function values to $f(x-\tau,y-\tau').$ Stretching or shrinking the object along the $x-$ or $y-$axis transforms the function value to $f(\xi x, \xi' y)$, where $\xi<1$ or $\xi'<1$ stretches the object along the $x-$ or $y-$axis and $\xi, \xi'>1$ shrinks it. Combining these transformations, the function becomes $(x,y)\mapsto \eta f(\xi x-\tau,\xi' y-\tau')$, capturing the effects of brightness adjustment, translation, and scaling on the image. See Figure~\ref{fig:foobar-2} for an example of a deformed image of a cat under different scaling, shifting, and brightness adjustments.

\begin{wrapfigure}{r}{.5\textwidth}
\centering
\vspace{-10pt}\includegraphics[width=0.4\textwidth]{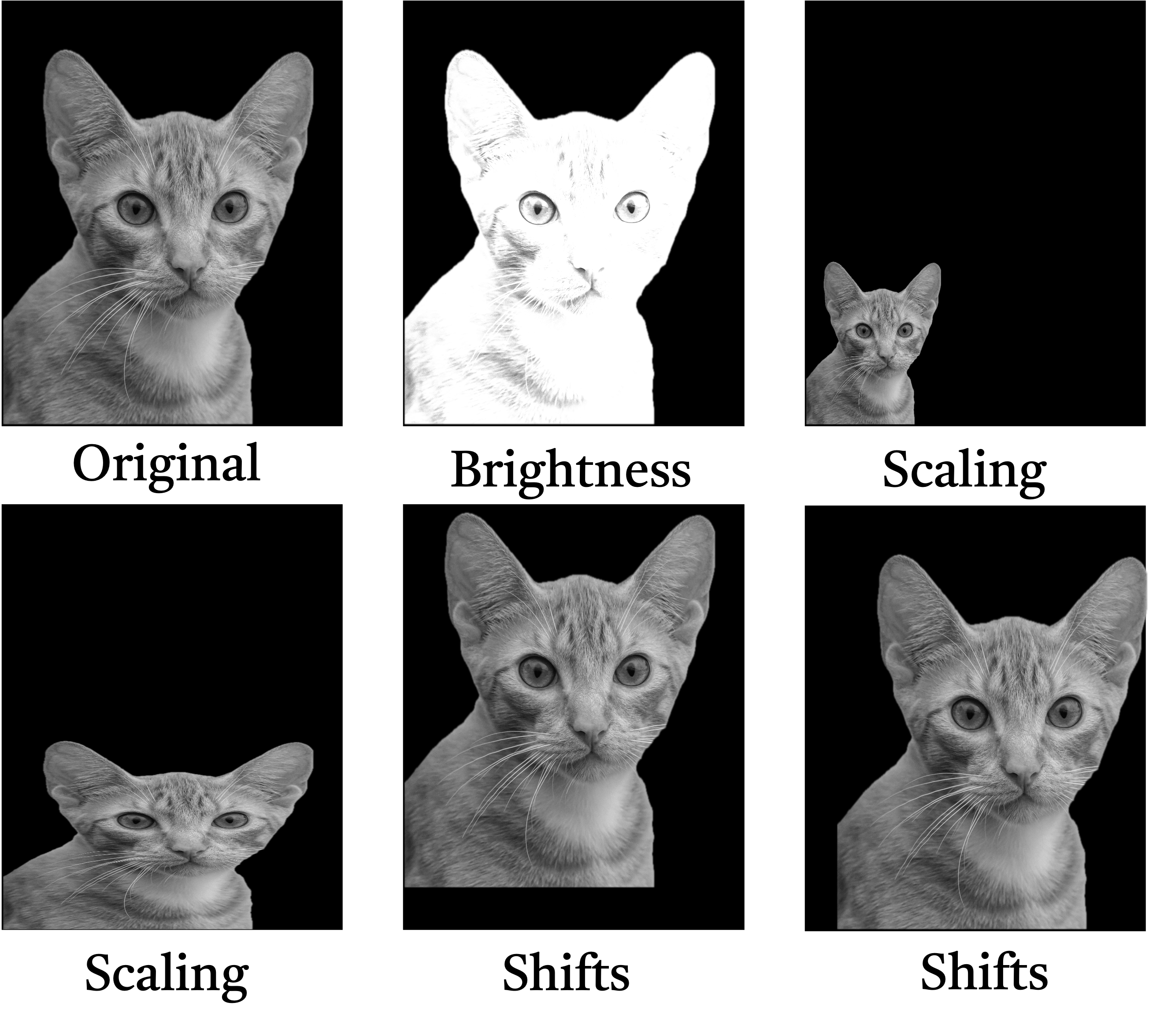}
\caption{Different deformations of a cat image.}
\label{fig:foobar-2}
\end{wrapfigure}

To distinguish between images of different object classes, the underlying idea of the data generating model is to assume that images from different classes correspond to different {\it template functions} $f$. By drawing the parameters $(\eta,\tau,\tau',\xi,\xi')$ randomly, each observed image in the dataset is then a random transformation of its corresponding template function. This means we observe $n$ independently generated pairs, each consisting of a $d\times d$ image and its corresponding class label. In the case of a supervised binary classification problem, these pairs are denoted by $(\bX_i,k_i)\in [0,\infty)^{d\times d} \times \{0,1\}.$ Here $k_i \in \{0,1\}$ is the $i$-th label and the $i$-th image is represented by a $d\times d$ matrix $\bX_i=(X^{(i)}_{j,\ell})_{j,\ell=1,\ldots,d}$ with entries
\begin{align}
    X^{(i)}_{j,\ell}= d^2\eta_i\int_{I_{j,\ell}}f_{k_i}\Big(\xi_i u - \tau_i, \xi^{\prime}_i v-\tau^{\prime}_i \Big) \, dudv,
    \label{eq.mod}
\end{align}
where $f_0,f_1$ are the two unknown template functions and $\eta_i,\xi_i, \xi_i^{\prime},\tau_i, \tau_i^{\prime}$ are unobserved independent random variables. Each image consists of $d^2$ pixels. The brightness factor $\eta_i$ is assumed to be positive. Throughout the article, we assume that the template functions $f_0,f_1$ are non-negative. This implies that all pixel values $X^{(i)}_{j,\ell}$ are also non-negative. 

In summary, model \eqref{eq.mod} generates images of the two objects using template functions $f_0, f_1$, where the shifts $\tau$, $\tau^{\prime}$, scaling factors $\xi$, $\xi^{\prime}$ and brightness $\eta$ are all random variables.

To extend model \eqref{eq.mod}, we introduce a more general framework in which the random transformations are deformations $A=(a_{1},a_{2}):\mathbb{R}^2\to \mathbb{R}^2$, belonging to a class of mappings $\mathcal{A}$.  In this generalized model, the deformed template function is expressed as
\begin{equation*}
f\circ A \mkern1mu (u,v) :=f\big(A(u,v)\big)=f\big(a_{1}(u,v),a_{2}(u,v)\big),
\end{equation*}
with $A \in \mathcal{A}$. Let $A_i$ denote the deformation applied to the $i$-th image in the sample. The image can then be represented by a $d\times d$ matrix $\bX_i=(X^{(i)}_{j,\ell})_{j,\ell=1,\ldots,d}$ with entries 
\begin{align}
X^{(i)}_{j,\ell}=d^{2}\eta_i\int_{I_{j,\ell}} f_{k_i} \circ A_i \mkern1mu (u,v) \, dudv.
\label{eq.mod-general}
\end{align}
Since each image can be viewed as an observation of a randomly deformed function, the proposed framework can be interpreted as a functional data analysis model adapted to image classification. This connection, along with its distinctions, will be further discussed in Section~\ref{defor-review}. For more on classification for functional data, see \cite{MR2168993, FDA_Review, ROSSI2006730, JacquesPreda, MR2899863}.

We now present some examples of specific deformation models. While these examples are parametric, the framework also allows for non-parametric models, see Section~\ref{sec.image_process}.

{\bf Affine transformations.} Affine transformations have been widely discussed in the fields of image processing and computer vision; see, e.g., \cite{fastaafine,reviewcnn}. The deformation is of the form
\begin{equation}
\label{affinetran}
    A(u,v)
    =\begin{pmatrix}
        b_1 & b_2\\
        b_3 & b_4
    \end{pmatrix}
    \begin{pmatrix}
        u\\
        v
    \end{pmatrix}
    -
    \begin{pmatrix}
        \tau\\
       \tau'
    \end{pmatrix},
\end{equation}
with real parameters $b_{1},\ldots,b_{4},\tau,\tau'.$ We recover model \eqref{eq.mod} as a special case by choosing deformations $b_1=\xi$, $b_4=\xi'$, and $b_2=b_3=0.$ Moreover, rotated, scaled and translated images $\bX_i=(X^{(i)}_{j,\ell})_{j,\ell=1,\ldots,d}$ with brightness adjustment can be described by composing a scaling in $x$- and $y$-direction with a rotation by an angle $\gamma\in[0,2\pi)$, that is, choosing
\begin{align}
    \begin{pmatrix}
    b_{1} & b_{2}\\
    b_{3} & b_{4}
    \end{pmatrix}=\begin{pmatrix}
    \xi\cos\gamma & -\xi'\sin\gamma\\
    \xi\sin\gamma & \xi'\cos\gamma
    \end{pmatrix}=\begin{pmatrix}
    \cos\gamma & -\sin\gamma\\
    \sin\gamma & \cos\gamma
    \end{pmatrix}\begin{pmatrix}
    \xi & 0\\
    0 & \xi'
    \end{pmatrix}.
    \label{eq.87etgb}
\end{align}

{\bf Nonlinear deformations.} Choosing $a_{1}(u,v) = u-\tau(u,v)$ and $a_{2}(u,v) = v-\tau'(u,v)$ for bivariate Lipschitz-continuous functions $\tau(u,v)$ and $\tau'(u,v)$ generates a class of nonlinear and local deformations \cite{5995635,M12,MR2957703}. Of particular interest among nonlinear deformations are wave-like deformations 
\begin{align*}
    a_1(u,v) = u + \alpha \sin(2\pi v / \lambda), \quad \text{and} \ \ a_2(u,v) = v,    
\end{align*}
which are used to model periodic textures, noise patterns, image warping, and spatial distortions such as those caused by lens aberrations \cite{processing,MR2723182,comvis}. The parameter $\alpha$ describes the amplitude of the wave-like deformation and $\lambda\not=0$ controls the wavelength.

\section{Classification via inverse mapping and image alignment}\label{sec.image_process}
We construct and analyze classifiers that are specifically tailored to the proposed image deformation models \eqref{eq.mod} and \eqref{eq.mod-general}.

\subsection{Classification via inverse mapping}\label{ima-inverse-sec}
Under the general deformation model \eqref{eq.mod-general}, each image $\bX=(X_{j,\ell})_{j, \ell=1, \dots, d}$ is generated by  
\begin{align*}   X_{j,\ell}=d^{2}\eta\int_{I_{j,\ell}}f \circ A \mkern1mu (u,v) \, dudv, 
\end{align*}
with $f$ the template function and $A$ the transformation modeling the deformation. Throughout this section, we assume $A$ is invertible. 

To define the classifier, we also interpret an image $\bX$ as a bivariate function on $\mathbb{R}^2$, via
\begin{equation}\label{image-as-function}
\bX(u,v):=\sum_{j,\ell\in\mathbb{Z}}X_{j,\ell}\1\big((u,v)\in I_{j,\ell}\big)
\end{equation}
for all $(u,v) \in \mathbb{R}^2$,
and setting $X_{j,\ell}:=0$ if $j,\ell\notin\{1,\ldots,d\}$. This means that $\bX$, viewed as a function, assigns to any point within the pixel its corresponding pixel value. As the random deformations $A\in\mathcal{A}$ do not contain information about the class label, a classifier should not depend on these deformations. To achieve this, we consider the set of inverse transformations $\mathcal{A}^{-1} = \{A^{-1} : A \in \mathcal{A}\}$. For computational feasibility, instead of using $\mathcal{A}^{-1}$ directly, we approximate it by a discretized subset $\mathcal{A}_d^{-1}$, which covers $\mathcal{A}^{-1}$ with balls of radius $1/d$ on a given domain $D_\mathcal{A}$, meaning that for any $A^{-1}\in\mathcal{A}^{-1}$, there exists a transformation $B\in\mathcal{A}_d^{-1}$ such that
\begin{equation}\label{gen-conver}
||A^{-1}-B||_{L^{\infty}(D_{\mathcal{A}})}\leq\frac{1}{d}.
\end{equation} 
To obtain theoretical guarantees, one needs to choose the domain $D_{\mathcal{A}}$ sufficiently large and depending on the regularity conditions imposed on the deformation class.

To construct the classifier, we apply each transformation $B\in \mathcal{A}^{-1}_{d}$ to the input of the bivariate image function and obtain
\begin{equation}\label{discre-inverse-deform}
\bX \circ B(u,v):=\mathbf{X}(B(u,v)) = \sum_{j,\ell \in \mathbb{Z}} X_{j,\ell}\1\big(B(u,v) \in I_{j,\ell}\big),\quad\mbox{for all}\  (u,v)\in\mathbb{R}^2.
\end{equation}
Given that any possible deformation $A$ is invertible, there exists an approximate inverse $B \in \mathcal{A}_d^{-1}$, such that $B\approx A^{-1}.$ For such a mapping $B,$ $\bX \circ B$ should generate a nearly deformation-free representation of the image. To account for the effects of the image brightness factor $\eta$,
we normalize the pixel values and obtain
\begin{align}
    T_{\bX\circ B}:=\frac{\bX \circ B}{\|\bX \circ B\|_{L^2(\mathbb{R}^2)}}.
    \label{eq.T_def-general}
\end{align}

Combining these steps, the proposed classifier $\wh k$ assigns a label to the new image $\bX$ by first applying all possible transformations from $\mathcal{A}_d^{-1}$ to both, the new image and each training image $\bX_i$. It then finds the training image whose transformed version $T_{\bX_i \circ B_i}$ is in Euclidean distance the closest fit to the transformed version $T_{\bX \circ B}$ of the new image. The label of this closest matching training image is assigned to the new image. This defines the  \textit{inverse mapping classifier}
\begin{align}
    \wh k:=k_{\hat{i}}, \quad \text{with} \ \ \wh i \in \argmin_{i\in\{1,\dots,n\}} \ \min_{B_{i},B\in\mathcal{A}^{-1}_{d}} \, \big\|T_{\bX_{i}\circ{B_{i}}}-T_{\bX\circ B}\big\|_{L^2(\mathbb{R}^2)}\label{gen-label}.
\end{align}
It can be interpreted as an one-nearest neighbor estimator in a transformed space.

We now state the assumptions for the statistical analysis. To make the theory tractable, we impose a Lipschitz condition on the template function.  

\begin{customthm}{1}\label{ass-1}
The supports of the two template functions $f_0, f_1$ are contained in a rectangle $[\beta_{\text{left}},\beta_{\text{right}}]\times[\beta_{\text{down}},\beta_{\text{up}}]\subseteq\cro{0,1}^2$. Additionally, $f_0, f_1$ are Lipschitz continuous, in the sense that there exists a positive constant $C_L$ such that for any real numbers $u,v,u',v',$
\begin{align}
    |f_k(u,v)-f_k(u',v')|\leq C_L \|f_k\|_1 \big(|u-u'|+|v-v'|\big), \quad k=0,1.
    \label{eq.Lip_cond-gen}
\end{align}
\end{customthm}

We also impose conditions to prevent the deformation from moving (part of) the object outside the image. Assumption~\ref{ass-1} ensures that the support of $f$ is contained in $[\beta_{\text{left}}, \beta_{\text{right}}] \times [\beta_{\text{down}},\beta_{\text{up}}]\subseteq\cro{0,1}^2$. The deformed object remains fully visible, if the support of the deformed function $f\circ A$ also lies in $[0,1]^2.$ This is the case if $[\beta_{\text{left}}, \beta_{\text{right}}] \times [\beta_{\text{down}},\beta_{\text{up}}]\subseteq A([0,1]^2).$

\begin{customthm}{2}\label{ass1}
(i). The class $\mathcal{A}$ contains the identity. For any $A=(a_{1},a_{2})\in \mathcal{A}$, $[\beta_{\text{left}}, \beta_{\text{right}}] \times [\beta_{\text{down}},\beta_{\text{up}}]\subseteq A([0,1]^2),$ and the functions $a_1,a_2$ have continuous partial derivatives on $\mathbb{R}^2$, bounded in the supremum norm by a constant $C_{\mathcal{A}}$. (ii). Assume $C_{\mathcal{J}}^2:=\inf_{A\in \mathcal{A}} \, \inf_{(u,v)\in\mathbb{R}^2}\big|\det(J_{A}(u,v))\big|>0$ with $J_A$ the Jacobian matrix of $A$.
\end{customthm}

To ensure correct classification, we also need to guarantee that two images from different object classes cannot be represented as transformations of the template function corresponding to the opposite class, as otherwise, distinguishing between the two classes becomes impossible. 

To formalize this separation between the two object classes with template function $f_0, f_1$, we introduce the separation quantity 
\begin{align}
D := D(f_0,f_1) \vee D(f_1,f_0),\  \mbox{with}\  D(f,g):=\frac{\inf_{a\in \mathbb{R},\,\{A_i\}_{i=1}^{4}\subseteq\mathcal{A}} \|a f\circ A_1\circ A_2^{-1}-g\circ A_3\circ A_{4}^{-1}\|_{L^2(\mathbb{R}^2)}}{\|g\|_{L^2(\mathbb{R}^2)}}.
\label{eq.D_part2_gen}
\end{align}
The quantity $D(f,g)$ measures the normalized minimal $L^2$-distance between all possible deformations of $g$ and $f$ under transformations from $\mathcal{A}\circ\mathcal{A}^{-1}$. When $\mathcal{A}$ forms a group, this expression simplifies to
$D(f,g)=\inf_{a\in \mathbb{R},\,A,A'\in \mathcal{A}} \|a f\circ A-g\circ A'\|_{L^2(\mathbb{R}^2)}/\|g\|_{L^2(\mathbb{R}^2)},$ measuring the normalized minimal $L^2$-distance between all possible deformations of $g$ and $f$ under transformations from $\mathcal{A}$. 

\begin{thm}
\label{thm.main_part1_general}
Let $(\bX, k), (\bX_1, k_1), \dots, (\bX_n, k_n)$ be defined as in \eqref{eq.mod-general}. Suppose that the labels $0$ and $1$ occur at least once in the training data, that is, $\{i:k_i=0\}\neq \varnothing$ and $\{i:k_i=1\}\neq \varnothing$. Assume moreover that $f_0$ and $f_1$ satisfy Assumption \ref{ass-1} with Lipschitz constant $C_L$ and that Assumption~\ref{ass1} holds with constants $C_{\mathcal{A}}, C_{\mathcal{J}}$. If $D_{\mathcal{A}}=[-2C_{\mathcal{A}} - 1, 2C_{\mathcal{A}} + 1]^2$ in \eqref{gen-conver} and 
\begin{align}
\label{condition}
D>C(C_L,C_\mathcal{A},C_{\mathcal{J}})/d,
\end{align}
 where $D$ is as defined in \eqref{eq.D_part2_gen}, and $C(C_L,C_\mathcal{A},C_{\mathcal{J}})$ is a sufficiently large constant only depending on $C_{L},C_\mathcal{A},C_{\mathcal{J}}$, then the classifier $\wh{k}$ defined in \eqref{gen-label} will recover the correct label, that is,
\begin{align*}
    \wh{k}=k.
\end{align*}
\end{thm}

The proof of Theorem~\ref{thm.main_part1_general} is deferred to Section~\ref{app.proofs_image_alignm_gen}. The result shows that classifier \eqref{gen-label} guarantees perfect classification of any given image if the separation quantity satisfies $D>C/d$ for a sufficiently large constant $C,$ and the dataset contains at least one image from each class. As $d\to \infty,$ $C/d\to 0$ and the condition $D>C/d$  holds for all sufficiently large $d$. This aligns with the intuition that a minimal image resolution is necessary for classification. The following result provides a simple tool to check condition \eqref{condition} if $\mathcal{A}$ is a group (see also Figure~\ref{fig:SA_visualization}). The proof is given in Section~\ref{app.proofs_image_alignm_gen}.
\begin{lemma}\label{example-lemma}
Let $f_0$, $f_1$ be two template functions and let the deformation set $\mathcal{A}$ be a group satisfying Assumption~\ref{ass1}-(i), with constant $C_{\mathcal{A}}$. Then, condition \eqref{condition} is satisfied, whenever $$d>\sqrt{2}C\big(C_L,C_\mathcal{A},C_{\mathcal{J}}\big)C_{\mathcal{A}} \sup_{A\in \mathcal{A}} \, \frac{\|f_1\|_{L^2(\mathbb{R}^2)}}{\|f_1\|_{L^2((\supp(f_0\circ A))^c)}}.$$
\end{lemma}
\begin{figure}[t]
    \centering
\includegraphics[width=0.8\textwidth,height=4cm]{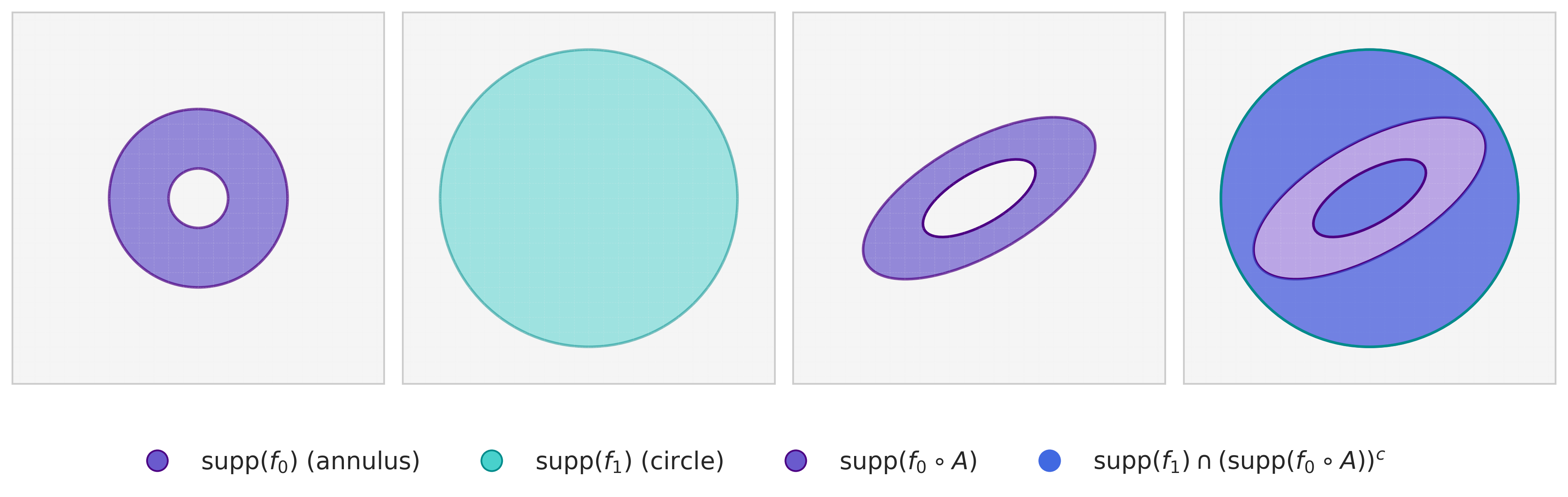}
\caption{Illustration of why the separation distance can be bounded from below as in Lemma~\ref{example-lemma} if deformations of the template functions cannot map the support on each other.}
\label{fig:SA_visualization}
\end{figure}

Separation under distance $D$ is different from the Euclidean distance-based criterion that is typically employed in the 1-nearest neighbor method. The Euclidean distance is significantly more rigid than the proposed distance metric $D$, making it inherently less suitable for classification in deformation-sensitive contexts. If, for example, the deformation set $\mathcal{A}$ includes random shifts, then two shifted functions $f_i \circ A$ and $f_i \circ A'$ based on the same template function $f_i$ will have a small $D$-distance, but their Euclidean distance may be large if $A$ and $A'$ shift in different directions. 

The classifier can account for a wide range of image deformations, but discretization of the entire set of inverse mappings can be computationally demanding. Therefore, we consider this classifier more as a theoretical benchmark for the general image deformation model \eqref{eq.mod-general} rather than an effective method for practical applications. When dealing with specific deformation models, the computational demands can be significantly reduced. For instance, if the deformation class $\mathcal{A}$ is a group, then $\mathcal{A}^{-1}=\mathcal{A}$ and we can instead consider the classifier
\begin{align*}
    \wh k:=k_{\hat{i}}, \quad \text{with} \ \ \wh i \in \argmin_{i\in\{1,\dots,n\}} \ \min_{A\in\mathcal{A}_{d}} \, \big\|T_{\bX_{i}}-T_{\bX\circ A}\big\|_{L^2(\mathbb{R}^2)},
\end{align*}
where $\mathcal{A}_{d}$ denotes an $1/d$-covering of $\mathcal{A}$.

\begin{wrapfigure}{r}{.5\textwidth}
\centering\vspace{-2pt}
\includegraphics[width=0.3\textwidth,height=2.8cm]{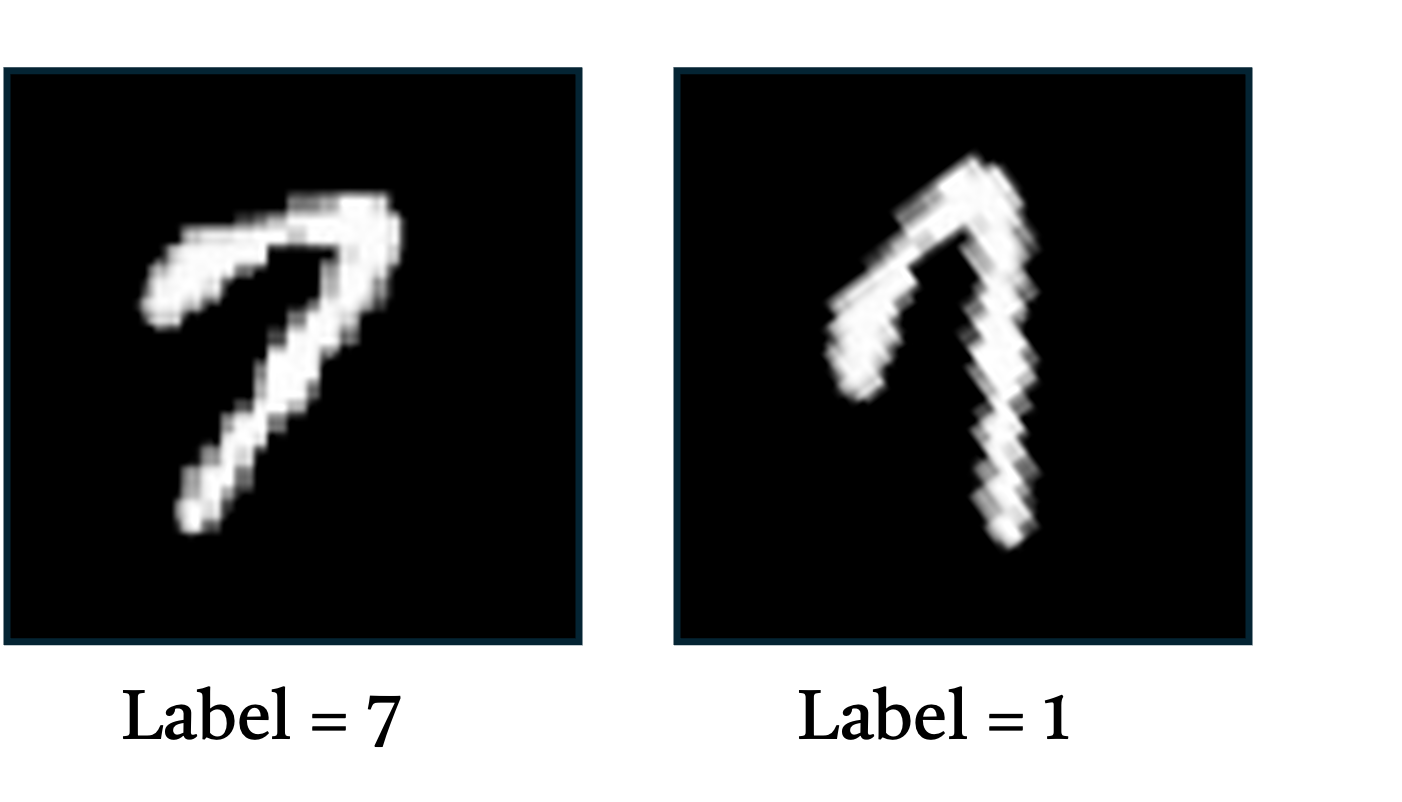}
\caption{Deformed MNIST images of digits $7$ and $1.$}
\label{fig:separation}
\end{wrapfigure}

This classifier resembles traditional registration techniques in medical imaging and computer vision. Given a (moving) source image $\mathrm{S}$ and a target image $\mathrm{T}$, the goal of image registration is to find a map $\phi: \mathbb{R}^2 \rightarrow \mathbb{R}^2$ such that $\mathrm{S} \circ \phi(\mathbf{x}) \approx \mathrm{T}(\mathbf{x})$ \cite{6522524, 2025103385}. This is typically formulated as an optimization problem
\begin{equation}\label{registr-obj}
\wh\phi=\argmin_{\phi} \, L(\mathrm{S} \circ \phi,\mathrm{T})+R(\phi),  \end{equation}
where $L$ is a loss function measuring similarity, and $R(\cdot)$ is a regularizer. The regularizer can be omitted for low-dimensional transformation models, such as rigid or affine transformations \cite{YANG2017378}. 

In Theorem~\ref{thm27}, we show that if the set $\mathcal{A}$ is sufficiently rich, the considered separation criterion $D\gtrsim 1/d$, as defined in \eqref{eq.D_part2_gen},
is optimal. This means that any smaller bound could result in deformed versions of one template function being representable by a template function of the opposite class, thereby making it impossible to distinguish between the two classes; see Figure \ref{fig:separation} for an illustration of two deformed MNIST images with labels $7$ and $1$ that can be transformed into each other by a rotation, making classification ambiguous.

We now verify the imposed assumptions for the specific deformation models introduced in Section~\ref{deformation-setting}, with proofs of the following lemmas provided in Section~\ref{app.proofs_image_alignm_gen}. Let $y_{+}:=\max\{y,0\}.$
\begin{lemma}\label{ex_scale}
Let $[\beta_{\text{left}}, \beta_{\text{right}}] \times [\beta_{\text{down}}, \beta_{\text{up}}]=[1/4,3/4]\times[1/4,3/4]$. The class of deformations in \eqref{eq.mod} with $1/2\leq |\xi|, |\xi'|\leq C_{\mathcal{A}},$ $|\tau|,|\tau'|\leq\ell_{s},$
\begin{align}
-(-\xi)_{+}-\frac{1}{4}\leq\tau\leq\xi_{+}-\frac{3}{4}\quad \text{and} \quad -(-\xi')_{+}-\frac{1}{4}\leq\tau'\leq\xi'_{+}-\frac{3}{4},
\label{eq.8e7gf}
\end{align}
satisfies Assumption \ref{ass1} with constants $C_{\mathcal{A}}$ and $C_{\mathcal{J}}=1/2$. Moreover, there exists a $1/d$-covering of $\mathcal{A}^{-1}$ with cardinality $|\mathcal{A}_d^{-1}|\asymp d^{4}$. 
\end{lemma}

\begin{lemma}
\label{lem.rot}
The deformation model described in \eqref{affinetran} and \eqref{eq.87etgb} is
\begin{align*}
A(u,v) =\begin{pmatrix}
    \cos\gamma & -\sin\gamma\\
    \sin\gamma & \cos\gamma
    \end{pmatrix}\begin{pmatrix}
    \xi & 0\\
    0 & \xi'
    \end{pmatrix}
    \begin{pmatrix}
        u\\
        v
    \end{pmatrix}
    -
    \begin{pmatrix}
        \tau\\
       \tau'
    \end{pmatrix}.
\end{align*}
Let $[\beta_{\text{left}}, \beta_{\text{right}}] \times [\beta_{\text{down}}, \beta_{\text{up}}]=[1/4,3/4]\times[1/4,3/4]$. The class of deformations with 
$1/2\leq |\xi|, |\xi'|\leq C_\mathcal{A},$ $|\tau|,|\tau'|\leq\ell_{s},$ $\gamma\in[0,\pi/2)$, and 
$$-(-\xi)_{+}-\frac{1}{4}(\cos\gamma+\sin\gamma)\leq\tau\cos\gamma+\tau'\sin\gamma\leq\xi_{+}-\frac{3}{4}(\cos\gamma+\sin\gamma),$$
$$-(-\xi')_{+}-\frac{1}{4}\cos\gamma+\frac{3}{4}\sin\gamma\leq\tau'\cos\gamma-\tau\sin\gamma\leq\xi'_{+}-\frac{3}{4}\cos\gamma+\frac{1}{4}\sin\gamma,$$ satisfies Assumption~\ref{ass1} with constants $C_\mathcal{A}$, $C_{\mathcal{J}}=1/2$ and there exists a $1/d$-covering of $\mathcal{A}^{-1}$ with $|\mathcal{A}_d^{-1}|\asymp d^5$.
\end{lemma}

Consider the non-linear deformations $A = (a_1, a_2) \in \mathcal{A}$, 
\begin{equation}\label{gen-non-linear-mod}
a_1(u, v) = h_1(u,v)\quad\mbox{and}\quad a_2(u, v) = h_2(v),    
\end{equation}
where $h_2(v)$ is strictly monotone with respect to $v$, and for any fixed $v \in \mathbb{R}$, $h_1(u, v)$ is strictly monotone with respect to $u$. In this case, the transformation is invertible because one can always retrieve $v$ from $h_2(v)$ and then $u$ from $h_1(u, v).$ A specific example is the wave-like deformation model discussed in Section~\ref{deformation-setting}. The following describes the construction of $\mathcal{A}_d^{-1}$.
\begin{lemma}\label{wave-verify}
Let $[\beta_{\text{left}}, \beta_{\text{right}}] \times [\beta_{\text{down}}, \beta_{\text{up}}]=[1/4,3/4]\times[1/4,3/4]$. For the class of deformations
$$A(u,v)=\big(a_1(u,v),a_2(u,v)\big)=\big(u+ \alpha \sin(2\pi v / \lambda),v\big),$$ with $|\lambda|\geq C_{\text{lower}}>0,$ and $|\alpha|\leq1/4,$ Assumption~\ref{ass1} holds with $C_{\mathcal{A}}=\max\{\pi/(2C_{\text{lower}}),1\}$ and $C_{\mathcal{J}}=1.$ Moreover, there exists a $1/d$-covering $\mathcal{A}_{d}^{-1}$ of $\mathcal{A}^{-1}$ with cardinality $|\mathcal{A}_{d}^{-1}|\asymp d^2$.
\end{lemma}

For the general nonlinear model \eqref{gen-non-linear-mod}, the set $\mathcal{A}_d^{-1}$ can be constructed analogously to Lemma~\ref{wave-verify}. Specifically, if $a_1 \in \mathcal{F}_1$, $a_2 \in \mathcal{F}_2$, with $\mathcal{F}_1,\mathcal{F}_2$ being classes of functions satisfying \eqref{gen-non-linear-mod}, and $\mathcal{F}_1^{\delta} \subseteq \mathcal{F}_1$ and $\mathcal{F}_2^{\delta} \subseteq \mathcal{F}_2$ denoting $\delta$-coverings with respect to the sup-norm, then $$\mathcal{A}_{\delta}:=\big\{(a_1,a_2):\;a_1\in\mathcal{F}_1^{\delta},\ a_2\in\mathcal{F}_2^{\delta}\big\}$$ forms a $\delta$-covering of $\mathcal{A}=\{(a_1,a_2):\ a_1\in\mathcal{F}_1,\ a_2\in\mathcal{F}_2\}$. If for all $A=(a_1, a_2) \in \mathcal{A}$, we have $|\partial_u a_1(u,v)|, |\partial_v a_2(u,v)| \geq \mathcal{K} > 0$, then $\mathcal{A}^{-1}$ admits a $\delta' = \delta / \mathcal{K}$-covering, formed by inverting the elements of $\mathcal{A}_{\delta}$.

The next lemma considers compositions of deformation classes, with the proof provided in Section \ref{app.proofs_image_alignm_gen}.  This allows to verify Assumption~\ref{ass1} for more involved deformation classes.

\begin{lemma}\label{composite-lemma}
Let $\mathcal{A}_{1}$ and $\mathcal{A}_{2}$ satisfy Assumption~\ref{ass1} with respective constants $C_{\mathcal{A}_1}$, $C_{\mathcal{J}_1}$ and $C_{\mathcal{A}_2},$ $C_{\mathcal{J}_2}$. If for any $A_{1}\in\mathcal{A}_{1}$, $\cro{0,1}^2\subseteq A_1(\cro{0,1}^2)$, then, the composite deformation class $\mathcal{A}_{2}\circ\mathcal{A}_{1}:=\{A_{2}\circ A_{1},\;A_{1}\in\mathcal{A}_{1},A_{2}\in\mathcal{A}_{2}\}$ satisfies Assumption~\ref{ass1} with constants $C_{\mathcal{A}}=2C_{\mathcal{A}_1}C_{\mathcal{A}_2}$ and $C_{\mathcal{J}}=C_{\mathcal{J}_1}C_{\mathcal{J}_2}$. 
\end{lemma}
\subsection{Classification via image alignment}\label{sec.image_alignment}
We now focus on the specific image deformation model \eqref{eq.mod} that incorporates random scaling, shifts, and brightness adjustment. An image 
$\bX=(X_{j,\ell})_{j, \ell =1,\ldots,d}$ is then generated by
\begin{align}
X_{j,\ell}=d^2\eta\int_{I_{j,\ell}}f\big(\xi u - \tau, \xi^{\prime} v-\tau^{\prime} \big) \, dudv,\label{eq.X_f_model}
\end{align}
with ($\eta, \xi,\xi',\tau,\tau')$ the random deformation parameters. In this setting, one can find a transformation that aligns the images, in the sense that the transformed images are nearly independent of the deformation parameters. We propose a one-nearest-neighbor classifier based on the aligned training and test images. This approach is similar to curve registration in functional data analysis, see for instance \cite{MR3432837}. The classifier can be efficiently computed but relies on this specific deformation model. 

The first step of the construction is to approximately detect the object within the image by identifying the smallest axis-aligned rectangle that contains all non-zero pixel values; see the left image in Figure \ref{fig:Z} for an illustration. We refer to this as the \textit{rectangular support}. To determine the rectangular support, we denote the smallest and largest indices corresponding to the non-zero pixels in the image by
\begin{align}
    j_{-}:=\argmin \big\{j: X_{j,\ell}>0\big\}, \quad 
    j_{+}:=\argmax \big\{j: X_{j,\ell}>0\big\}
    \label{eq.j_pm_def}
\end{align}
and
\begin{align}
    \ell_{-}:=\argmin \big\{\ell: X_{j,\ell}>0\big\}, \quad 
    \ell_+:=\argmax \big\{\ell: X_{j,\ell}>0\big\}.
    \label{eq.l_pm_def}
\end{align}
The rectangular support of the image is then given by the rectangle $[j_-/d,j_+/d]\times [\ell_-/d,\ell_+/d]$. Similarly, we define the rectangular support of a function as the smallest rectangle containing the support. From the definition of the model \eqref{eq.mod}, it follows that the rectangular support of the image $\bX$ should be close to the rectangular support of the underlying deformed function $f(\xi\cdot-\tau,\xi'\cdot-\tau')$.

We now rescale the rectangular support of the image to the unit square $[0,1]^2.$ The line $[0,1] \ni t\mapsto j_{-}+t(j_+-j_-)$ starts at $j_-$ and ends for $t=1$ at $j_+$. We define the rescaled pixel values as
\begin{align}
    Z_{\bX}(t,t^{\prime}):=X_{\lfloor j_{-}+t(j_+-j_-)\rfloor,\lfloor \ell_{-}+t^{\prime}(\ell_+-\ell_-)\rfloor},
    \label{eq.Z_def}
\end{align}
with $\lfloor \cdot \rfloor$ the floor function. The function $(t,t^\prime) \mapsto Z_{\bX}(t,t^{\prime})$ runs through the pixel values on the rectangular support, now rescaled to the unit square $[0,1]^2;$ see the middle image of Figure \ref{fig:Z} for an illustration. This rescaling makes $Z_{\bX}(t,t^{\prime})$ approximately invariant to random shifts and scalings of the image, up to smaller-order effects.

\begin{wrapfigure}{r}{.5\textwidth}
\centering\vspace{-8pt}
\includegraphics[width=0.45\textwidth]{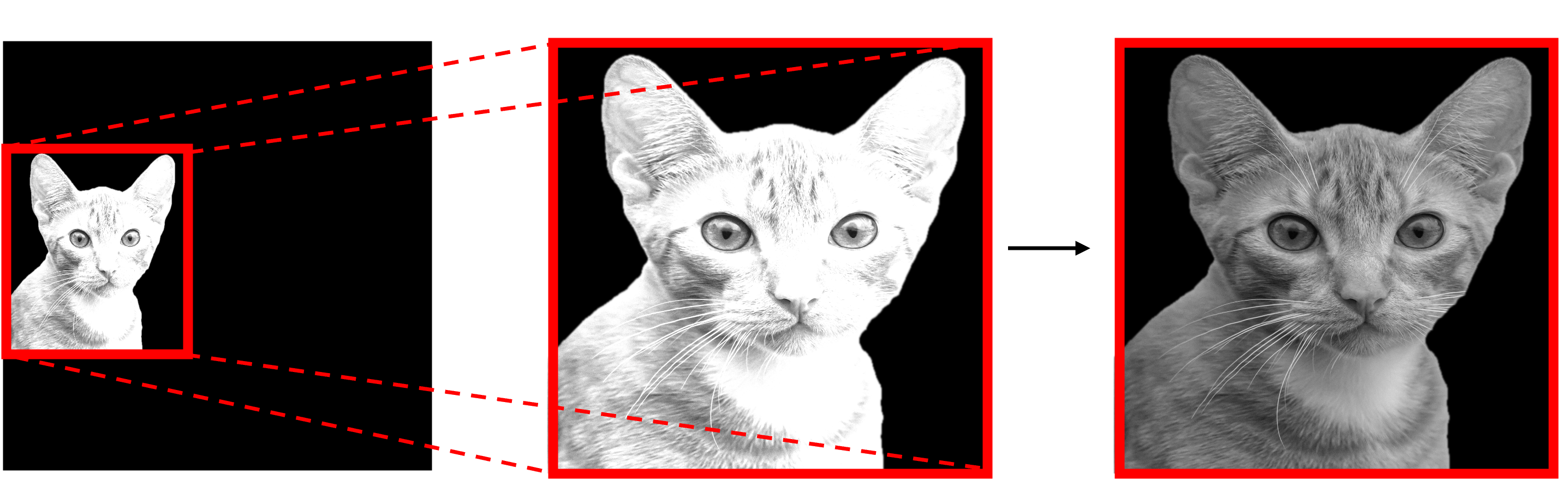}
\caption{$\bX,$ $Z_\bX,$ and $T_\bX$.}
\label{fig:Z}
\end{wrapfigure}
To find a quantity that is independent of the brightness adjustment $\eta,$ we normalize the pixel values by $Z_{\bX}/\|Z_{\bX}\|_2.$ The image alignment transformation is then given by
\begin{align}
T_\bX:=\frac{Z_{\bX}}{\|Z_{\bX}\|_2};
\label{eq.T_def}
\end{align}
see again Figure \ref{fig:Z} for an illustration. Based on these aligned and normalized images, we study a one-nearest-neighbor classifier $\wh k$ that assigns the label of $\bX_i$ from the training set to $\bX$, where $T_{\bX_i}$ is the closest to $T_\bX$. The \textit{image alignment classifier} is defined as 
\begin{align}
\wh k:=k_{\hat{i}}, \quad \text{with} \ \ \wh i \in \argmin_{i\in\{1,\dots,n\}} \, \big\|T_{\bX}-T_{\bX_i}\big\|_2.
    \label{eq.k_def}
\end{align}
This is an interpolating classifier, in the sense that if the new $\bX$ coincides with one of the images in the training set $\bX_i,$ then, $T_{\bX}=T_{\bX_i}$, $\wh i=i,$ and $\wh k=k_i.$

To study this model, we assume that $\xi, \xi' \geq 1/2$. Applying Lemma \ref{ex_scale} leads to the following assumption on the parameters to ensure full visibility of the objects on the deformed images.
\begin{customthm}{2'}\label{ass1-prime}
The supports of the two template functions $f_0, f_1$ are contained in $[1/4,3/4]^2$, and the random parameters $(\tau,\tau',\xi,\xi')$ satisfy $\xi,\xi^{\prime}\geq 1/2,$
$$-\frac{1}{4}\leq\tau\leq\xi-\frac{3}{4},\quad\mbox{and}\quad -\frac{1}{4}\leq\tau'\leq\xi'-\frac{3}{4}.$$
\end{customthm}

Assumption \ref{ass1-prime} indicates that the range of possible shifts $\tau, \tau'$ increases as $\xi, \xi' $ become larger. This is reasonable, as larger values of $\xi,\xi'$ shrink the object. Consequently, larger shifts $\tau, \tau'$ can be applied without moving parts of the object out of the image. 

Under the deformation model \eqref{eq.mod}, we have $f \circ A \mkern1mu (u,v)=f(\xi u-\tau,\xi'v-\tau')$. Based on this, we consider the separation quantity $D=D(f_0,f_1) \vee D(f_1,f_0)$ with
\begin{align}
D=D(f_0,f_1) \vee D(f_1,f_0),\quad\mbox{with}\quad D(f,g):=\frac{\inf_{a,b,c,b',c' \in \mathbb{R}} \|a f\big(b\cdot-c, b' \cdot -c'\big) -g\|_{L^2(\mathbb{R}^2)}}{\|g\|_{L^2(\mathbb{R}^2)}}.
\label{eq.D_part2}
\end{align}
It measures the normalized minimal $L^2$-distance between the function $g$ and any potential deformation of the function $f$ due to rescaling, shifting, and change in brightness.

\begin{thm}
\label{thm.main_part1}
Let $(\bX, k), (\bX_1, k_1), \dots, (\bX_n, k_n)$ be defined as in \eqref{eq.mod}. Suppose that the labels $0$ and $1$ occur at least once in the training data, that is, $\{i:k_i=0\}\neq \varnothing$ and $\{i:k_i=1\}\neq \varnothing$. Assume moreover that $f_0$ and $f_1$ satisfy Assumption \ref{ass-1} with Lipschitz constant $C_L$ and that Assumption~\ref{ass1-prime} holds. Set $\Xi_n := \max\{1,\xi,\xi^{\prime},\xi_1,\xi^{\prime}_1,\ldots,\xi_n,\xi^{\prime}_n\}$. If $D>4K (C_L\vee C_L^{2})\Xi_n^2/d$, with $D$ as defined in \eqref{eq.D_part2}, and $K$ the universal constant in Lemma~\ref{le1}, then the classifier $\wh{k}$ as defined in \eqref{eq.k_def} will recover the correct label, that is,
\begin{align*}
    \wh{k}=k.
\end{align*}
\end{thm}
The proof of Theorem~\ref{thm.main_part1} is postponed to Section~\ref{proofs-image-align}. The result indicates that, under proper conditions, the classifier accurately identifies the label when the template functions $f_0$ and $f_1$ are separated by $\gtrsim 1/d$ in $L^2$-norm, consistently across all conceivable image deformations. This finding aligns with the theoretical performance outlined in Theorem~\ref{thm.main_part1_general} for the general classification approach. The advantage of implementing the image alignment approach is that it eliminates the need to discretize the set of inverse mappings, thereby substantially improving computational efficiency.

We further prove a corresponding lower bound, showing that a $1/d$-rate in the separation criterion is necessary. Without this condition, the same image could be represented by deformations of both template functions, making classification impossible.

\begin{thm}
\label{thm27}
For any $\tau,\tau',\xi,\xi'$ satisfying Assumption~\ref{ass1-prime}, there exist non-negative Lipschitz continuous functions $f_{0}$, $f_{1}$ with Lipschitz constants $C_{f_{0}}$ and $C_{f_{1}}=C_{f_{1}}(\xi,\xi')$ respectively, such that for any $d\geq32(\xi\vee\xi')$,
\begin{align*}
\frac{\big\|f_1- f_0\big\|_{L^2(\mathbb{R}^2)}}{\|f_0\|_{L^2(\mathbb{R}^2)}}\geq \frac{1}{28 d},
\end{align*}
and the data generating model \eqref{eq.X_f_model} can be written as
\begin{align*}
X_{j,\ell}=d^2\eta\int_{I_{j,\ell}} f_{1}(\xi u-\tau,\xi'v-\tau') \, dudv=d^2\eta\int_{I_{j,\ell}} f_{0}(\xi u-\tau,\xi'v-\tau') \, dudv.
\end{align*}
Consequently, the same pixel values are generated under both classes.
\end{thm}
The proof of Theorem~\ref{thm27} is deferred to Section~\ref{proofs-image-align}, which shows that the separation rate $1/d$ arises from the Lipschitz continuity of $f_0$ and $f_1.$ If, instead, we assume H\"older regularity with index $\beta \leq 1,$ we expect the lower bound to be of order $d^{-\beta},$ which we also conjecture to be the optimal separation rate in this case. 

Since the deformation model \eqref{eq.mod} is a specific case of the general model \eqref{eq.mod-general}, the lower bound derived in Theorem~\ref{thm27} also applies to the general deformation model \eqref{eq.mod-general}. This indicates that the rate $1/d$ is indeed necessary to distinguish between the two classes.

In the presence of background noise, finding the rectangular support of the object is hard, as non-zero pixel values in the image may belong to the background. As an alternative one could instead rely on $t$-level sets $\{ \mathbf{x}: g( \mathbf{x}) > t\}$. Define the $t$-rectangular support as the smallest rectangular containing the $t$-level set. For non-negative $g$, the previously introduced rectangular support corresponds to $t=0.$ To construct a classifier, we can first normalize the pixel values to eliminate the brightness factor $\eta$, then follow a similar strategy as in the zero-background case by determining the $t$-rectangular support for each image in the dataset. While increasing $t$ enhances robustness to background noise, it also reduces the $t$-rectangular support and causes larger constants in the separation condition between the two classes. 

If an image contains multiple non-overlapping objects, we suggest to first apply an image segmentation method (see, e.g., \cite{HARALICK1985100, 9356353}) to isolate each object. The image alignment classifier can then be applied to each segment separately.

The analysis of both classifiers generalizes to the multi-class case with $K$ classes, provided that the label of each class appears at least once in the training data, and the separation quantity between each pair of class templates, defined as
$D_{i,j} := D(f_i, f_j) \vee D(f_j, f_i)$, satisfies $D_{i,j} \gtrsim 1/d$ for all $i, j \in \{1, \ldots, K\}$. Under the conditions outlined above and Assumptions~\ref{ass-1} (adapted to the multi-class setting) and \ref{ass1}, the classifier will correctly recover the class labels.

The image alignment step in the construction of the classifier leads to a representation of the image that is, up to discretization effects, independent to rescaling and shifting of the object; see Figure \ref{fig:Z}. While the proposed image alignment transformation is natural and mathematically tractable for this specific deformation model, other transformations could be employed instead, such as Fourier transform, Radon transform, and scattering transform \cite{5995635,M12,MR2957703}.

\section{Classification with convolutional neural networks}
\label{sec.CNN}
Convolutional neural networks (CNNs) have achieved remarkable practical success, particularly in the context of image recognition \cite{LBH15,KSG17,Sch15,RZ17}. In this section, we analyze the performance of CNN-based classifiers within the framework of the general deformation model~\eqref{eq.mod-general}, introduced in Section~\ref{deformation-setting}. We begin by introducing the mathematical notation to formalize the structure of a CNN. Here we focus on a particular CNN structure and refer to \cite{Yamashita_2018, RZ17} for a broader introduction.

\subsection{Convolutional neural networks}\label{cnn-ope}
We analyse a CNN with a rectified linear unit (ReLU) activation function and a softmax output layer. Generally, a CNN consists of three fundamental components: Convolutional, pooling and fully connected layers. The input to a CNN is a $d\times d$ matrix representing the pixel values of an image. In the convolutional layer, so-called filters (that is, weight matrices of pre-defined size) slide across the image, performing convolutions at each spatial location. Finally, an element-wise nonlinear activation function $\sigma: \mathbb{R} \to \mathbb{R}$, in our case the ReLU function, is applied to the outcome of the convolutions, producing the output matrices known as feature maps. 

In this work, we consider CNNs with a single convolutional layer followed by one pooling layer. For mathematical simplicity, we introduce a compact notation tailored to our setting and refer to \cite{KKW20, KL2025} for a general mathematical definition. Recall that the input to the network is an image represented by a $d \times d$ matrix $\bX$. For a $d \times d$ matrix $ \mathbf{W}$, we define its \textit{quadratic support} $[ \mathbf{W}]$ as the smallest square sub-matrix of $ \mathbf{W}$ that contains all its non-zero entries. For instance, 
\begin{align*}
    [ \mathbf{W}] = \begin{pmatrix}
    1 & 1\\
    0 & 0
    \end{pmatrix}
\end{align*}
is the quadratic support of the matrix
\begin{align*}
     \mathbf{W} = \begin{pmatrix}
0 & 1 & 1 \\
0 & 0 & 0 \\
0 & 0 & 0 
\end{pmatrix}.
\end{align*}
In this context, $[ \mathbf{W}]$ represents the network filter. To describe the action of the filter on the image, denoted as $[ \mathbf{W}] \star \bX$, assume that $[ \mathbf{W}]$ is a filter of size $\ell \in \{1, \dots, d\}$.  We extend the matrix $\bX$ by padding it with zero matrices on all sides. Specifically, we define the enlarged matrix as 
\begin{align*}
    \bX':= \begin{bmatrix}
     0_{\ell \times \ell} & 0_{\ell \times d} & 0_{\ell \times \ell}\\
     0_{d \times \ell} & \bX & 0_{d \times \ell}\\
     0_{\ell \times \ell} & 0_{\ell \times d} & 0_{\ell \times \ell}
    \end{bmatrix},
\end{align*}
where $0_{j \times k}$ denotes a $j\times k$ zero matrix. The $(i,j)$-th patch is defined as the $\ell \times \ell$ submatrix 
$\bX'_{i,j}:=(X'_{i+a, j+b})_{a,b = 0, \dots, \ell-1}.$ We further define $([ \mathbf{W}] \star \bX)_{i,j}$ as the entry-wise sum of the Hadamard product of $[ \mathbf{W}]$ and  $\bX'_{i,j}$. Thus, the matrix $[ \mathbf{W}] \star \bX$ contains all entry-wise sums of the Hadamard product of $[ \mathbf{W}]$ with all patches. Finally the ReLU activation function $\sigma(x)=\max\{x,0\}$ is applied element-wise. A feature map can then be expressed as
\begin{align*}
    \sigma([ \mathbf{W}] \star \bX).
\end{align*}
This extension of the matrix $\bX$ to $\bX'$ is a form of zero padding, which ensures that the in-plane dimension of the input remains equal after convolution \cite{H19}. A pooling layer is typically applied to the feature map. While max-pooling extracts the maximum value from each patch of the feature map, average-pooling computes the average over each patch. In this work we consider CNNs with \textit{global} max-pooling in the sense that the max-pooling extracts from every feature map $\sigma([ \mathbf{W}] \star \bX)$ the largest absolute value. The feature map after global max-pooling is then given by 
\begin{align*}
     \mathbf{O}(\bX) = |\sigma([ \mathbf{W}] \star \bX)|_{\infty}.
\end{align*}
For $k$ filters described by the matrices $ \mathbf{W}_1,\ldots,  \mathbf{W}_k,$ we obtain the $k$ values
\begin{align}
\label{fmap}
\mathbf{O}_s(\bX) = |\sigma([ \mathbf{W}_s] \star \bX)|_{\infty}, \quad s=1,\ldots,k.
\end{align}

For $\alpha\in (0,1),$ define $\mathcal{K}(i) :=\{ \lfloor d^{1-\alpha} + 1 \rfloor (i - 1) + 1,\; \ldots,\; \big( \lfloor d^{1-\alpha} + 1 \rfloor\, i \big) \wedge d \}.$ We say that a $d\times d$ filter matrix $\mathbf{W}=(W_{j,\ell})_{j,\ell=1,\ldots,d}$ has an $(\alpha,d)$-block structure if $W_{j,\ell}=W_{j',\ell'}$ whenever $j,j'\in \mathcal{K}(i)$ and $\ell,\ell'\in \mathcal{K}(i').$ By convention, any $d \times d$ matrix $\mathbf{W}$ is said to have a $(1, d)$-block structure. An illustration of the $(\alpha, d)$-block structure is provided in Figure~\ref{superpixel} in Section~\ref{sec.proofs_CNNs}. For $0< \alpha \leq 1,$ we denote by 
\begin{equation}\label{cnn-layer-def}
\mathcal{F}^C(\alpha, k)
\end{equation}
the class of all CNN layers computing $k$ outputs of the form \eqref{fmap}, where each filter matrix $\mathbf{W}_s$ has an $(\alpha,d)$-block structure and all parameters take values in the interval $[-1,1]$. The output of the last convolutional layer is flattened, this means, it is transformed into a vector before several fully connected layers with ReLU activation function are applied.

For any vector $ \mathbf{v} = (v_1, \dots, v_r)^{\top}$, $ \mathbf{y}=(y_1, \dots, y_r)^{\top} \in \mathbb{R}^r$, we define $\sigma_{ \mathbf{v}} \mathbf{y} = (\sigma(y_1-v_1), \dots, \sigma(y_r-v_r))^{\top}$. In the context of binary classification, the last layer of the network should extract a two-dimensional probability vector. To achieve this, the softmax function
\begin{align}\label{eq.def_beta_SM}
    \Phi(x_1,x_2) = \left(\frac{e^{ x_1}}{e^{ x_1}+e^{ x_2}}, \frac{e^{ x_2}}{e^{ x_1}+e^{ x_2}}\right)
\end{align}
is typically applied. A feedforward neural network with $L$ fully connected hidden layers and width vector $ \mathbf{m}=(m_0, \dots, m_{L+1}) \in \mathbb{N}^{L+2}$, where $m_i$ denotes the number of hidden neurons in the $i$-th hidden layer, can then be described by a function $f: \mathbb{R}^{m_0} \to \mathbb{R}^{m_{L+1}}$ with
\begin{align*}
     \mathbf{x} \mapsto f( \mathbf{x})=\psi\mkern1mu \sigma_{ \mathbf{v}_{L+1}}  \mathbf{W}_L\sigma_{ \mathbf{v}_L} \mathbf{W}_{L-1}\sigma_{ \mathbf{v}_{L-1}}\cdots  \mathbf{W}_1\sigma_{ \mathbf{v}_1} \mathbf{W}_0 \mathbf{x},
\end{align*}
where $ \mathbf{W}_j$ is a $m_j \times m_{j+1}$ weight matrix, $ \mathbf{v}_j$ is the bias vector in layer $j$ and $\psi$ is either the identity function $\psi=id$ or the softmax function $\psi=\Phi.$ We consider the class of fully connected neural networks in which all entries of the weight matrices and bias vectors are bounded in absolute value by 1, and denote this class by
\begin{equation}\label{fully-layers-def}
\mathcal{F}_{\psi}(L,  \mathbf{m}).
\end{equation}

We will construct CNN classifiers based on a CNN architecture of the form
\begin{align}
\label{CNN}
\begin{split}
\mathcal{G}(\alpha,m) := \Big\{f \circ g: 
 f \in \mathcal{F}_{\Phi}\big(1+2\lceil \log_2 m\rceil, (2m,4m,\ldots,4m,2)\big), g \in \mathcal{F}^C(\alpha,2m)\Big\},
\end{split}
\end{align}
with $m$ a positive integer and $0<\alpha\leq1$. Given that we only consider one convolutional and one pooling layer, the number of feature maps equals the input dimension of the fully connected subnetwork.

\begin{figure}
    \centering
\includegraphics[width=0.8\textwidth]{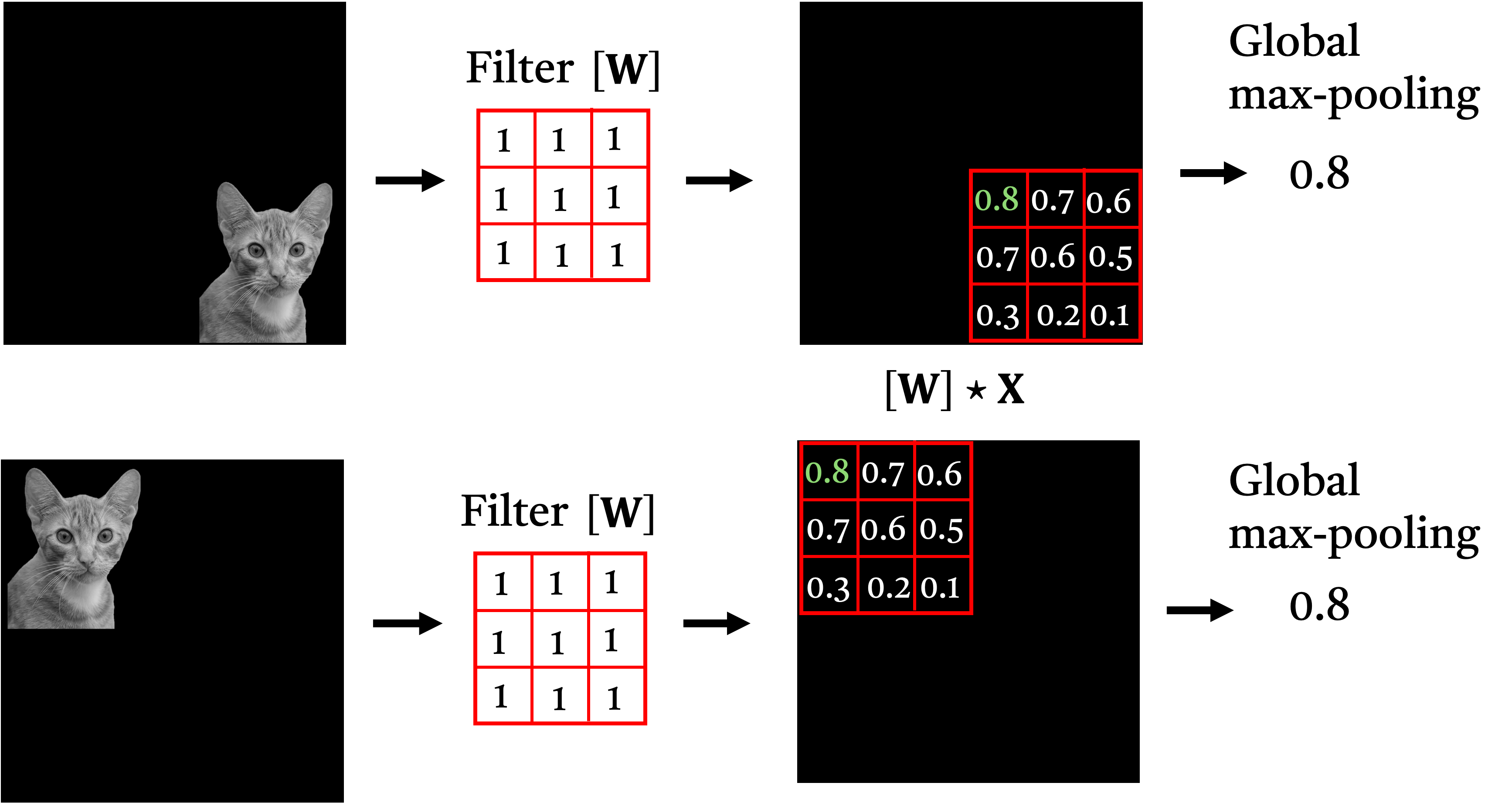}
    \caption{Shift invariance of CNNs}
\label{fig:translation}
\end{figure}

As the filters are applied to all patches of the image, CNNs are translation-invariant, meaning that up to boundary and discretization effects, the CNN classifier does not depend on the values of the shifting parameters $\tau$ and $\tau^{\prime}$ in the deformation model. For instance, if a cat in an image is moved from the upper left corner to the lower right corner, the convolutional filter will produce, up to discretization effects, the same feature values at potentially different locations within the feature map; see Figure \ref{fig:translation} for an illustration. A shift of the image pixels causes therefore a permutation of the values in the feature map. Since the global max-pooling layer is invariant to permutations, the CNN output is thus invariant under translations.

More challenging for CNNs are varying object sizes and rotation angles. \cite{NP17, PRF10, G06} argue that scale-invariance is undesirable in image classification as classifiers can benefit from scale information of the object and that architectures can assign different filters to capture different scales. Similarly, \cite{7899932} employs separate filters for different rotation angles, thereby achieving rotation invariance for texture classification. Data augmentation enhances the learning of CNNs in the presence of rotation and scale deformations at the expense of additional computational cost. Before training, data augmentation applies simple deformations such as rotations and different scaling to the training images and learns a CNN on the augmented dataset consisting of the original and the transformed training samples (see, e.g., \cite{1227801}). Group equivariant convolutional networks \cite{cohen2016group} extend CNNs to handle invariances induced by arbitrary groups. An essential part of the derived theory in the next section shows that for rich classes of deformations, CNNs are expressive enough to separate deformed images. 

\subsection{Misclassification bounds for CNN-based classifiers}
We suppose that the training data consists of $n$ i.i.d.\ data points, generated as follows: For $\pi \in [0,1]$ and each $i,$ we draw a label $k_i\in \{0,1\}$ from the Bernoulli distribution with success probability $\pi$. Let $Q_{\mathcal{A}}$ be a distribution over the deformation class $\mathcal{A}$, and let $Q_{\eta}$ be a distribution on $(0,\infty)$ for the random brightness factor. Here, we assume that $\eta$ and $A$ are independent. The $i$-th sample is  $(\bX_i, k_i),$ where $\bX_i$ is an independent draw from the general model \eqref{eq.mod-general} with template function $f_{k_i}$, deformation $A_i\in\mathcal{A}$ generated from $Q_{\mathcal{A}}$, and brightness factor $\eta_i$ generated from $Q_{\eta}$. The full dataset is denoted by
\begin{align}
\mathcal{D}_n=\big((\bX_1, k_1), \dots, (\bX_n, k_n)\big).\label{eq.class_data_mod}
\end{align}
In expectation, the dataset consists of $n(1-\pi)$ samples from class $0$ and $n\pi$ samples from class $1.$

The parameters of a CNN are then fitted to the normalized images
\begin{align}
\ol{ \mathbf{X}}_i=(\ol X_{j,\ell}^{(i)})_{j, \ell =1, \dots, d}, \quad \ \text{with} \ \  \ol X_{j,\ell}^{(i)}:= \frac{X_{j,\ell}^{(i)}}{\sqrt{\sum_{j,\ell=1}^d ( X_{j,\ell}^{(i)})^2}}. \label{eq.97bc}
\end{align}
This normalization can be viewed as pre-processing step. It ensures that the images are invariant to variations of the brightness $\eta$ and all pixel values lie between $0$ and $1$. 

We fit a CNN by minimizing the empirical error under the 0-1 loss and the cross-entropy loss. As the most natural choice, empirical risk minimization based on the 0-1 loss has been extensively studied in classification theory within the standard statistical learning framework \cite{Arlot_2011,B05,Vapnik1998}, but optimization with respect to the 0-1 loss is considered to be computationally intractable due to the loss function's non-convexity and discontinuity. The cross-entropy loss serves as a tractable surrogate for the 0-1 loss and is widely adopted in practice due to its favorable optimization properties \cite{zhang2004,B06}. 

For the 0-1 loss, the empirical risk minimizer over the CNN class~\eqref{CNN} is defined as
\begin{align}
\wh{ \mathbf{p}} =\big(\wh p_1,\wh p_2\big) \,  \in \  \argmin_{ \mathbf{q}=(q_{1},q_{2}) \in \mathcal{G}(\alpha,m)} \, \frac{1}{n}\sum_{i=1}^{n}\1\left(\1\Big(q_2(\ol{\bX}_i)>\frac 12 \Big)\not=k_i\right),\label{eq.p_LS}
\end{align}
with $\1(\wh p_2(\ol{\bX}_i)>1/2)$ the predicted label of the $i$-th sample based on the network $\wh{\mathbf{p}}=(\wh p_{1},\wh p_{2}) \in \mathcal{G}(\alpha,m)$. The learned network $\wh{ \mathbf{p}}$ outputs estimates for the two conditional class probabilities $p_1(\bx)= \mathbf{P}(k=0|\bX=\bx)$ and $p_2(\bx)=\mathbf{P}(k=1|\bX=\bx).$ These probabilities sum to one and optimizing over $q_2$ suffices.

For a new test image $\ol \bX$ that has been normalized according to \eqref{eq.97bc}, the classifier $\wh k(\bX):=\1(\wh p_2(\ol{\bX})>1/2)$ assigns the label $1$ if the estimated probability belonging to class $1$ exceeds $1/2$ and assigns class label $0$ otherwise.

For the theory, we consider the general deformation model \eqref{eq.mod-general} and impose the following assumption.
\begin{customthm}{3}[Covering of deformation class]\label{ass-cover}
For any $\alpha\in(0,1]$, the deformation class $\mathcal{A}$ contains a finite subset $\mathcal{A}_{d_{\alpha}}$ such that for any $A \in \mathcal{A}$, there exists an $A' \in \mathcal{A}_{d_{\alpha}}$ and indices $j,\ell\in\{1,\ldots,d\}$ such that $A'(\cdot + j/d, \cdot + \ell/d)$ satisfies Assumption~\ref{ass1}-(i) and
\begin{equation*}
\left\|A'\left(\cdot+\frac{j}{d},\cdot+\frac{\ell}{d}\right)-A\right\|_{\infty}\leq d^{-\alpha}.
\end{equation*}
\end{customthm}
Similarly to the derivation of $\mathcal{A}_d^{-1}$ in Section~\ref{sec.image_process}, the subset $\mA_{d_{\alpha}}$ can be obtained through suitable discretization of the deformation class $\mathcal{A}$. The cardinality of the discretized class $\mathcal{A}_{d_{\alpha}}$ is typically of the order $d^{\alpha r}$ with $r$ the number of free parameters that are not related to the shifts. For instance, two out of the four parameters in the deformation model \eqref{eq.mod} control the shift such that $|\mathcal{A}_{d_\alpha}|\asymp d^{2\alpha}.$ Adding one parameter controlling the rotation, as in \eqref{affinetran} and \eqref{eq.87etgb} yields $|\mathcal{A}_{d_{\alpha}}|\asymp d^{3\alpha}$.

For the CNN-based method under the deformation model \eqref{eq.mod-general}, we consider the separation quantity $D=D(f_0,f_1) \vee D(f_1,f_0)$ with
\begin{align}
D(f,g):=\frac{\inf_{a,s,s'\in \mathbb{R},A,A'\in\mathcal{A}} \|a f\circ A(\cdot+s,\cdot+s') -g\circ A'\|_{L^2(\mathbb{R}^2)}}{\|g\|_{L^2(\mathbb{R}^2)}}.
\label{eq.D_part-cnn}
\end{align}
It measures the normalized minimal $L^2$-distance between any deformed versions of $f$ and $g$ (up to spatial shifts) under transformations from the class $\mathcal{A}$. 

The next result states the misclassification bound for the CNN-based classifier with 0-1 loss.

\begin{thm}
\label{thm4}
Consider the general deformation model \eqref{eq.mod-general} and suppose Assumptions \ref{ass-1}, \ref{ass1}-(i), \ref{ass-cover} hold. Let $\wh{\mathbf{p}} =\big(\wh p_1,\wh p_2\big)$ be the estimator in \eqref{eq.p_LS}, based on the CNN class $\mathcal{G}(\alpha,|\mathcal{A}_{d_{\alpha}}|)$ defined in \eqref{CNN}, with $|\mathcal{A}_{d_\alpha}|\geq 2$. Suppose a new data point $(\bX,k)$ is independently drawn from the same distribution as the data in \eqref{eq.class_data_mod}. If the separation quantity in \eqref{eq.D_part-cnn} satisfies $D\geq\sqrt{\kappa/d^{\alpha}}$, where $\alpha\in(0,1]$ and $\kappa$ is a constant depending only on $C_{L}$ and $C_{\mathcal{A}}$, then for the classifier $\wh k(\bX)=\1(\wh p_2(\ol\bX)>1/2)$, there exists a universal constant $C>0$ such that
\begin{align}
\label{thmeq}
&\PROB\big(\, \wh{k}(\bX)\not=k\big)\leq C\frac{|\mathcal{A}_{d_\alpha}|(d^{2\alpha}+|\mathcal{A}_{d_\alpha}|)}{n} \log^3(d|\mathcal{A}_{d_\alpha}|)\log n.
\end{align}
\end{thm}

The proof of Theorem~\ref{thm4} is postponed to Section~\ref{sec.proofs_CNNs}. Here, $\PROB$ denotes the distribution over all randomness in the data and the new sample $\bX.$ For fixed $\alpha$, since the cardinality $|\mathcal{A}_{d_{\alpha}}|$ depends on $d$, the upper bound \eqref{thmeq} grows in the image resolution $d.$ At first glance, this might seem counterintuitive as a high image resolution typically provides more information about the object and should lead to improved misclassification bounds. However, as $d$ increases, the template functions can become closer in their separation distance \eqref{eq.D_part-cnn} making the two objects more similar and thus harder to separate. 

The CNN architecture employs $2|\mathcal{A}_{d_{\alpha}}|$ filters and contains $\lesssim( d^{2\alpha}|\mathcal{A}_{d_{\alpha}}|+|\mathcal{A}_{d_{\alpha}}|^2)\log(|\mathcal{A}_{d_{\alpha}}|)$ parameters (see Lemma~\ref{le6}). Up to logarithmic factors, \eqref{thmeq} means that we obtain a consistent classifier if the sample size is of larger order than the number of network parameters. More generally, Theorem~\ref{thm4} implies that for sufficiently large sample sizes, the misclassification error of the proposed CNN-based classifier can become arbitrary small. Compared to the image classifiers discussed in Section~\ref{sec.image_process}, the CNN-based classifier only requires knowledge of the deformation class for the choice of the architecture. Additionally, for Theorem~\ref{thm4} and all subsequent results in this section, Assumption~\ref{ass1}-(i) on the deformation set $\mathcal{A}$ (necessary in Section~\ref{sec.image_process}) can be relaxed to Lipschitz continuity. This indicates that the CNN-based approach does not require invertible deformations. To guarantee that a sufficient amount of ``information'' is still preserved under the deformations, one instead needs to assume the existence of a universal constant $c_{\mathcal{A}} > 0$ such that for any $A \in \mathcal{A}$, $\|f_i\|_1 \leq c_{\mathcal{A}} \|f_i \circ A\|_1$ with $i\in\{0,1\}$.

That the CNN misclassification error can become arbitrarily small is in line
with the nearly perfect classification results of deep learning for a number of image classification tasks in practice.
Interestingly, most of the previous statistical analysis for neural networks considers settings with asymptotically non-vanishing
prediction error. Those are statistical models where every new image contains randomness that is independent
of the training data and can therefore not be predicted by the classifier. To illustrate this, consider the nonparametric regression model $Y_i=f(\bX_i)+\sigma \varepsilon_i,$ $i=1,\ldots,n$ with fixed noise variance $\sigma^2$. The squared prediction error of the predictor $\wh Y=\wh f_n(\bX)$ for $Y$ is $ \mathbf{E}(\wh Y-Y)^2=\sigma^2+ \mathbf{E}\big[(\wh f_n(\bX)-f(\bX))^2\big].$ This implies that even if we can perfectly learn the function $f$ from the data, the prediction error is still at least $\sigma^2.$ Taking a highly suboptimal but consistent estimator $\wh f_n$ for $f,$ yields a prediction error $\sigma^2+o(1)$, thereby achieving the lower bound $\sigma^2$ up to a vanishing term. For instance, the one-nearest neighbor classifier does not employ any smoothing and results in suboptimal rates for conditional class probabilities but is optimal for the misclassification error up to a factor of 2, \cite{1053964}. This shows that optimal estimation of $f$ only affects the second order term of the prediction error. As classification and regression are closely related, the same phenomenon also occurs in classification, whenever the conditional class probabilities lie strictly between $0$ and $1$. The only possibility to achieve small misclassification error requires that the conditional class probabilities are consistently close to either zero or one. This means that the covariates $\bX$ contain (nearly) all information about the label $k,$ \cite{MR4406243}. The main source of randomness lies then in the sampling of the covariates $\bX.$ An example are the random deformations considered in this work that only affect the covariates but not the labels. This highlights the main difference from existing standard classification settings.

Theorem~\ref{thm4} differs from the generalization bounds in \cite{cnngen2020} and related works such as \cite{peter2017, neyshabur2018}, which focus on generalization performance of non-learned functions (this means the functions are not allowed to depend on the data)-a distinct perspective that can be traced back to \cite{peter1998}. For example, as can be seen in Theorem~2.1 of \cite{cnngen2020}, their bound involves the term 
$\sum_{i=1}^n \ell(f(\mathbf{X}_i), k_i)$, with $\ell$ the loss function used. 
Since $f$ does not depend on the data, this term cannot be further specified, and consequently, 
no statistical convergence rates for estimators can be deduced. In contrast, the derived misclassification error $\mathbf{P}(\widehat{k}(\mathbf{X}) \neq k)$ accounts for all sources of randomness, including both the training data and the test sample. Furthermore, the derivation of $\widehat{k}$ links the generalization analysis to approximation theory, enabling us to explicitly construct a suitable CNN model and thus avoid any implicit terms in the bound. 

To prove Theorem~\ref{thm4}, one can decompose the misclassification error into an approximation error and a stochastic error term (see Lemma \ref{le5}). The stochastic error can be bounded via statistical learning tools such as the Vapnik-Chervonenkis (VC) dimension of the CNN class $\mathcal{G}(\alpha,|\mathcal{A}_{d_{\alpha}}|)$; see Lemma~\ref{le6}. The approximation error vanishes as CNNs from the class $\mathcal{G}(\alpha,|\mathcal{A}_{d_{\alpha}}|)$ with suitably chosen parameters can achieve perfect classification. This result is stated in the following theorem and is proved in Section~\ref{sec.proofs_CNNs}. 

\begin{thm}
\label{th2-gen}
If Assumptions \ref{ass-1}, \ref{ass1}-(i), \ref{ass-cover} hold and the separation quantity $D$ in \eqref{eq.D_part-cnn} satisfies $D\geq\sqrt{\kappa/d^\alpha}$, with $\alpha\in(0,1]$ and  $\kappa$ a constant depending only on the constants $C_{L}$ and $C_{\mathcal{A}}$, then, for any $(\bX,k)$ generated from the same distribution as the data in \eqref{eq.class_data_mod}, there exists a network $\mathbf{p} =\big(p_1,p_2\big) \in \mathcal{G}(\alpha,|\mathcal{A}_{d_{\alpha}}|)$, such that the corresponding classifier $\widetilde{k}(\bX)=\1(p_2(\ol\bX)>1/2)$ satisfies
\begin{align*}
\widetilde{k}(\bX)=k, \ \text{almost surely.}
\end{align*}
\end{thm}
To ensure the existence of the interpolating classifier in the previous result, the best achievable order of the separation quantity $D$ with respect to the resolution level $d$ is 
\begin{align*}
D\gtrsim \frac1{\sqrt{d}},
\end{align*}
which is more restrictive if compared to the lower bound $D \gtrsim 1/d$ 
imposed on the image classification methods discussed in Section~\ref{sec.image_process}. This discrepancy arises from the construction of the proposed CNN architecture. To handle different deformations, the CNN construction uses $2|\mathcal{A}_{d_{\alpha}}|$ separate filters to test whether the image was approximately generated by applying one of the deformations in the discretized set $\mathcal{A}_{d_{\alpha}}$ to either of the two possible template functions. More specifically, consider a given input image that has been generated by deforming the template function of class $k\in \{0,1\}$ by $A.$ In Proposition~\ref{prop.cnns}, we show that the convolutional filter corresponding to the correct class $k$ and the deformation in $\mathcal{A}_{d_{\alpha}}$ that is closest to the true deformation $A$ produces the highest activation, that is, the highest output value after convolution and global max-pooling; see Figure \ref{fig:filters} for an illustration. This requires $D\gtrsim 1/{\sqrt{d}}$ when $\alpha=1.$

The $1/d$ separation rate can be obtained, under additional conditions, if we instead take $2|\mathcal{A}_{d_1}|^2$ many convolutional filters, resulting in a CNN architecture with $\lesssim(|\mathcal{A}_{d_1}|^2d^2+|\mathcal{A}_{d_1}|^4)\log(|\mathcal{A}_{d_1}|)$ network parameters. To achieve this, the high level idea is to test for all possible differences of the two template functions $f_0, f_1$, deformed by $A_0, A_1\in \mathcal{A}_{d_1}.$ The key step in Theorem~\ref{th2-gen} is to show that for input image generated by $f \circ A$, we can find a filter $\phi$ with $\|\phi\|_2=1$ such that the output of the feature map after applying the global max-pooling layer is
\begin{align*}
    \max_{s,t} \, \frac{\int_{[0,1]^2} \phi(u-s,v-t) \, f \circ A \mkern1mu (u,v) \, du dv}{\|f \circ A\|_2} +O\Big(\frac 1d\Big).
\end{align*} 
Choosing now $\phi_{A_0,A_1}:=(r_1 f_1 \circ A_1-r_0 f_0\circ A_0)/\|r_1 f_1 \circ A_1-r_0 f_0\circ A_0\|_2$ with $r_k:=\|f_k\circ A_k\|_2^{-1},$ and ignoring the maximum over $s,t$  by just considering at the moment $s=t=0,$ we find
\begin{align*}
&\frac{\int_{[0,1]^2} \phi_{A_0,A_1}(u,v) \, f_k \circ A_k \mkern1mu (u,v) \, du dv}{\|f_k\circ A_k\|_2} +O\Big(\frac 1d\Big)\\
=&\frac{\int_{[0,1]^2}
    \cro{r_1 f_1 \circ A_1(u,v)-r_0 f_0\circ A_0(u,v)} \, r_k f_k\circ A_k \mkern1mu (u,v) \, du dv}{\|r_1 f_1 \circ A_1-r_0 f_0\circ A_0\|_2} +O\Big(\frac 1d\Big)\\
    =&\frac{(-1)^{k+1}}{2}\big\|r_1 f_1 \circ A_1-r_0 f_0\circ A_0\big\|_2+O\Big(\frac 1d\Big).
\end{align*}
This indicates that one can discriminate between the two classes under the separation condition $\|r_1 f_1 \circ A_1-r_0 f_0\circ A_0\|_2 \geq c/d,$ where $c$ has to be chosen large enough, such that the $O(1/d)$ term does not cause overlap between the signals from the two classes in the previous display. However, since the rate $1/\sqrt{d}$ already covers most practical application scenarios, and the goal is to analyze the performance of commonly used CNN architectures with as few filters as possible, we will not comment further on this direction.

\begin{figure}[t]
    \centering
\includegraphics[width=0.7\textwidth]{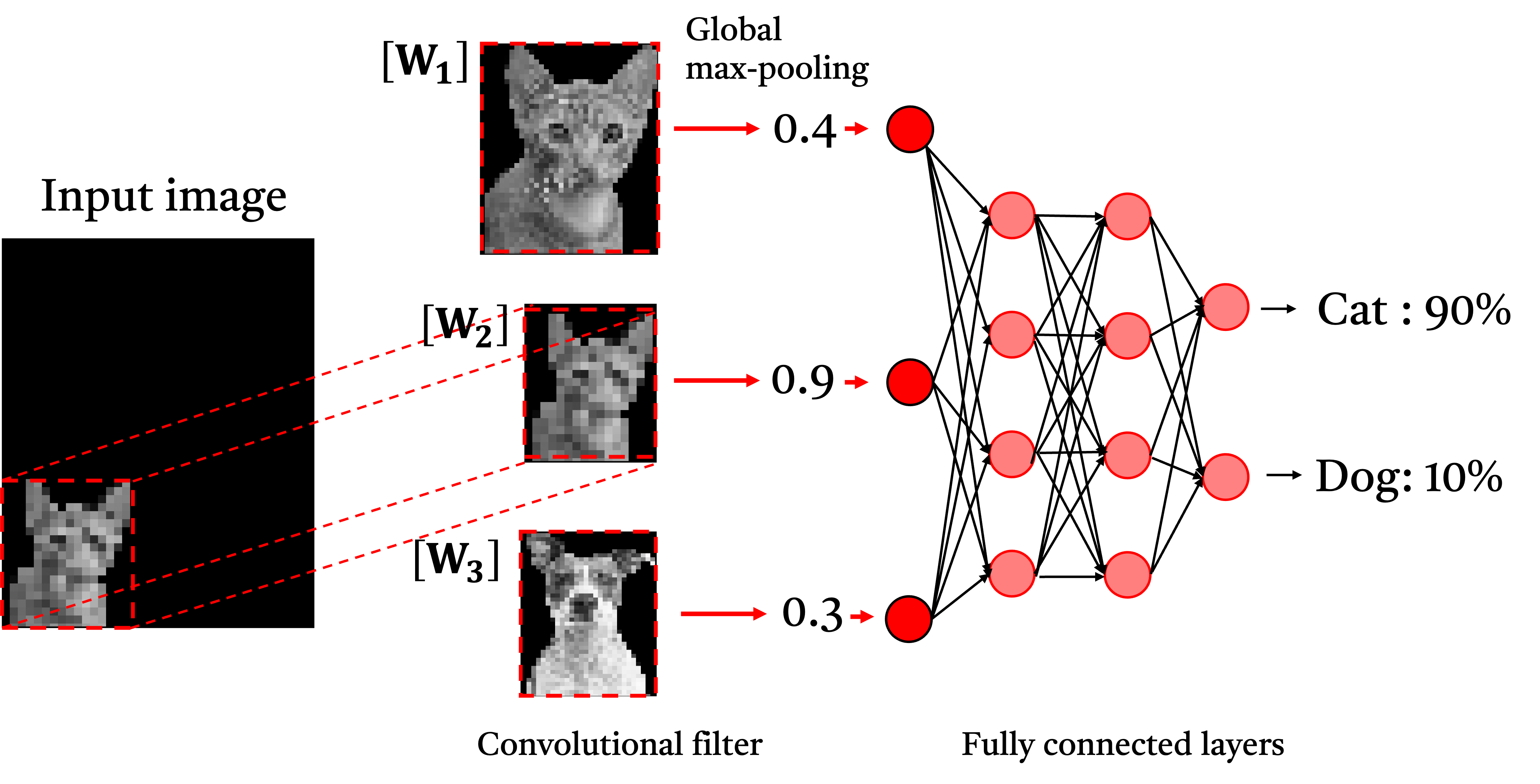}
\caption{Effect of different filters on the same image. The global max-pooling layer will generate the largest values for filters that are most similar to the object.}
\label{fig:filters}
\end{figure}

A consequence of the approximation result is that, under the conditions of Theorem~\ref{th2-gen}, the label can be retrieved from the image.

\begin{lemma}
\label{lem.perfect_recovery-gen}
Let Assumptions \ref{ass-1}, \ref{ass1}-(i), \ref{ass-cover} hold. If the separation quantity $D$ in \eqref{eq.D_part-cnn} satisfies $D\geq\sqrt{\kappa/d^{\alpha}}$, where $\kappa$ is a constant depending only on the constants $C_L$ and $C_{\mathcal{A}}$, then, for any $(\bX,k)$ generated from the same distribution as the data in \eqref{eq.class_data_mod}, its label $k$ can be written as a deterministic function evaluated at $\bX$ and we have
\begin{align*}
    p(\bX)=k(\bX),
\end{align*}
where $p(\bx)=\mathbf{P}(k=1|\bX=\bx)$.
\end{lemma}
A consequence is that under the imposed conditions 
\begin{align*}
    \min_{q:[0,1]^2\to \{0,1\}}  \mathbf{P}(q(\bX)\neq k(\bX))=0.
\end{align*}

Next, we consider CNN-based classifiers trained using the commonly employed cross-entropy loss. For a pair $(\mathbf{X},k)$ with image $\mathbf{X}$ and its label $k \in \{0,1\}$, drawn from the data distribution, and $\overline{\mathbf{X}}$ the pre-processed image, the objective is to minimize the population risk
$$-\,\mathbf{E}\left[k\log\big(f(\overline{\mathbf{X}})\big) + (1-k)\log\big(1 - f(\overline{\mathbf{X}})\big)\right]$$
over a suitable class of CNNs. Since the true distribution of $(\mathbf{X}, k)$ is unknown, one instead minimizes the empirical version of this risk.

To avoid degeneracy in the cross-entropy loss, we consider the CNN class $\mathcal{G}(\alpha,m)$ with a slight modification that prevents the outputs to be too close to the boundaries $0$ and $1.$ Specifically, we define
\begin{equation}\label{g-beta-B}
\widetilde{\mathcal{G}}(\alpha,m) := \big\{(1 - \tilde{g}, \tilde{g}) :\; \tilde{g} = \big(g \vee \rho(1)\big) \wedge \rho(-1),\; (1 - g, g) \in \mathcal{G}(\alpha,m) \big\},
\end{equation}
where $\rho(z) := 1/(1 + e^z)$. The specific values $\rho(1), \rho(-1)$ are convenient but not essential. The empirical risk minimizer over all CNNs of the form \eqref{g-beta-B} is 
\begin{equation}
\wh{ \mathbf{p}}^{\text{CE}} =\big(\wh p_1^{\text{CE}},\wh p_2^{\text{CE}}\big)\in\argmin_{ \mathbf{q}=(q_{1},q_{2}) \in \widetilde{\mathcal{G}}(\alpha,m)}\hspace{-2pt}-\frac{1}{n}\sum_{i=1}^{n}\cro{k_{i}\log\big(q_2(\ol{\bX}_i)\big)+(1-k_i)\log\big(1-q_2(\ol{\bX}_i)\big)},
    \label{eq.p_CE}  
\end{equation}
where $\ol{\bX}_i$ denotes the normalized image as defined in \eqref{eq.97bc}. The misclassification bound for this CNN based classifier 
$$\wh k^{\text{CE}}(\bX)=\1(\wh p_2^{\text{CE}}(\ol\bX)>1/2)$$
is given below.

\begin{thm}\label{CE-bound}
Consider the general deformation model \eqref{eq.mod-general} and suppose Assumptions \ref{ass-1}, \ref{ass1}-(i), \ref{ass-cover} hold. Let $0<\alpha \leq 1$ and $\wh{ \mathbf{p}}^{\text{CE}}$ be the estimator in \eqref{eq.p_CE}, based on the CNN class $\widetilde{\mathcal{G}}(\alpha,|\mathcal{A}_{d_{\alpha}}|)$ defined in \eqref{g-beta-B}. Suppose a new data point $(\bX,k)$ is independently drawn from the same distribution as the data in \eqref{eq.class_data_mod}. If $|\mathcal{A}_{d_{\alpha}}|\geq (d^{\alpha}\vee2)$ and the separation quantity in \eqref{eq.D_part-cnn} satisfies $D\geq\sqrt{\kappa/ d^{\alpha}}$ for a constant $\kappa$ only depending on $(C_{L},C_{\mathcal{A}})$, then, there exists a universal constant $C>0$ such that for any $\gamma>0$ and any sufficiently large $n$, 
\begin{align}
\label{thmeq-CE}
    &\PROB\big(\, \wh{k}^{\text{CE}}(\bX)\not=k\big)\leq  
C\frac{|\mathcal{A}_{d_{\alpha}}|(d^{2\alpha}+|\mathcal{A}_{d_{\alpha}}|)}{n} \log(|\mathcal{A}_{d_{\alpha}}|)\log^{1+\gamma} n,
\end{align}
where $\PROB$ denotes the distribution over all randomness in the data and the new sample $\bX.$
\end{thm}
The proof of Theorem~\ref{CE-bound} is postponed to Section~\ref{sec.proofs_CNNs}. Up to logarithmic terms, CNN-based classifiers trained with the cross-entropy loss exhibit similar generalization performance to those trained with the 0-1 loss, and the misclassification error of both classifiers converges to zero at the rate $V/n$, where $V$ denotes the VC dimension of the underlying CNN architecture. The rate $V/n$ is minimax optimal, see Theorem~4 in \cite{risk2006}.

The CNN-based approach discussed in this section can be readily extended to the multi-class classification setting. Generalizing from $2$ to $K> 2$ classes, the number of filters and the width of all layers in the fully-connected network have to be multiplied by $K/2.$ The softmax function is then given by $\Phi(x_1,\ldots,x_{K})=(e^{x_1}/\sum_{i=1}^K e^{x_i}, \ldots, e^{x_K}/\sum_{i=1}^K e^{x_i}).$ 
Regarding the approximation theory, if Assumptions \ref{ass-1}, \ref{ass1}-(i), \ref{ass-cover} hold and the separation quantity satisfies the pairwise lower bound $D_{i,j} \gtrsim d^{-\alpha/2}$, for all $i,j=1,\ldots,K,$ with $i\neq j,$ then, there exists a network $\mathbf{p} = (p_1, \ldots, p_K)$ with the architecture described above such that the corresponding classifier $\widetilde{k}(\mathbf{X}) = \arg\max_{\ell\in\{1,\ldots,K\}}p_{\ell}(\ol\bX)$ satisfies $\widetilde{k}(\mathbf{X}) = k$, almost surely. Following the proof of Lemma~\ref{entropy-sup}, the covering number of the model class can be shown to be $K$ times that of the binary case (up to a logarithmic factor). With both the approximation and covering number results, one can apply existing error decomposition results, such as Theorem~3.5 (for the cross-entropy loss) in \cite{MR4406243} or Lemma~A.9 (for the hinge loss) in \cite{KOK19}, to bound the final misclassification error in the multi-class setting. Compared to the binary case, the resulting misclassification error bound scales with the number of classes $K.$ 

On the optimization side, current deep learning theory still cannot incorporate the algorithm (with the common initialization schemes and without excessive number of restarting) into the learning process for general nonparametric estimation problems. The NTK regime, studied, e.g., in \cite{jacot2018ntk}, allows to study gradient descent, but in an infinite-width limit with NTK scaling where the training is essentially \textit{lazy}, i.e., the network's feature remain fixed and the method converges to kernel regression.  Another approach is the so-called mean-field regime \cite{mei2018meanfield} with a different infinite-width scaling in which parameters move significantly, features evolve during training, and the dynamics are captured by deterministic PDEs over the parameter distribution. Beyond these asymptotic regimes, several works analyze how gradient-based methods can learn concrete function classes in finite-width settings, including multivariate polynomials \cite{damian2022neural} or the Barron class \cite{braun2024convergence}. However, these results use algorithmic adaptations such as only training separate layers or choosing initialization such that the weights change only slightly during training. In turn, most of the statistical results exclude the training routine assuming ideal optimization and analyze the empirical risk minimizer (ERM) instead. The discrepancy might be less severe given the claims in the literature that gradient based methods find local minima with training loss close to the global minimum. For large networks, it can be rigorously shown that gradient descent can reach zero training loss and thus gradient descent eventually converges to an ERM \cite{liangli2018,Du2018,NEURIPS2019_62dad6e2,aZhu2019b,DBLP:conf/icml/DuLL0Z19,zou2020}.

\section{Numerical results}
\label{sec.num_sim}
We empirically compare the three proposed methods: the inverse mapping classifier \eqref{gen-label}, the image alignment classifier \eqref{eq.k_def}, and CNN-based classifiers with three different architectures.
\subsection{Image inverse mapping and alignment}\label{ima-pro-com}
Theorems~\ref{thm.main_part1_general} and~\ref{thm.main_part1} show that both the inverse mapping classifier and the image alignment classifier can recover the correct label, provided each class appears at least once in the training data. We computed the empirical performance for both methods using a balanced design with $n =\{2, 4, 8, 16, 32\}$ labeled samples, meaning that each of the two classes contributes $n/2$ images.

We consider the deformation model \eqref{eq.mod} and generate the template functions from the FashionMNIST and CIFAR-100 datasets. FashionMNIST contains $28 \times 28$ grayscale images from different fashion categories, while CIFAR-100 includes $32 \times 32$ color images across 100 classes. In our setup, this corresponds to $d = 28$ and $d = 32$, respectively. In the experiments using FashionMNIST, we select as template functions one image with label `pullover' and one image with label `coat'. For CIFAR-100, the template functions are generated from the classes `fox' and `flatfish' and preprocessed to grayscale images with black backgrounds. Training samples are then generated by applying random shifts ($\tau, \tau'$), scalings ($\xi, \xi'$), and brightness changes ($\eta$) using the Keras ImageDataGenerator. For FashionMNIST, 
examples of generated deformed images are displayed in Figure~\ref{ima-pro-example}. 
\begin{figure}
\centering
\includegraphics[width=0.9\textwidth,height=5.5cm]{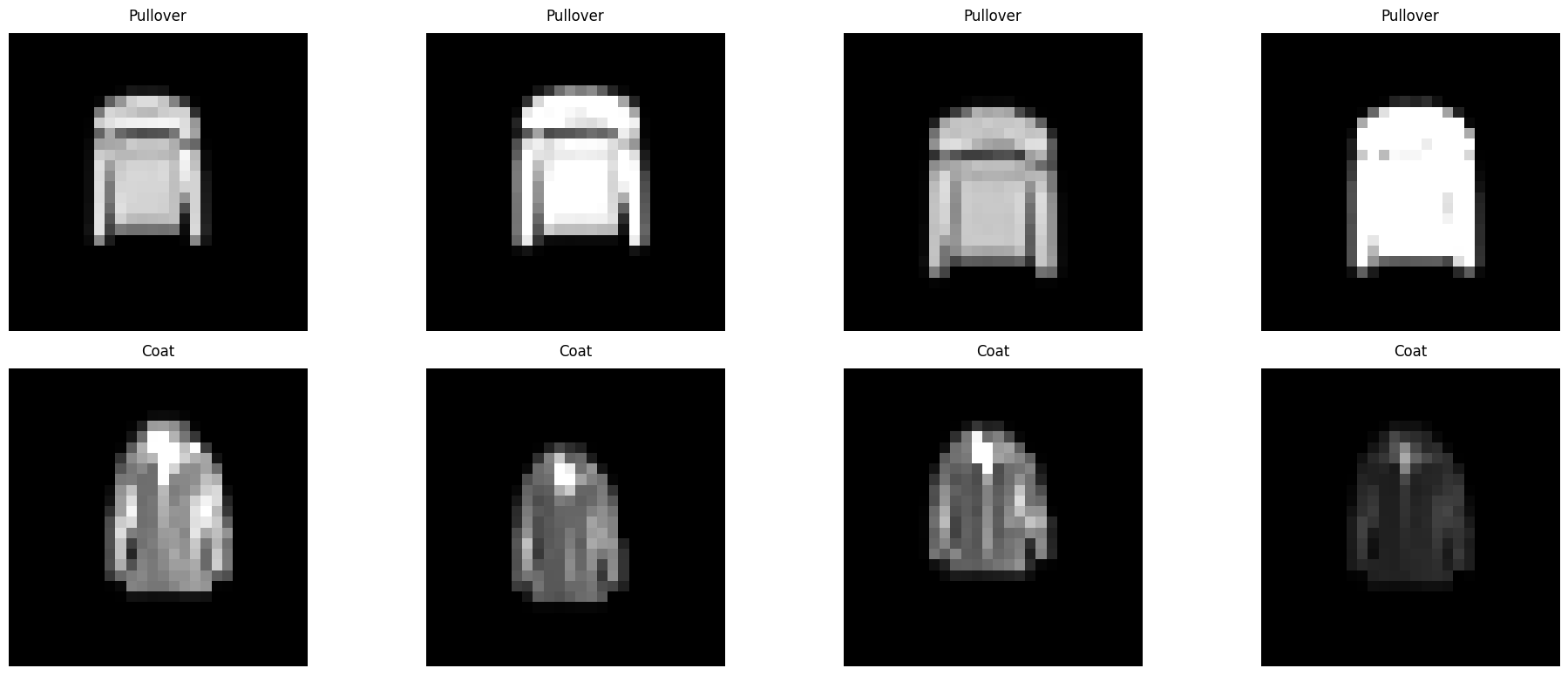}
\caption{Augmented images with label `pullover' (top row) and `coat' (bottom row) from the FashionMNIST dataset.}
\label{ima-pro-example}
\end{figure}

As explored in Section~\ref{ima-inverse-sec}, the inverse mapping classifier \eqref{gen-label} is based on a discretization of the deformation class. For correct classification, the grid size of this discretization should be $\leq$ `small constant' $\times$ `separation distance between the template functions'. The theory focuses on bounds that also apply to the worst-case separation distance $D \asymp 1/d.$ Thus, the grid size becomes of the order $1/d.$ In practice, the separation distance could be much larger which means that a less fine discretization of the deformation class is sufficient. We empirically compare the image alignment classifier with the inverse mapping classifier based on different grid sizes. The performance of each classifier $\wh k$ is evaluated by the empirical misclassification risk (test error)
\begin{align}
R_N = \frac{1}{N} \sum_{i=1}^N \1\big(\, \wh k( \mathbf{X}_{n+i}) \neq k_{n+i}\big),\label{mis-cla-err}
\end{align}
based on test data $(\bX_{n+1},k_{n+1}),\ldots,(\bX_N,k_N),$ that are independently generated from the same distribution as the training data. For $N=50,$ Figure \ref{fig:align-vs-inverse} reports the median of the test error based on $10$ repetitions in each setting.
\begin{figure}[t]
 \centering
 \subfigure[Pullover vs. Coat (FashionMNIST)]{\includegraphics[width=0.49\textwidth]{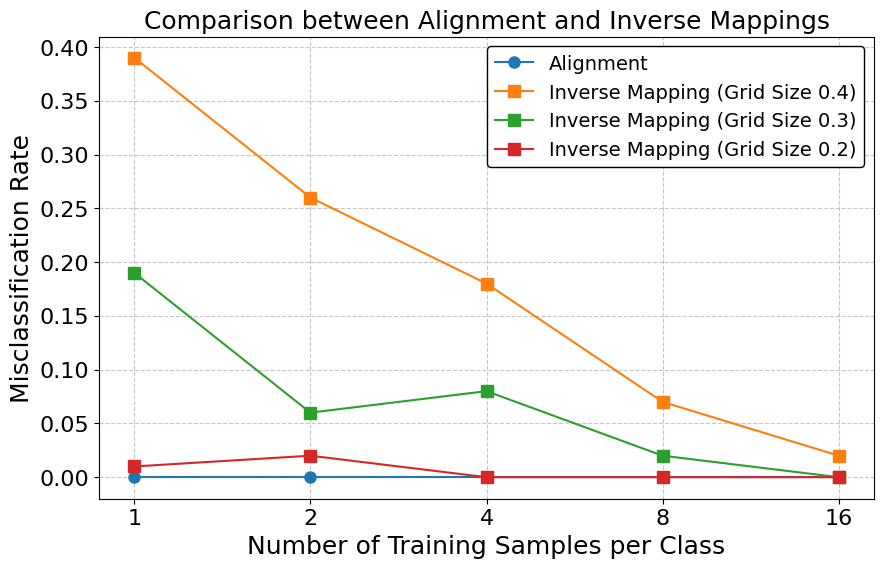}}
  \subfigure[Fox vs. Flatfish (CIFAR-100)]{\includegraphics[width=0.49\textwidth]{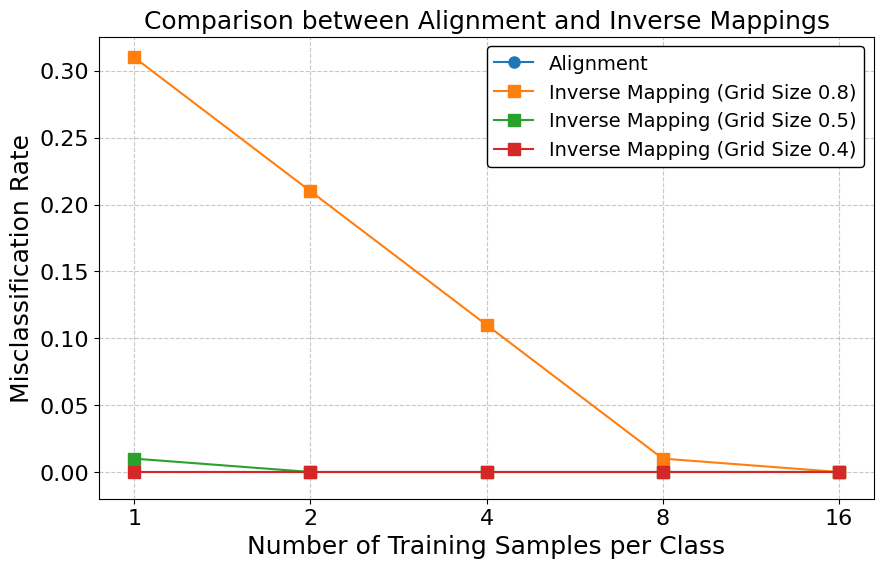}}
  \caption{Comparison of the image alignment classifier (blue) and inverse mapping classifiers with different discretization grid sizes (orange, green, red). The plots show median test errors over 10 repetitions for varying training sample sizes. In (b), both the alignment classifier and the inverse mapping classifier with grid size 0.4 achieve zero misclassification, resulting in the blue and red lines overlapping.}
  \label{fig:align-vs-inverse}
\end{figure}

The results shown in Figure~\ref{fig:align-vs-inverse} align with the theory presented in Section~\ref{sec.image_process}. Specifically, for the image alignment method, the misclassification error is zero as long as each class appears at least once in the training data. For the inverse mapping classifier, the misclassification error decreases as the discretization becomes finer, eventually also reaching zero misclassification error. Moreover, the misclassification error of all inverse mapping classifiers decreases as the training sample size increases. The reason is that more samples yield a denser set of discretized inverse mappings to compare to, thereby improving prediction accuracy.

\subsection{Classifiers via CNN-based approach}\label{sim-cnn-cla}
This section investigates the learning performance of three different CNN architectures with template functions generated from MNIST, FashionMNIST, CIFAR-100, and ImageNet. The MNIST dataset consists of $28\times28$ pixel grayscale images of handwritten digits $(0-9)$. ImageNet consists of labeled high-resolution (typically $224\times224$) color images. We select the template functions by drawing images with label `0' vs.\ `4' (MNIST), `T-shirt' vs.\ `dress' (FashionMNIST), `fox' vs.\ `flatfish' (CIFAR-100), and `cat' vs. `dog' (ImageNet). 

The training and test samples are generated using the same procedure as in Section~\ref{ima-pro-com}, but with a more complex deformation model described by \eqref{affinetran} and \eqref{eq.87etgb}, which also incorporates random rotations. To accelerate training, we rescale the ImageNet images from $224 \times 224$ to $64 \times 64$. Examples of template and deformed images from ImageNet are shown in Figure~\ref{cnn-example}, and examples from the other datasets are provided in Section~\ref{sim-add}.
\begin{figure}
\centering
{\includegraphics[width=\textwidth]{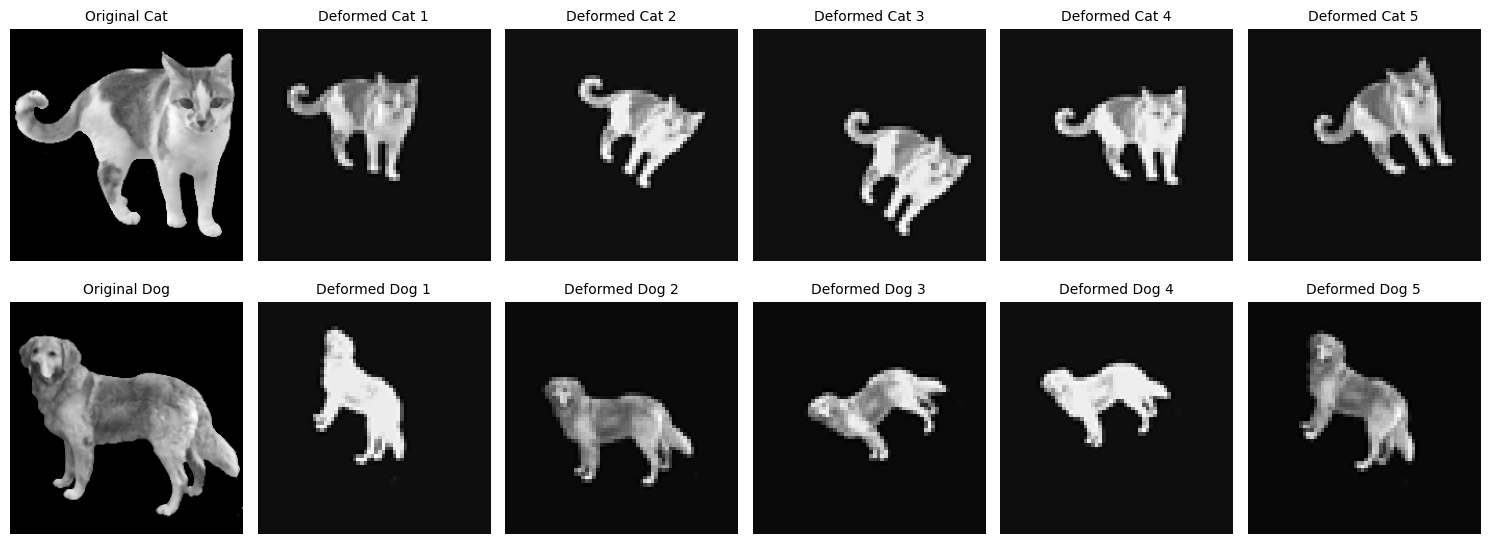}}
\caption{The template cat and dog images ($224\times224$) and their corresponding deformed samples ($64\times64$).}
\label{cnn-example}
\end{figure}

We consider three different CNN architectures, each with one convolutional layer as in Section~\ref{sec.CNN}, sharing the same filter size but differing in the number of filters and the width of the feed-forward layers, as summarized in Table~\ref{tab:cnn-architectures}.
\begin{table}[htbp]
\centering
\begin{tabular}{|c|c|c|c|}
\hline
\textbf{} & \textbf{Filter size}& \textbf{Number of filters} & \textbf{Width of feed-forward layers} \\
\hline
CNN1 & $10\times10$ & 32&128\\
\hline
CNN2 & $10\times10$ & 64 & 256\\
\hline
CNN3 & $10\times10$ & 128 & 512\\
\hline
\end{tabular}
\caption{Details of three CNN architectures.}
\label{tab:cnn-architectures}
\end{table}
All CNNs are trained using cross-entropy loss and the Adam optimizer in Keras (with the TensorFlow backend), employing the default learning rate of $0.001$. For each task, training is conducted on balanced samples with $n/2 = \{2^6, 2^7, 2^8, 2^9, 2^{10}, 2^{11}\}$ samples per class. The performance of each classifier $\widehat{k}$ is evaluated using the misclassification error in \eqref{mis-cla-err}, computed on i.i.d. test samples. For $N = 200$, Figure~\ref{cnn-results} reports the median test errors over 10 repetitions for each CNN classifier across various classification tasks.
\begin{figure}[t]
   \centering
    \subfigure[0 vs. 4 (MNIST)]{\includegraphics[width=0.42\textwidth,height=4.2cm]{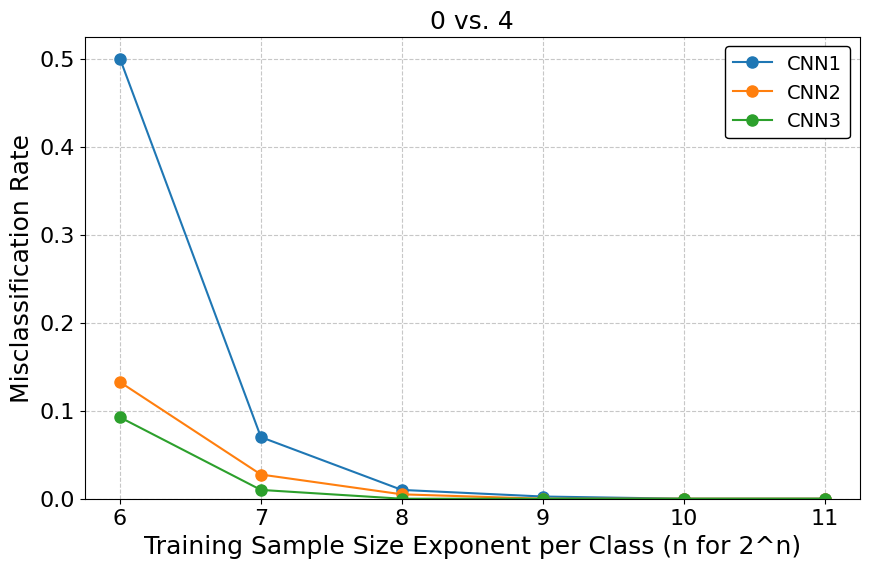}}\hspace{12pt}
    \subfigure[T-shirt vs. Dress (FashionMNIST)]{\includegraphics[width=0.42\textwidth,height=4.2cm]{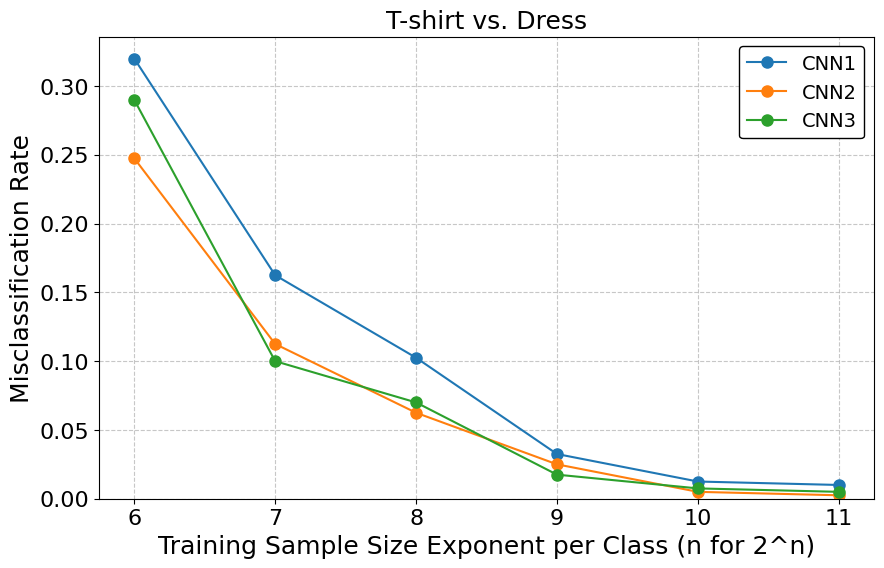}}
\vspace{0.3cm}
    \subfigure[Fox vs. Flatfish (CIFAR-100)]{\includegraphics[width=0.42\textwidth,height=4.2cm]{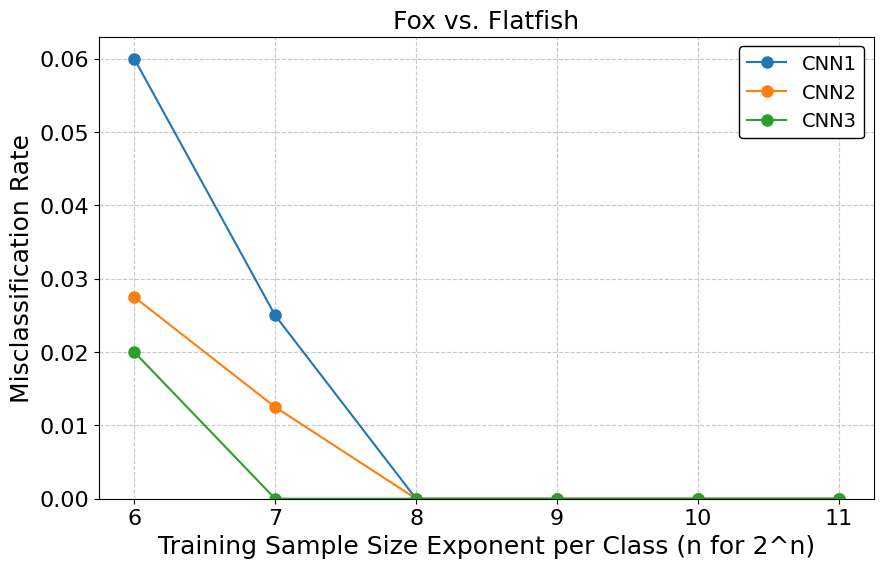}}\hspace{12pt}
    \subfigure[Cat vs. Dog (ImageNet)]{\includegraphics[width=0.42\textwidth,height=4.2cm]{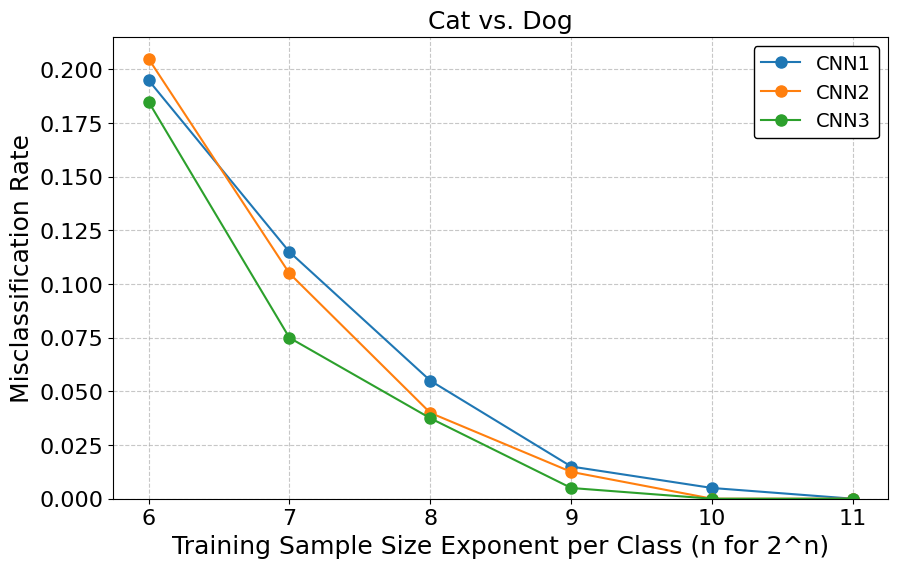}}
\caption{Comparison of three trained CNNs. Reported is the median of the test error with $N=200$ over $10$ repetitions.}
\label{cnn-results}
\end{figure}

For all four classification tasks and all three CNN-based classifiers, the misclassification error decreases and eventually converges to zero as the training sample size increases. The rate of this decay varies across tasks. Among the four tasks, classifying the labels `0' vs.\ `4' and `fox' vs.\ `flatfish' is relatively easy, achieving near-zero error with fewer samples. In contrast, classifying `T-shirt' vs.\ `dress' and `cat' vs.\ `dog' is more challenging, requiring larger sample sizes to reach similarly low error rates. In the comparison among \text{CNN1}, \text{CNN2}, and \text{CNN3}, increasing the number of filters and the width of the fully connected layers enhances performance by enabling a finer approximation of the deformation set $\mathcal{A}$ (Assumption~\ref{ass-cover}). These numerical results are consistent with the theoretical findings presented in Section~\ref{sec.CNN}.

\section{Overview of deformation-based analysis}\label{defor-review}
We already discussed the similarity with the optimization problem \eqref{registr-obj} in image registration.

In functional data analysis (FDA), curves typically exhibit amplitude (vertical) and phase (horizontal) variation. Let $y_1,\ldots,y_n:\mathbb{R}_{+} \rightarrow \mathbb{R}$ be observed functions with both types, and $x_1,\ldots,x_n$ the amplitude-only functions. They can be related by unknown time-warping functions $h_i\in\mathcal{H}: \mathbb{R}_{+} \rightarrow \mathbb{R}_{+}$ such that $y_i(t) = x_i(h_i(t))$ \cite{Liu01092004,kneip2008}. Depending on the application, $h_i$ may take various forms, such as uniform scaling, shifts, or more generally affine transformations. More involved statistical analysis often requires conditions such as smoothness and orientation preservation (i.e., $h_i(0) = 0$, $h_i(1) = 1$, and strict monotonicity); for example, \cite{srivastava2011} considers orientation-preserving diffeomorphisms. Additionally, some works impose the unbiasedness condition $\mathbf{E}[h_i(t)] = t$ \cite{10.1093/biomet/asn047, paz2016}.

When phase variation is treated as a nuisance \cite{MR3432837, srivastava2011}, the goal is to remove it to improve feature estimation, such as the mean curve. This is achieved through curve registration (also referred to as curve alignment in biology), encompassing a wide range of techniques, including earlier heuristics like dynamic time warping (DTW) and landmark registration. These alignment techniques use a template function $x_0$ as the target, seeking a time-warping function $g_i$ such that $y_i\circ g_i \approx x_0$, where closeness is measured by a chosen metric. For simplicity of illustration, let us assume that $\mathcal{H}$ admits a group structure with identity, inversion, and closure under composition. Most methods optimize objective functions of the form:
\begin{equation}\label{ob-regis}
\mathcal{L}_{\lambda,i}[\mu]=\inf_{h_i\in\mathcal{H}}\big(\|y_i\circ h_i-\mu\|^{2}+\lambda\mathcal{R}(h_i)\big),    
\end{equation}
with $\mathcal{R}$ a regularizer. In \cite{10.1093/biomet/asn047}, the authors take $\mu = y_j$, compute pairwise warping functions $h_{ij}$ for each $j$ via \eqref{ob-regis}, and average them to obtain $g_i$. They establish an upper bound on the sup-norm distance between estimated and true warping functions. In contrast, \cite{kneip2008} adopt a template-based method, assigning each $i$ a distinct $\mu_i = \sum_{j=1}^p \beta_{ij} \xi_j$, where the $\xi_j$ are data-driven basis elements closely related to the principal components. Besides implementing optimization as in \eqref{ob-regis}, another approach models amplitude and phase variations through equivalence classes, drawing inspiration from shape analysis \cite{dryden2016} and pattern theory \cite{grenander1993}. This absorbs non-shape attributes like translation, rotation, and scaling into equivalence classes \cite{younes2008,5601739}, motivating the use of a metric $d$ that satisfies the invariance for all $h\in\mathcal{H}$,
\begin{equation}\label{invariance-rep}
d(x_1,x_2) = d(x_1 \circ h, x_2 \circ h).  
\end{equation}
In \cite{srivastava2011,TUCKER201350,Lu02102017}, the authors use the Fisher-Rao Riemannian metric and represent each function $x(t)$ by its square-root velocity function (SRVF) $q(t) := x'(t)/\sqrt{|x'(t)|}$. This links the Fisher-Rao metric between functions $x_1, x_2$ to the Euclidean distance between their SRVFs $q_1, q_2$, making distance computation feasible. The strict monotonicity of $h$ is essential for the derivation. Under certain conditions, \cite{srivastava2011} prove that their alignment algorithm, based on computing the Fr\'echet mean, yields a consistent estimator of the original warping-free function. From the same perspective, \cite{Lu02102017} developed a Bayesian model for estimating warping functions, supported by numerical experiments. Compared with the above literature, our work may be regarded as an adaptation of the concept of phase variation to the context of 2D image classification, where such variability is modeled as geometric deformations. This adaptation is non-trivial, as conditions commonly required in FDA, such as orientation preservation, are no longer appropriate in the context of 2D image classification. Additionally, while the concept of equivalence classes is intuitive for classification tasks, pursuing strict invariance often goes beyond what is necessary or beneficial in this setting. Indeed, if each class with deformations is regarded as an orbit, the goal is not to have two parallel orbits as in \eqref{invariance-rep}, but rather to ensure that the distance between orbits is sufficiently large to offset errors introduced by pixel-based image representation.

Another related line of research is Mallat's series of works on scattering transforms \cite{MR2957703,5995635,M12} in pattern recognition, which offers an alternative representation approach to achieving invariance. Specifically, given a signal $f$ on $\mathbb{R}^d$, they focus on deformations modeled as small diffeomorphisms close to translations, expressed by $L_{\tau}f(x) = f(x - \tau(x))$, where $\tau(x) \in \mathbb{R}^d$ is a displacement field. The scattering representation $\Phi$ is constructed using a wavelet scattering network, which cascades wavelet transforms with nonlinear modulus and averaging operators \cite{bruna9}. As demonstrated in \cite{MR2957703}, this representation is translation-invariant (Theorem 2.10) and Lipschitz continuous with respect to the action of diffeomorphisms on compactly supported functions (Corollary 2.15), satisfying
$$\|\Phi(L_{\tau}f)-\Phi(f)\|\lesssim\|f\|_{2}\left(\sup_{x}|\nabla\tau(x)|+\sup_{x}|H\tau(x)|\right),$$ 
where $\nabla \tau$ denotes the deformation gradient tensor and $H \tau$ the Hessian tensor. With additional constructions, scattering networks can also be made rotation-invariant (Section 5 of \cite{MR2957703}). A key difference from CNNs, which we study in this paper, is that scattering networks are not learned from data but explicitly designed based on prior requirements to ensure invariance or stability to specific signal deformations. This inherently limits their flexibility to pre-designed invariances. These works contribute to the broader field of invariant representation learning. Related research includes \cite{10.1093/imaiai/iaw009,5459468soatto}, where \cite{10.1093/imaiai/iaw009} considers deformations forming a compact group and shows that representations defined by nonlinear group averages are invariant under the group actions.

\section{Conclusion and extensions}
\label{s6}
This paper introduces a novel statistical framework for image classification. Instead of treating each pixel as a variable and analyzing a nonparametric denoising problem where randomness occurs as additive noise, the proposed deformation framework models the variability of objects within one class as geometric deformations of template functions. The abstract framework encompasses a wide range of linear and nonlinear deformations, including commonly considered transformations such as rotations, shifts, and rescaling.

In real world images, deformations can be highly complex, but the analytical approach discussed here may still offer insights. Extensions to background noise and multiple objects have been briefly discussed in Section \ref{sec.image_alignment}. The CNN approach in Section~\ref{sec.CNN} seems quite robust to independent background noise in the image as it relies on inner products, where, by the CLT, the noise should be rather small relative to the object signal. However, analyzing complex, structured backgrounds is more challenging and requires proper statistical modeling. Modifying the approach in Sections~\ref{sec.image_process} and~\ref{sec.CNN}, misclassification error bounds can be derived for the case where fixed images from two different classes are randomly occluded. Closely related are extensions to partially visible objects. In this scenario relying on global characteristics, such as the full support of the object, seems unreasonable. However, it might still be possible to construct classifiers, that provide similar theoretical guarantees by focusing on local properties instead. The presence of sharp edges corresponds to a locally large (or even infinite) Lipschitz constant $C_L$ in the template functions and requires a more refined analysis.

A further potential extension is to incorporate perspective transformations from the computer vision literature \cite{perlay,pernet}. The underlying idea is that images captured from different perspectives can be modeled as
\begin{align*}
    \begin{pmatrix}
        \widetilde a_1(u,v)\\
        \widetilde a_2(u,v)\\
        w(u,v)
    \end{pmatrix}
    =
    \begin{pmatrix}
        h_{11} & h_{12} & h_{13}\\
        h_{21} & h_{22} & h_{23}\\
        h_{31} & h_{32} & h_{33}
    \end{pmatrix}
    \begin{pmatrix}
        u\\
        v\\
        1
    \end{pmatrix},
\end{align*} 
 where $h_{ij}$ are the parameters of the non-singular homography matrix and $w$ is the so-called scaling factor. In this framework, $a_1(u,v)$ and $a_2(u,v)$ are obtained by normalizing the output by $w$, namely $a_1(u,v)=\widetilde a_1(u,v)/w(u,v)$ and $a_2(u,v)=\widetilde a_2(u,v)/w(u,v)$. Affine transformations can be recovered as a special case by choosing $h_{31}=h_{32}=0$ and $h_{33}=1$. To include perspective transformations one needs to relax the partial differentiability imposed in Assumption \ref{ass1}. 

The work in \cite{Mumford2000} and Section 5 of \cite{MR2723182} discuss further deformation classes, including those that account for noise and blur, multi-scale superposition, domain warping and interruptions.

Additionally, various sophisticated image deformation models have been proposed for medical image registration. As mentioned in \eqref{registr-obj}, image registration seeks an optimal transformation mapping the target image to the source image. To study and compare image registration methods, it is essential to construct realistic image deformation models describing the generation of the deformed image from the template. The survey article \cite{6522524} classifies these deformation models into several categories, such as ODE/PDE based models, interpolation-based models and knowledge-based models. For instance, a simple ODE based random image deformation model takes $\bX$ as template/source image and generates a random vector field $u.$ This can be achieved by selecting a basis and generating independent random coefficients according to a fixed distribution. Given the vector field $u,$ a continuous image deformation $\bX(t)$ is generated by solving the differential equation $\partial_t \bX(t) = u(\bX(t))$ with $\bX(0)=\bX.$ The randomly deformed image is then $\bX(1)$. The DARTEL algorithm \cite{ASHBURNER200795} is a widely recognized approach for image registration within this deformation model. A statistical analysis of these methods is still lacking.

\section*{Acknowledgement}
A first version of the paper was written by S.L. and J.S-H. The current version is a generalisation that has been written as a joint effort including J.C. The research of J. S.-H. and J.C. was supported by the NWO Vidi grant VI.Vidi.192.021. S.L. was supported in part by the NWO Veni grant VI.Veni.232.033. The authors sincerely thank Yanis Bosch and Mark Podolskij for pointing out a mistake in the previous version of the manuscript.
\appendix

\section{Proofs for Section \ref{sec.image_process}}
\label{app.proofs_image_alignm}
Throughout this section, assume that $f$ is one of the template functions $f_0, f_1.$
\subsection{Proofs for general deformation model}\label{app.proofs_image_alignm_gen}
\begin{lemma}\label{Lip-composite}
If the function $f$ satisfies Assumption~\ref{ass-1} and $A\in\mathcal{A}$ satisfies Assumption~\ref{ass1}-(i), then, $f\circ A$ is Lipschitz continuous in the sense that for any $(u,v),(u',v')\in\mathbb{R}^2$,
\[
|f\circ A(u,v)-f\circ A(u',v')|\leq2C_{\mathcal{A}}C_L\|f\|_1(|u-u'|+|v-v'|).
\]
\end{lemma}

\begin{proof}
Recall that $A=(a_1,a_2)$. Note that Assumption~\ref{ass1}-(i) implies that for any $(u,v),(u',v')\in\mathbb{R}^2$ and $k=1,2$,
\begin{align*}
      |a_{k}(u,v)-a_{k}(u',v')|\leq C_\mathcal{A} \big(|u-u'|+|v-v'|\big).
\end{align*}
Together with the Lipschitz continuity of $f$, this implies that for any $(u,v),(u',v')\in\mathbb{R}^2$,
\begin{align*}
|f\circ A(u,v)-f\circ A(u',v')|&=|f(a_1(u,v),a_2(u,v))-f(a_1(u',v'),a_2(u',v'))|\\
&\leq C_L\|f\|_1\left(\left|a_1(u,v)-a_1(u',v')\right|+\left|a_2(u,v)-a_2(u',v')\right|\right)\\
&\leq C_L\|f\|_1 2C_{\mathcal{A}}(|u-u'|+|v-v'|)\\
&=2C_{\mathcal{A}}C_L\|f\|_1(|u-u'|+|v-v'|).
\end{align*}
\end{proof}

\begin{lemma}\label{support-range}
Suppose $[\beta_{\text{left}}, \beta_{\text{right}}] \times [\beta_{\text{down}}, \beta_{\text{up}}] \subseteq [0,1]^2$ and Assumption~\ref{ass1}-(i) holds. For all $A \in \mathcal{A}$, we have $$A\big([0,1]^2\big) \subseteq D_{\mathcal{A}}=[-2C_{\mathcal{A}} - 1, 2C_{\mathcal{A}} + 1]^2.$$
\end{lemma}
\begin{proof}
Recall that $A = (a_1, a_2)$. Under Assumption~\ref{ass1}-(i), for any $(u,v), (u',v') \in [0,1]^2$ and $k = 1, 2$,
\begin{equation}\label{length-b}
|a_k(u,v) - a_k(u',v')| \leq C_{\mathcal{A}}\big(|u - u'| + |v - v'|\big)\leq2C_{\mathcal{A}}.
\end{equation}

Take a point $(u_0,v_0)\in[\beta_{\text{left}},\beta_{\text{right}}]\times[\beta_{\text{down}},\beta_{\text{up}}]\subseteq\cro{0,1}^2$. Since, under Assumption~\ref{ass1}-(i), we have $$[\beta_{\text{left}},\beta_{\text{right}}]\times[\beta_{\text{down}},\beta_{\text{up}}]\subseteq A([0,1]^2),$$ there exists a point $(u',v')\in[0,1]^2$ such that $(u_0,v_0)=A(u',v')=(a_1(u',v'),a_2(u',v'))$. By \eqref{length-b}, for any $(u,v) \in [0,1]^2$ and $k=1,2$,
$$|a_{k}(u,v)|\leq2C_{\mathcal{A}}+|a_{k}(u',v')|\leq2C_{\mathcal{A}}+1.$$
This completes the proof.
\end{proof}

\begin{lemma}\label{inverse-l2-connect}
Let $f$ satisfy Assumption~\ref{ass-1}, and let $\mathcal{A}$ satisfy Assumption~\ref{ass1}. Then, for any $A_1,A_2\in\mathcal{A}$,
$$\frac{C_{\mathcal{J}}}{\sqrt{2}C_{\mathcal{A}}}\|f\circ A_1\circ A_{2}^{-1}\|_{L^2(\mathbb{R}^2)}\leq\|f\|_{2}\leq\frac{\sqrt{2}C_{\mathcal{A}}}{C_{\mathcal{J}}}\|f\circ A_1\circ A_{2}^{-1}\|_{L^2(\mathbb{R}^2)}.$$
\end{lemma}
\begin{proof}
Under Assumption~\ref{ass1}, $\mathcal{A}$ contains the identity, which implies that $C_{\mathcal{A}}\geq1$ and $C_{\mathcal{J}}\leq1$, ensuring that the inequalities are well-defined. 

We prove only the second inequality, as the first can be shown using a similar argument. Recall that $J_{A}(x,y)$ represents the Jacobian matrix of $A$ at $(x,y)$. Since, under Assumptions~\ref{ass-1} and \ref{ass1}-(i), the supports of $f$ and $f\circ A_1$ are both contained in $[0,1]^2$, and the partial derivatives of $a_1$ and $a_2$ are bounded in the supremum norm by $C_{\mathcal{A}}$, we obtain jointly with the change of variables theorem
\begin{align}
\|f\|^{2}_{2}&=\int_{[0,1]^2}f^2(u,v)dudv\nonumber\\
&=\int_{\mathbb{R}^2}f^2(u,v)dudv\nonumber\\
&=\int_{\mathbb{R}^2}\cro{f(A_1(x,y))}^2\cdot|\det(J_{A_1}(x,y))|dxdy\nonumber\\
&\leq 2C_{\mathcal{A}}^2\|f\circ A_1\|^2_{2}.\label{connect-1}
\end{align}
Moreover, under Assumption~\ref{ass1}-(ii), 
\begin{align}
\|f\circ A_1\circ A_{2}^{-1}\|_{L^2(\mathbb{R}^2)}^2&=\int_{\mathbb{R}^2}\cro{f(A_1\circ A_2^{-1}(u,v))}^2dudv\nonumber\\
&=\int_{\mathbb{R}^2}\cro{f(A_1(x,y))}^2\cdot|\det({J_{A_{2}}(x,y)})|dxdy\nonumber\\
&\geq C_{\mathcal{J}}^2\|f\circ A_1\|_{L^2(\mathbb{R}^2)}^2\nonumber\\
&=C_{\mathcal{J}}^2\|f\circ A_1\|_{2}^2.\label{connect-2}
\end{align}
Combining \eqref{connect-1} and \eqref{connect-2} yields
$$\|f\|_2^2\leq\frac{2C_{\mathcal{A}}^2}{C_{\mathcal{J}}^2}\|f\circ A_1\circ A_{2}^{-1}\|_{L^2(\mathbb{R}^2)}^2,$$ which implies the conclusion.
\end{proof}

\begin{prop}\label{dis-A-B}
Given Assumptions \ref{ass-1} and \ref{ass1}-(i), let $\mathcal{A}_{d}^{-1
}$ be a covering of $\mathcal{A}^{-1}$ with balls of radius $1/d$ satisfying \eqref{gen-conver} with $D_{\mathcal{A}}=[-2C_{\mathcal{A}} - 1, 2C_{\mathcal{A}} + 1]^2$. Then, for any $A_{*}^{-1}\in\mathcal{A}^{-1}$ and any $B_{*}\in\mathcal{A}^{-1}_{d}$ such that $\|A^{-1}_{*}-B_{*}\|_{L^{\infty}(D_{\mathcal{A}})}\leq1/d$, it holds that
\begin{align*}
\left|\bX\big(A_{*}^{-1}(u,v)\big)-\bX\big(B_{*}(u,v)\big)\right|\leq 8\eta C_\mathcal{A}C_L\|f\|_1\frac{1}{d},\quad\mbox{for all}\ (u,v)\in D_{\mathcal{A}}.
\end{align*}
\end{prop}

\begin{proof}
Write $A_{*}^{-1}=(b_{1},b_{2})$ and $B_{*}=(b_{1}^{*},b_{2}^{*})$. For any $(u,v)\in D_{\mathcal{A}}\subseteq\mathbb{R}^2$, there exist integers $j,\ell$ such that $A_{*}^{-1}(u,v)\in I_{j,\ell}$ and integers $j',\ell'$ such that $B_{*}(u,v)\in I_{j',\ell'}$. The $I_{j,\ell}$ and $I_{j',\ell'}$ represent the pixel locations in the original image $\mathbf{X}$ before applying the mappings $A_{*}^{-1}$ and $B_{*}$, respectively. 

We first deal with the case where both $I_{j,\ell}$ and $I_{j',\ell'}$ are contained in $\cro{0,1}^2$. By the definition of $\mathcal{A}_{d}^{-1}$ and the fact that $\|A^{-1}_{*}-B_{*}\|_{L^{\infty}(D_{\mathcal{A}})}\leq1/d$, we know for any $(u,v)\in D_{\mathcal{A}}$,
\begin{equation*}
\left|b_{1}(u,v)-b_{1}^{*}(u,v)\right|\leq\frac{1}{d}\quad\mbox{and}\quad\left|b_{2}(u,v)-b_{2}^{*}(u,v)\right|\leq\frac{1}{d},
\end{equation*}
which implies that 
\begin{equation}
\left|j-j'\right|\leq 1\quad\mbox{and}\quad\left|\ell-\ell'\right|\leq 1.\label{lip-A-dis}
\end{equation}
As a consequence of \eqref{lip-A-dis}, for any $(x,y)\in I_{j,\ell}$ and any $(x',y')\in I_{j',\ell'}$,
\begin{equation}\label{x-y-dis}
    |x-x'|\leq\frac{|j-j'|+1}{d}\leq\frac{2}{d},\quad|y-y'|\leq\frac{|\ell-\ell'|+1}{d}\leq\frac{2}{d}.
\end{equation}
Recall that the image $\bX=(X_{j,\ell})_{j,\ell=1,\ldots,d}$ is generated by $X_{j,\ell}=d^2\eta\int_{I_{j,\ell}}f\circ A(u,v)dudv$. Under Assumptions~\ref{ass-1} and \ref{ass1}-(i), we know from Lemma~\ref{Lip-composite} that $f\circ A$ is Lipschitz continuous. With Lemma~\ref{Lip-composite} and \eqref{x-y-dis}, we can derive that for any $(x,y)\in I_{j,\ell}$ and any $(x',y')\in I_{j',\ell'}$,
\begin{align}
|f(A(x,y))-f(A(x',y'))|&\leq2C_{\mathcal{A}}C_{L}\|f\|_{1}(|x-x'|+|y-y'|)\nonumber\\
&\leq8C_{\mathcal{A}}C_{L}\|f\|_{1}\frac{1}{d}.\label{lip-difference}
\end{align}
Therefore, under Assumptions~\ref{ass-1} and \ref{ass1}-(i), using \eqref{lip-difference}, we have
\begin{align*}
\left|\bX\big(A_{*}^{-1}(u,v)\big)-\bX\big(B_{*}(u,v)\big)\right| =&\eta\left|d^{2}\int_{I_{j,\ell}}f(A(x,y))dxdy-d^{2}\int_{I_{j',\ell'}}f(A(x,y))dxdy\right|\\
\leq&8\eta C_\mathcal{A}C_L\|f\|_1\frac{1}{d}.
\end{align*}
This proves the result in this case.

Next, we handle the case where neither $I_{j,\ell}$ nor $I_{j',\ell'}$ is contained in $\cro{0,1}^2$. According to the definition of $\bX$ in \eqref{image-as-function}, we have 
\begin{align*}
\left|\bX\big(A_{*}^{-1}(u,v)\big)-\bX\big(B_{*}(u,v)\big)\right|=0,
\end{align*}
which satisfies the conclusion.

Finally, we consider the case where $I_{j,\ell}\subseteq\cro{0,1}^2$ but $I_{j',\ell'}$ is not contained in $\cro{0,1}^2$. If $I_{j',\ell'}\subseteq\cro{0,1}^2$ but $I_{j,\ell}$ is not contained in $\cro{0,1}^2$, the proof is the same and therefore omitted. Under Assumption~\ref{ass1}-(i), the image is fully visible hence $\int_{I_{j',\ell'}}f\big(A(x,y)\big)dxdy=0$, if $I_{j',\ell'}$ is not contained in $\cro{0,1}^2$. As Assumptions~\ref{ass-1},~\ref{ass1}-(i) hold, we similarly derive using Lemma~\ref{Lip-composite} and $\|A^{-1}_{*}-B_{*}\|_{L^{\infty}(D_{\mathcal{A}})}\leq1/d$ that
\begin{align*}
\left|\bX\big(A_{*}^{-1}(u,v)\big)-\bX\big(B_{*}(u,v)\big)\right|&=\eta\left|d^{2}\int_{I_{j,\ell}}f\big(A(x,y)\big)dxdy-0\right|\\
&=\eta\left|d^{2}\int_{I_{j,\ell}}f\big(A(x,y)\big)dxdy-d^{2}\int_{I_{j',\ell'}}f\big(A(x,y)\big)dxdy\right|\\
&\leq8\eta C_\mathcal{A}C_L\|f\|_1\frac{1}{d},
\end{align*}
proving the claim also in this case. 
\end{proof}

\begin{lemma}
\label{general-error}
Consider an image $\bX=(X_{j,\ell})_{j,\ell=1,\ldots,d}$ generated by $X_{j,\ell}=d^2\eta\int_{I_{j,\ell}}f\circ A(u,v)dudv$, and assume that Assumptions \ref{ass-1} and \ref{ass1} hold. Then, for any $A_{*}\in\mathcal{A}$, there exists a universal constant $K>0$ such that, for any $B_{*}\in\mathcal{A}^{-1}_{d}$ satisfying $\|A^{-1}_{*}-B_{*}\|_{L^{\infty}(D_{\mathcal{A}})}\leq1/d$ with $D_{\mathcal{A}}=[-2C_{\mathcal{A}} - 1, 2C_{\mathcal{A}} + 1]^2$,
\begin{align*}
\left\|T_{\bX\circ B_{*}}- \frac{f\circ A\circ A_{*}^{-1}}{\|f\circ A\circ A_{*}^{-1}\|_{L^2(\mathbb{R}^2)}}\right\|_{L^2(\mathbb{R}^2)} \leq K\max\{C_L^2C_{\mathcal{J}}^{-2}(C_{\mathcal{A}}+1)^{6},C_LC_{\mathcal{J}}^{-1}(C_{\mathcal{A}}+1)^3\}\frac{1}{d}.
\end{align*}
\end{lemma}

\begin{proof}
As a first step of the proof, we fix $A_{*}\in\mathcal{A}$ and show that for any $B_{*}\in\mathcal{A}^{-1}_{d}$ satisfying $\|A^{-1}_{*}-B_{*}\|_{L^{\infty}(D_{\mathcal{A}})}\leq1/d$, 
\begin{align}
\label{second-summ-d}
\Big|\bX\big(B_{*}(u,v)\big) - \eta f\circ A\circ A_{*}^{-1}(u,v)\Big| \leq 12\eta C_\mathcal{A}C_L\|f\|_1\frac{1}{d}, \quad \text{for all} \ (u,v)\in D_{\mathcal{A}}\subseteq\mathbb{R}^2,
\end{align}
where $\bX\big(B_{*}(u,v)\big)$ is as defined in \eqref{discre-inverse-deform}.

For any $(u,v)\in D_{\mathcal{A}}\subseteq\mathbb{R}^2$, set $(u_0,v_0)=A_{*}^{-1}(u,v)$. We first consider the case where $(u_0,v_0)\in[0,1)^2$, which implies that there exist $j,\ell\in\{1,\ldots,d\}$ such that $(u_0,v_0)\in I_{j,\ell}$. With the definition of $\bX$ in \eqref{image-as-function}, we then compute that 
\begin{align}
\left|\bX\big(A_{*}^{-1}(u,v)\big) - \eta f\circ A\circ A_{*}^{-1}(u,v)\right|&=\left|\bX(u_0,v_0) - \eta f\circ A\circ A_{*}^{-1}(u,v)\right|\nonumber\\
&=\eta\left|d^{2}\int_{I_{j,\ell}}f\big(A(x,y)\big)dxdy-f\big(A(u_0,v_0)\big)\right|.\label{inverse-1}
\end{align}
For any $(x,y)\in I_{j,\ell}$, due to the Lipschitz continuity of $f$ and $A$, under Assumptions~\ref{ass-1} and \ref{ass1}-(i), we obtain by applying Lemma~\ref{Lip-composite} that
\begin{align}
\left|f\big(A(x,y)\big)-f\big(A(u_0,v_0)\big)\right|
&\leq 2C_\mathcal{A}C_L \|f\|_1\big(|x-u_0|+|y-v_0|\big)\nonumber\\
&\leq4C_\mathcal{A}C_L \|f\|_1\frac{1}{d}.\label{lip-bound-all}
\end{align}
With \eqref{lip-bound-all}, we deduce from \eqref{inverse-1} that 
\begin{align}
\left|\bX\big(A_{*}^{-1}(u,v)\big) - \eta f\circ A\circ A_{*}^{-1}(u,v)\right|&=\eta\left|d^{2}\int_{I_{j,\ell}}f\big(A(x,y)\big)dxdy-f\big(A(u_0,v_0)\big)\right|\nonumber\\
&\leq d^{2}\eta\int_{I_{j,\ell}}\left|f\big(A(x,y)\big)-f\big(A(u_0,v_0)\big)\right|dxdy\nonumber\\
&\leq4\eta C_\mathcal{A}C_L \|f\|_1\frac{1}{d}.\label{perfect-inverse-1}
\end{align}

Then, we consider all $(u,v)\in D_{\mathcal{A}}\subseteq\mathbb{R}^2$ such that $(u_0,v_0)=A_{*}^{-1}(u,v)\notin[0,1)^2$. In this case, by definition of $\bX$ in \eqref{image-as-function}, $\bX(A_{*}^{-1}(u,v))=0$. Moreover, the value of $f\big(A(u_0,v_0)\big)$ must be zero; otherwise, this contradicts Assumption~\ref{ass1}-(i), which ensures that the image is fully visible, as $f\circ A$ is Lipschitz continuous according to Lemma~\ref{Lip-composite}. Therefore, 
\begin{equation}\label{boundary-case}
\left|\bX\big(A_{*}^{-1}(u,v)\big) - \eta f\circ A\circ A_{*}^{-1}(u,v)\right|=\left|\bX(u_0,v_0) - \eta f\big(A(u_0,v_0)\big)\right|=0.
\end{equation}

Combining \eqref{perfect-inverse-1} and \eqref{boundary-case}, we obtain that for any $(u,v)\in D_{\mathcal{A}}\subseteq\mathbb{R}^2$,
\begin{equation}\label{support-bound-image}
\left|\bX\big(A_{*}^{-1}(u,v)\big) - \eta f\circ A\circ A_{*}^{-1}(u,v)\right|\leq4\eta C_\mathcal{A}C_L \|f\|_1\frac{1}{d}.
\end{equation}

Under the condition $\|A_{*}^{-1}-B_{*}\|_{L^{\infty}(D_{\mathcal{A}})}\leq1/d$, we deduce from \eqref{support-bound-image} and Proposition~\ref{dis-A-B} that, for any $(u,v)\in D_{\mathcal{A}}$,
\begin{align*}
\left|\bX\big(B_{*}(u,v)\big) - \eta f\circ A\circ A_{*}^{-1}(u,v)\right|&\leq\left|\bX\big(B_{*}(u,v)\big)-\bX\big(A_{*}^{-1}(u,v)\big)\right|+\left|\bX\big(A_{*}^{-1}(u,v)\big) - \eta f\circ A\circ A_{*}^{-1}(u,v)\right|\nonumber\\
&\leq\left|\bX\big(B_{*}(u,v)\big)-\bX\big(A_{*}^{-1}(u,v)\big)\right|+4\eta C_\mathcal{A}C_L \|f\|_1\frac{1}{d}\nonumber\\
&\leq12\eta C_\mathcal{A}C_L\|f\|_1\frac{1}{d}.
\end{align*}

In the next step, we show that 
\begin{align}
\label{inverse-2}
    \left|\|\bX\circ B_{*}\|_{L^2(\mathbb{R}^2)} - \eta\|f\circ A\circ A_{*}^{-1}\|_{L^2(\mathbb{R}^2)}\right| \leq 3355\eta C_{L}(C_{\mathcal{A}}+1)^2\max\{C_LC_{\mathcal{J}}^{-1}C_{\mathcal{A}}^3,1\}\|f\|_{1}\frac{1}{d}.
\end{align}
Using that for real numbers $a,b\not=0$ $a-b=(a^2-b^2)/(a+b),$ we rewrite
\begin{align}
&\left|\|\bX\circ B_{*}\|_{L^2(\mathbb{R}^2)} - \eta\|f\circ A\circ A_{*}^{-1}\|_{L^2(\mathbb{R}^2)}\right|\nonumber\\ &= \left|\|\bX\circ B_{*}\|_{L^2(\mathbb{R}^2)}^2- \eta^2\|f\circ A\circ A_{*}^{-1}\|_{L^2(\mathbb{R}^2)}^2\right| \frac{1}{\|\bX\circ B_{*}\|_{L^2(\mathbb{R}^2)}+\eta\|f\circ A\circ A_{*}^{-1}\|_{L^2(\mathbb{R}^2)}}\nonumber\\ 
&\leq \left|\|\bX\circ B_{*}\|_{L^2(\mathbb{R}^2)}^2- \eta^2\|f\circ A\circ A_{*}^{-1}\|_{L^2(\mathbb{R}^2)}^2\right|\frac{1}{\eta\|f\circ A\circ A_{*}^{-1}\|_{L^2(\mathbb{R}^2)}}.\label{inverse-3}
\end{align}
We bound the first term in \eqref{inverse-3} by 
\begin{align}
&\left|\|\bX\circ B_{*}\|_{L^2(\mathbb{R}^2)}^2- \eta^2\|f\circ A\circ A_{*}^{-1}\|_{L^2(\mathbb{R}^2)}^2\right| \nonumber\\
&= \left|\int_{\mathbb{R}^2} \left(\bX\big(B_{*}(u,v)\big) - \eta f\circ A\circ A_{*}^{-1}(u, v)+ \eta f\circ A\circ A_{*}^{-1}(u, v)\right)^2dudv
- \eta^2\|f\circ A\circ A_{*}^{-1}\|_{L^2(\mathbb{R}^2)}^2\right|\nonumber\\
& = \left| \int_{\mathbb{R}^2} \left(\bX\big(B_{*}(u,v)\big) - \eta f\circ A\circ A_{*}^{-1}(u, v)\right)^2 dudv\right.\nonumber\\
& \left. \quad 
    + 2 \eta\int_{\mathbb{R}^2}\left(\bX\big(B_{*}(u,v)\big) - \eta f\circ A\circ A_{*}^{-1}(u, v)\right)\cdot\cro{f\circ A\circ A_{*}^{-1}(u, v) }dudv\right.
    \nonumber\\
    & \left. \quad 
    +
    \int_{\mathbb{R}^2} \eta^2\cro{f\circ A\circ A_{*}^{-1}(u,v)}^2dudv -\eta^2\|f\circ A\circ A_{*}^{-1}\|_{L^2(\mathbb{R}^2)}^2\right|\nonumber\\
    & \leq \int_{\mathbb{R}^2} \left|\bX\big(B_{*}(u,v)\big)- \eta f\circ A\circ A_{*}^{-1}(u, v)\right|^2 dudv\nonumber\\
    &\quad+2\eta \int_{\mathbb{R}^2} \left|\bX\big(B_{*}(u,v)\big) - \eta f\circ A\circ A_{*}^{-1}(u, v) \right|\cdot\left|f\circ A\circ A_{*}^{-1}(u, v)\right|dudv.\label{R-square-int}
\end{align}
Since the support of $f\circ A$ and the support of $\bX$ are both contained within $[0,1]^2$, applying Lemma~\ref{support-range} shows that the support of $f\circ A\circ A_{*}^{-1}$ and the support of $\bX\circ A_{*}^{-1}$ are both contained within $D_{\mathcal{A}}=[-2C_{\mathcal{A}}-1,2C_{\mathcal{A}}+1]^2$. Moreover, since $\mathcal{A}_d^{-1}$ is a subset of $\mathcal{A}^{-1}$, and thus $B_*^{-1}\in\mathcal{A}$, applying Lemma~\ref{support-range} again implies that the support of $\bX\circ B_*$ is also contained within $D_{\mathcal{A}}$. Using \eqref{second-summ-d}, we can derive from \eqref{R-square-int} that
\begin{align*}
&\left|\|\bX\circ B_{*}\|_{L^2(\mathbb{R}^2)}^2- \eta^2\|f\circ A\circ A_{*}^{-1}\|_{L^2(\mathbb{R}^2)}^2\right| \\
& \leq \int_{D_{\mathcal{A}}}  
    12^2\eta^2 C^2_\mathcal{A}C^2_L\|f\|_1^2\frac{1}{d^2} \, dudv+ 24\eta^2C_\mathcal{A}C_L\|f\|_1\frac{1}{d} \int_{\mathbb{R}^2} f\circ A\circ A_{*}^{-1}(u,v) \,  dudv\\
    &\leq12^{2}\eta^2 C^2_\mathcal{A}C^2_L\|f\|_1^2\frac{1}{d^2}\cdot(4C_\mathcal{A}+2)^2+ 24\eta^2 C_\mathcal{A}C_L\|f\|_1\frac{1}{d}\cdot\|f\circ A\circ A_{*}^{-1}\|_{L^1(\mathbb{R}^2)}.
\end{align*}

By the Cauchy-Schwarz inequality, $\|f\|_1\leq \|f\|_2$ and 
\begin{align*}
\|f\circ A\circ A_{*}^{-1}\|_{L^1(\mathbb{R}^2)}&=\int_{D_{\mathcal{A}}}f\circ A\circ A_{*}^{-1}(u,v) \,  dudv\\
&\leq(4C_{\mathcal{A}}+2)\sqrt{\int_{D_{\mathcal{A}}}\cro{f\circ A\circ A_{*}^{-1}(u,v)}^2dudv}\\
&=(4C_{\mathcal{A}}+2)\|f\circ A\circ A_{*}^{-1}\|_{L^2(\mathbb{R}^2)}.   
\end{align*}
Summarizing and applying Lemma~\ref{inverse-l2-connect}, \eqref{inverse-3} is bounded by
\begin{align*}
&\left|\|\bX\circ B_{*}\|_{L^2(\mathbb{R}^2)} - \eta\|f\circ A\circ A_{*}^{-1}\|_{L^2(\mathbb{R}^2)}\right|\\
&\leq 2304\eta C^2_LC^2_\mathcal{A}(C_\mathcal{A}+1)^2\frac{1}{d}\|f\|_1\frac{\|f\|_1}{\|f\circ A\circ A_{*}^{-1}\|_{L^2(\mathbb{R}^2)}}\\
&\quad+24\eta C_\mathcal{A}C_L\frac{1}{d}\|f\|_1\frac{\|f\circ A\circ A_{*}^{-1}\|_{L^1(\mathbb{R}^2)}}{\|f\circ A\circ A_{*}^{-1}\|_{L^2(\mathbb{R}^2)}}\\
&\leq 2304\eta C^2_LC^2_\mathcal{A}(C_\mathcal{A}+1)^2\frac{1}{d}\|f\|_1\frac{\|f\|_2}{\|f\circ A\circ A_{*}^{-1}\|_{L^2(\mathbb{R}^2)}}
+24\eta C_\mathcal{A}C_L\frac{1}{d}\|f\|_1\cdot(4C_{\mathcal{A}}+2)\\
&\leq 2304\eta C^2_LC^2_\mathcal{A}(C_\mathcal{A}+1)^2\frac{1}{d}\|f\|_1\cdot\frac{\sqrt{2}C_{\mathcal{A}}}{C_{\mathcal{J}}}
+96\eta C_\mathcal{A}(C_{\mathcal{A}}+1)C_L\frac{1}{d}\|f\|_1\\
&\leq3355\eta C_{L}(C_{\mathcal{A}}+1)^2\max\{C_LC_{\mathcal{J}}^{-1}C^3_{\mathcal{A}},1\}\|f\|_{1}\frac{1}{d},
\end{align*}
proving \eqref{inverse-2}.

We now finish the proof. Using $T_{\bX\circ B_{*}}=\bX\circ B_{*}/\|\bX\circ B_{*}\|_{L^2(\mathbb{R}^2)}$ and $(a+b)^2 \leq 2 a^2 + 2b^2$ for arbitrary real numbers $a,b$, we bound
\begin{align*}
&\left\|T_{\bX\circ B_{*}} - \frac{f\circ A\circ A_{*}^{-1}}{\|f\circ A\circ A_{*}^{-1}\|_{L^2(\mathbb{R}^2)}}\right\|_{L^2(\mathbb{R}^2)}^2 \\
&\leq 2\left\|\frac{\bX\circ B_{*}}{\|\bX\circ B_{*}\|_{L^2(\mathbb{R}^2)}} - \frac{\bX\circ B_{*}}{\eta\|f\circ A\circ A_{*}^{-1}\|_{L^2(\mathbb{R}^2)}}\right\|_{L^2(\mathbb{R}^2)}^2 \\
&\quad+ 2 \left\|\frac{\bX\circ B_{*}}{\eta\|f\circ A\circ A_{*}^{-1}\|_{L^2(\mathbb{R}^2)}} - \frac{\eta f\circ A\circ A_{*}^{-1}}{\eta\|f\circ A\circ A_{*}^{-1}\|_{L^2(\mathbb{R}^2)}}\right\|_{L^2(\mathbb{R}^2)}^2\\
&\leq2\left(\frac{1}{\|\bX\circ B_{*}\|_{L^2(\mathbb{R}^2)}}-\frac{1}{\eta\|f\circ A\circ A_{*}^{-1}\|_{L^2(\mathbb{R}^2)}}\right)^2\int_{\mathbb{R}^2}\big(\bX\circ B_{*}(u,v)\big)^2dudv\\
&\quad+\frac{2}{\eta^2\|f\circ A\circ A_{*}^{-1}\|_{L^2(\mathbb{R}^2)}^2} \left\|\bX\circ B_{*} - \eta f\circ A\circ A_{*}^{-1}\right\|_{L^2(\mathbb{R}^2)}^2\\
&\leq\frac{2}{\eta^2\|f\circ A\circ A_{*}^{-1}\|_{L^2(\mathbb{R}^2)}^2} \left|\eta\|f\circ A\circ A_{*}^{-1}\|_{L^2(\mathbb{R}^2)}- \|\bX\circ B_{*}\|_{L^2(\mathbb{R}^2)} \right|^2\\
&\quad+\frac{2}{\eta^2\|f\circ A\circ A_{*}^{-1}\|_{L^2(\mathbb{R}^2)}^2} \left\|\bX\circ B_{*} - \eta f\circ A\circ A_{*}^{-1}\right\|_{L^2(\mathbb{R}^2)}^2.
\end{align*}
Applying \eqref{inverse-2} to the first and \eqref{second-summ-d} to the second term and using again $\|f\|_1\leq \|f\|_2$ and Lemma~\ref{inverse-l2-connect}, it follows
\begin{align*}
&\left\|T_{\bX\circ B_{*}} - \frac{f\circ A\circ A_{*}^{-1}}{\|f\circ A\circ A_{*}^{-1}\|_{L^2(\mathbb{R}^2)}}\right\|_{L^2(\mathbb{R}^2)}^2\\
&\leq \frac{2\|f\|^2_{1}}{\|f\circ A\circ A_{*}^{-1}\|_{L^2(\mathbb{R}^2)}^2}3355^2C_{L}^2(C_{\mathcal{A}}+1)^4\max\{C_L^2C_{\mathcal{J}}^{-2}C_{\mathcal{A}}^6,1\}\frac{1}{d^2}\\
   & \quad + \frac{2}{\|f\circ A\circ A_{*}^{-1}\|_{L^2(\mathbb{R}^2)}^2}\int_{D_{\mathcal{A}}}12^2 C^2_\mathcal{A}C^2_L\|f\|_1^2\frac{1}{d^{2}} \, dudv\\
   &\leq\frac{2\|f\|^2_{2}}{\|f\circ A\circ A_{*}^{-1}\|_{L^2(\mathbb{R}^2)}^2}3355^2C_{L}^2(C_{\mathcal{A}}+1)^4\max\{C_L^2C_{\mathcal{J}}^{-2}C_{\mathcal{A}}^6,1\}\frac{1}{d^2}\\
   & \quad + \frac{32\|f\|_{2}^2}{\|f\circ A\circ A_{*}^{-1}\|_{L^2(\mathbb{R}^2)}^2}12^2 C^2_\mathcal{A}(C_\mathcal{A}+1)^2C^2_L\frac{1}{d^{2}}\\
   & \leq K^2\max\{C_L^4C_{\mathcal{J}}^{-4}(C_{\mathcal{A}}+1)^{12},C_L^2C_{\mathcal{J}}^{-2}(C_{\mathcal{A}}+1)^6\}\frac{1}{d^2},
\end{align*}
for a universal constant $K>0.$ Since $A_*$ was chosen arbitrarily, the assertion follows.
\end{proof}

\begin{proof}[Proof of Theorem \ref{thm.main_part1_general}]
Without loss of generality, for a test image $\mathbf{X}$, we assume its label $k$ is $0$. The analysis for $k = 1$ follows analogously due to the symmetry of $f_0$ and $f_1$.

When $k = 0$, the test image $\mathbf{X}$ is generated by the deformed template $f_0 \circ A$. Recall that each training image $\mathbf{X}_i$ is generated by $f_{k_i} \circ A_i$ with $A_i \in \mathcal{A}$ and $k_i\in\{0,1\}$.

If $k_i=0$, it follows from the triangle inequality that
\begin{align*}
\left\|T_{\bX_i\circ B_{i}^{*}} - T_{\bX\circ B_{*}}\right\|_{L^2(\mathbb{R}^2)}
&=\left\|T_{\bX_i\circ B_{i}^{*}}-\frac{f_{0}}{\|f_{0}\|_{2}}-T_{\bX\circ B_{*}}+\frac{f_{0}}{\|f_{0}\|_{2}}\right\|_{L^2(\mathbb{R}^2)}\\
&\leq\left\|T_{\bX_i\circ B_{i}^{*}}-\frac{f_{0}}{\|f_{0}\|_{2}}\right\|_{L^2(\mathbb{R}^2)}+\left\|T_{\bX\circ B_{*}}-\frac{f_{0}}{\|f_{0}\|_{2}}\right\|_{L^2(\mathbb{R}^2)}.
\end{align*}
Applying Lemma \ref{general-error} with $A_{*}=A_i$ to the first summand and $A_{*}=A$ to the second, there exist $B_{i}^{*},B_{*}\in\mathcal{A}^{-1}_{d}$ such that
\begin{align}
\left\|T_{\bX_i\circ B_{i}^{*}} - T_{\bX\circ B_{*}}\right\|_{L^2(\mathbb{R}^2)}
\leq 2K\max\{C_L^2C_{\mathcal{J}}^{-2}(C_{\mathcal{A}}+1)^{6},C_LC_{\mathcal{J}}^{-1}(C_{\mathcal{A}}+1)^3\}\frac{1}{d}.\label{eq.TX_ub_general}
\end{align}
By the definition of the separation constant $D$ in \eqref{eq.D_part2_gen}, for any $a,b>0$ and any $A_{\alpha},A_{\beta},A_{\gamma},A_{\delta}\in\mathcal{A}$,
\begin{align*}
\left\|af_{1}\circ A_{\alpha}\circ A_{\beta}^{-1}-bf_{0}\circ A_{\gamma}\circ A_{\delta}^{-1}\right\|_{L^2(\mathbb{R}^2)}&=b\left\|\frac{a}{b}f_{1}\circ A_{\alpha}\circ A_{\beta}^{-1}-f_{0}\circ A_{\gamma}\circ A_{\delta}^{-1}\right\|_{L^2(\mathbb{R}^2)}\\
&\geq b\|f_{0}\|_{L^2(\mathbb{R}^2)}D(f_{1},f_{0})\\
&=b\|f_{0}\|_{2}D(f_{1},f_{0}),
\end{align*}
and thus, by applying Lemma~\ref{inverse-l2-connect}, we obtain that for any $A,A_i,A_{*},A_{i,*}\in\mathcal{A},$
\begin{align}
&\left\|\frac{f_{1}\circ A_i\circ{A}_{i,*}^{-1}}{\|f_1\circ A_i\circ{A}_{i,*}^{-1}\|_{L^2(\mathbb{R}^2)}} - \frac{f_{0}\circ A\circ{A}_{*}^{-1}}{\|f_0\circ A\circ{A}_{*}^{-1}\|_{L^2(\mathbb{R}^2)}}\right\|_{L^2(\mathbb{R}^2)}\nonumber\\
&\geq\frac{\|f_0\|_{2}\cdot D(f_{1},f_{0})}{\|f_0\circ A\circ{A}_{*}^{-1}\|_{L^2(\mathbb{R}^2)}}\vee \frac{\|f_1\|_{2}\cdot D(f_{0},f_{1})}{\|f_1\circ A_i\circ{A}_{i,*}^{-1}\|_{L^2(\mathbb{R}^2)}}\nonumber\\
&\geq \frac{C_{\mathcal{J}}}{\sqrt{2}C_{\mathcal{A}}}\cro{D(f_{1},f_{0})\vee D(f_{0},f_{1})}\nonumber\\
&>4K\max\{C_L^2C_{\mathcal{J}}^{-2}(C_{\mathcal{A}}+1)^{6},C_LC_{\mathcal{J}}^{-1}(C_{\mathcal{A}}+1)^3\}\frac{1}{d},\label{lower-gen-con}
\end{align}
where we used the assumption $D>4\sqrt{2}K\max\{C_L^2C_{\mathcal{J}}^{-3}(C_{\mathcal{A}}+1)^{7},C_LC_{\mathcal{J}}^{-2}(C_{\mathcal{A}}+1)^4\}/d$ for the last step. For any index $i$ such that $k_i=1,$ we use the reverse triangle inequality
\begin{align*}
|a^{\prime} - b^{\prime}| \geq |a-b| - |a^{\prime} - a| - |b^{\prime} - b|, \quad \text{for all} \ a,b,a^{\prime}, b^{\prime} \in \mathbb{R}.
\end{align*}
For any $B_{i},B\in\mathcal{A}_{d}^{-1}\subseteq\mathcal{A}^{-1}$, we have $B_{i}^{-1},B^{-1}\in\mathcal{A}$ and thus
\begin{align*}
&\left\|T_{\bX_{i}\circ B_{i}} - T_{\bX\circ B}\right\|_{L^2(\mathbb{R}^2)} \\
&\geq\left\|\frac{f_{1}\circ A_i\circ{B}_{i}}{\|f_1\circ A_i\circ{B}_{i}\|_{L^2(\mathbb{R}^2)}} - \frac{f_{0}\circ A\circ{B}}{\|f_0\circ A\circ{B}\|_{L^2(\mathbb{R}^2)}}\right\|_{L^2(\mathbb{R}^2)}\\
&\quad- \left\|T_{\bX_{i}\circ B_{i}} - \frac{f_{1}\circ A_{i}\circ{B}_{i}}{\|f_1\circ A_i\circ{B}_{i}\|_{L^2(\mathbb{R}^2)}}\right\|_{L^2(\mathbb{R}^2)} - \left\|T_{\bX\circ B} - \frac{f_{0}\circ A\circ{B}}{\|f_0\circ A\circ{B}\|_{L^2(\mathbb{R}^2)}}\right\|_{L^2(\mathbb{R}^2)}\\
&>4K\max\{C_L^2C_{\mathcal{J}}^{-2}(C_{\mathcal{A}}+1)^{6},C_LC_{\mathcal{J}}^{-1}(C_{\mathcal{A}}+1)^3\}\frac{1}{d}- 2K\max\{C_L^2C_{\mathcal{J}}^{-2}(C_{\mathcal{A}}+1)^{6},C_LC_{\mathcal{J}}^{-1}(C_{\mathcal{A}}+1)^3\}\frac{1}{d}\\
&\geq2K\max\{C_L^2C_{\mathcal{J}}^{-2}(C_{\mathcal{A}}+1)^{6},C_LC_{\mathcal{J}}^{-1}(C_{\mathcal{A}}+1)^3\}\frac{1}{d},
\end{align*}
where the second-to-last inequality follows from \eqref{lower-gen-con} and Lemma \ref{general-error}.

Combining this with \eqref{eq.TX_ub_general}, we conclude that
\begin{align*}
\wh i \in \argmin_{i\in\{1,\dots,n\}} \ \min_{B_{i},B\in\mathcal{A}^{-1}_{d}} \, \big\|T_{\bX_{i}\circ{B_{i}}}-T_{\bX\circ B}\big\|_{L^2(\mathbb{R}^2)}
\end{align*}
holds for some $i$ with $k_i=0$ implying $\wh{k} =0.$ Since $k=0$, this shows the assertion $\wh k=k$.
\end{proof}

\begin{proof}[Proof of Lemma \ref{example-lemma}]
Since $\mathcal{A}$ is a group, it follows that for any $A\in\mathcal{A}$, the inverse $A^{-1}$ lies in $\mathcal{A}$, and for any $A_1,A_2\in\mathcal{A}$, the composition $A_1\circ A_2$ lies in $\mathcal{A}$. Therefore,
\begin{align*}
D(f,g)&=\frac{\inf_{a\in\mathbb{R},\{A_i\}_{i=1}^{4}\subseteq\mathcal{A}}\|af\circ A_1\circ A_2^{-1}-g\circ A_3\circ A_4^{-1}\|_{L^2(\mathbb{R}^2)}}{\|g\|_{L^2(\mathbb{R}^2)}}\\
&=\frac{\inf_{a\in\mathbb{R},A,A'\in\mathcal{A}}\|af\circ A-g\circ A'\|_{L^2(\mathbb{R}^2)}}{\|g\|_{L^2(\mathbb{R}^2)}}.
\end{align*}
For any $A,A'\in\mathcal{A}$, there exists an $\tilde A=A\circ(A')^{-1}\in\mathcal{A}$ such that for any $a\in\mathbb{R}$,
\begin{align}
\left\|a f_0\circ A-f_1\circ A'\right\|_{L^2(\mathbb{R}^2)}^2&=\int_{\mathbb{R}^2}\left(af_0\circ A(u,v)-f_1\circ A'(u,v)\right)^2dudv\nonumber\\
&\geq\frac{1}{2C^2_{\mathcal{A}}}\int_{\mathbb{R}^2}\left(af_0\circ \tilde A(x,y)-f_1(x,y)\right)^2dxdy\nonumber\\
&\geq\frac{\|f_1\|_{L^2((\supp(f_0\circ A))^c)}^{2}}{2C^2_{\mathcal{A}}},\label{d-con-lemma}
\end{align}
where the first inequality follows from the change of variables together with Assumption~\ref{ass1}-(i). Since $D=D(f_0,f_1)\vee D(f_1,f_0)\geq D(f_0,f_1)$, \eqref{d-con-lemma} further implies that $D>C(C_{L},C_{\mathcal{A}},C_{\mathcal{J}})/d$ whenever $$d>\sqrt{2}C(C_{L},C_{\mathcal{A}},C_{\mathcal{J}})C_{\mathcal{A}} \sup_{A\in \mathcal{A}}\frac{\|f_1\|_{L^2(\mathbb{R}^2)}}{\|f_1\|_{L^2((\supp(f_0\circ A))^c)}}.$$
\end{proof}

\begin{lemma}\label{affine-verify}
Let $\mathcal{A}$ be the class of affine transformations on $\mathbb{R}^2$, where each $A \in \mathcal{A}$ is defined as in \eqref{affinetran} with parameters $(b_1, \dots, b_4,\tau,\tau')$ satisfying $|b_1|, \dots, |b_4|\leq C_{\mathcal{A}},$ $|\tau|,|\tau'|\leq\ell_s,$ and $|b_1b_4-b_2b_3|\geq \beta$, for positive constants $\beta,\ell_s,C_{\mathcal{A}}$. If we further constrain the class so that only $p\leq 6$ of the six parameters ${b_1, \ldots, b_4, \tau, \tau'}$ vary while the rest remain constant, then there exists a $1/d$-covering $\mathcal{A}_d^{-1}$ of $\mathcal{A}^{-1}$ with cardinality $|\mathcal{A}_d^{-1}|\asymp d^{p}$.
\end{lemma}

\begin{proof}
Fix a deformation $A\in \mathcal{A}$ with parameters $b_{1},\ldots,b_{4},\tau,\tau'$ and define
\begin{align*}
{\mathbf{B}}=\begin{pmatrix}
b_1 & b_2\\
b_3 & b_4
\end{pmatrix}.
\end{align*}   
If any of the parameters $b_i$, $\tau$, or $\tau'$ is fixed, we take $\widetilde b_i =b_i$, $\widetilde\tau=\tau$ and $\widetilde\tau'=\tau'$. Otherwise, we consider the perturbed parameters $\widetilde{b}_1, \dots, \widetilde{b}_4, \widetilde\tau, \widetilde\tau'$ such that 
\begin{align*}
\max_{i=1,\ldots,4} |b_i - \widetilde{b}_i| \leq \frac{C_b}{d}, \quad |\tau - \widetilde\tau| \vee |\tau' - \widetilde\tau'| \leq \frac{C_s}{d},
\end{align*}
and define the perturbed matrices
\begin{align*}
\widetilde {\mathbf{B}}=\begin{pmatrix}
\widetilde{b}_1 & \widetilde{b}_2\\
\widetilde{b}_3 & \widetilde{b}_4
\end{pmatrix}
\quad \text{and} \quad
\widetilde {\mathbf{B}}^{-1}=\frac{1}{\det(\widetilde {\mathbf{B}})} \begin{pmatrix}
        \widetilde{b}_4 & -\widetilde{b}_2\\
        -\widetilde{b}_3 & \widetilde{b}_1 
    \end{pmatrix}.
\end{align*}
Denote $\widetilde{\bs{\tau}}=(\widetilde\tau,\widetilde\tau')^{\top}$. The inverse of the perturbed affine transformation is given by
\begin{equation*}
    \widetilde{A}^{-1}(u,v)=(\widetilde a_{1}^{-1}(u,v),\widetilde a_{2}^{-1}(u,v))= \widetilde {\mathbf{B}}^{-1} \big((u, v)^\top + \widetilde{\bs{\tau}}\big).
\end{equation*}
Recall that $D_{\mathcal{A}}=[-2C_{\mathcal{A}}-1,2C_{\mathcal{A}}+1]^2$. For any $(u,v)\in D_{\mathcal{A}}$,
\begin{align*}
 &|a_{1}^{-1}(u,v)-\widetilde{a}_{1}^{-1}(u, v)|\\
 =&\left|\frac{1}{\det(\mathbf{B})}\cro{b_{4}(u+\tau)-b_{2}(v+\tau')}-\frac{1}{\det(\widetilde{\mathbf{B}})}\cro{\widetilde b_{4}(u+\widetilde \tau)-\widetilde b_{2}(v+\widetilde \tau')}\right|\\
 \leq&\frac{1}{|\det(\mathbf{B})|}\left|b_{4}(u+\tau)-b_{2}(v+\tau')-\widetilde b_{4}(u+\widetilde\tau)+\widetilde b_{2}(v+\widetilde \tau')\right|+\left|\frac{1}{\det(\mathbf{B})}-\frac{1}{\det(\widetilde {\mathbf{B}})}\right|\cdot\left|\widetilde b_{4}(u+\widetilde\tau)-\widetilde b_{2}(v+\widetilde\tau')\right|\\
 \leq&\frac{2}{\beta}\cro{\left(2C_{\mathcal{A}}+\ell_{s}+1\right)\frac{C_{b}}{d}+(C_b+C_{\mathcal{A}})\frac{C_{s}}{d}}+2\left|\frac{1}{\det(\mathbf{B})}-\frac{1}{\det(\widetilde{\mathbf{B}})}\right|(C_{b}+C_{\mathcal{A}})(2C_{\mathcal{A}}+C_{s}+\ell_{s}+1).
\end{align*}
Since 
\begin{align*}
\big|\det(\widetilde{\mathbf{B}})-\det(\mathbf{B})\big|&=\Big|\left(\widetilde b_{1}\widetilde b_{4}-\widetilde b_{2}\widetilde b_{3}\right)-\left(b_{1}b_{4}-b_{2}b_{3}\right)\Big|\leq\frac{2C_b(2C_{\mathcal{A}}+C_b)}{d},
\end{align*}
and for sufficiently large $d$ such that $2C_b(2C_{\mathcal{A}}+C_b)/d \leq \beta/2$,
it follows that $$\big|\det(\widetilde{\mathbf{B}} )\big| \geq \big|\det(\mathbf{B} )\big| - \frac{2C_b(2C_{\mathcal{A}}+C_b)}{d} \geq \beta - \frac{\beta}{2}=\frac{\beta}{2}.$$ Therefore, we can bound
\begin{align*}
    \left|\frac{1}{\det(\mathbf{B})}-\frac{1}{\det(\widetilde{\mathbf{B}} )}\right|&\leq \frac{2}{\beta^2}\left|\det(\mathbf{B})-\det(\widetilde{\mathbf{B}} )\right|\leq \frac{4}{\beta^2}\left(2C_{\mathcal{A}}+C_{b}\right)\frac{C_{b}}{d}.
\end{align*}
This shows that for sufficiently large $d$ and any $(u,v)\in D_{\mathcal{A}}$,
\begin{equation}\label{inverse-1-bound-affine}
\big|a_{1}^{-1}(u,v)-\widetilde{a}_{1}^{-1}(u, v)\big|\leq\frac{C(\beta,\ell_{s},C_{\mathcal{A}},C_{b},C_{s})}{d},
\end{equation}
where $C(\beta,\ell_{s},C_{\mathcal{A}},C_{b},C_{s})$ is a constant depending only on $\beta,\ell_{s},C_{\mathcal{A}},C_{b},C_{s}$. For any $(u,v)\in D_{\mathcal{A}}$, one can similarly derive 
$$\big|a_{2}^{-1}(u,v)-\widetilde{a}_{2}^{-1}(u, v)\big|\leq\frac{C(\beta,\ell_{s},C_{\mathcal{A}},C_{b},C_{s})}{d},$$
which together with \eqref{inverse-1-bound-affine} gives
$$\|A^{-1}-\widetilde A^{-1}\|_{L^{\infty}(D_{\mathcal{A}})}\leq\frac{C(\beta,\ell_{s},C_{\mathcal{A}},C_{b},C_{s})}{d}.$$ With suitable chosen constants $C_{b},C_{s}$, condition \eqref{gen-conver} holds and $|\mathcal{A}_{d}^{-1}|\asymp d^p$.
\end{proof}

\begin{proof}[Proof of Lemma \ref{ex_scale}]
The first claim is a special case of Lemma \ref{lem.rot} when $\gamma=0$. The bound for the covering number follows from Lemma \ref{affine-verify}.  
\end{proof}

\begin{proof}[Proof of Lemma \ref{lem.rot}]
The inverse of the rotation matrix   
\begin{align*}
    \mathbf{D}_{\gamma} := \begin{pmatrix}
    \cos\gamma & -\sin\gamma\\
    \sin\gamma & \cos\gamma
    \end{pmatrix}
\end{align*}
is $\mathbf{D}_{-\gamma}.$ The conditions on the parameter now ensure that $[-(-\xi)_+,\xi_+]\times [-(-\xi')_+,\xi'_+]\supseteq\mathbf{D}_{-\gamma}([1/4,3/4]^2+(\tau,\tau')^\top).$ To see this, it is enough to check the four vertices of $[1/4,3/4]^2.$ 

Now we examine the partial differentiability condition in Assumption~\ref{ass1}. Consider any transformation $A=(a_1,a_2)\in\mathcal{A}$ with parameters $b_1,b_2,b_3,b_4,\tau,\tau'$. Observe that $a_1(u,v)=b_1 u+b_2 v-\tau$ and $a_2(u,v)=b_3 u+b_4 v-\tau'$ are continuously differentiable with $|\partial_u a_1(u,v)|\leq |b_1|,$ $|\partial_v a_1(u,v)|\leq |b_2|,$ $|\partial_u a_2(u,v)|\leq |b_3|,$ and $|\partial_v a_2(u,v)|\leq |b_4|.$ Since $|b_1|\ldots,|b_4|\leq C_{\mathcal{A}}$, Assumption \ref{ass1}-(i) is satisfied with $C_{\mathcal{A}}$. Moreover, for any $A\in\mathcal{A}$, $|\det(J_{A}(u,v))|=|\xi\xi'|\geq1/4$, for all $(u,v)\in\mathbb{R}^2$, under the given condition that $|\xi|,|\xi'|\geq1/2.$ This implies that $C_{\mathcal{J}}=1/2$ in Assumption~\ref{ass1}-(ii).

The bound for the covering number can be derived similarly as the one in Lemma \ref{affine-verify}, based on the perturbation of the five parameters $\gamma,\xi,\xi',\tau,\tau'$. This yields $|\mathcal{A}_{d}^{-1}|\asymp d^5$.
\end{proof}

\begin{proof}[Proof of Lemma~\ref{wave-verify}]
For $\lambda\not=0$, $A$ is invertible and $$A^{-1}(u,v)=(a_1^{-1}(u,v),a_2^{-1}(u,v))=(u-\alpha\sin\left(2\pi v/\lambda\right),v).$$ Moreover, if $|\alpha|\leq1/4$ and $\lambda\not=0$, then for any point $(u,v)\in[1/4,3/4] \times [1/4,3/4]\subseteq\cro{0,1}^2$,
$$u-\alpha\sin\left(2\pi v/\lambda\right)\leq u+|\alpha|\leq1\quad\mbox{and}\quad u-\alpha\sin\left(2\pi v/\lambda\right)\geq u-|\alpha|\geq0,$$ which implies the full visibility condition in Assumption~\ref{ass1}-(i). 

Consider a fixed $A\in\mathcal{A}$ associated with the parameters $\alpha$ and $\lambda$. Then, the partial derivatives of $a_{2}$ are bounded in the supremum norm by $1.$ Moreover, under the condition $|\alpha|\leq1/4$ and $|\lambda|\geq C_{\text{lower}}>0$, the function $a_1(\cdot,\cdot)$ is continuously partially differentiable. Furthermore, at any $(u,v)\in \mathbb{R}^2,$ we have
\begin{align*}
\left|\partial_u a_1(u,v)\right|\leq 1,\quad \left|\partial_v a_1(u,v)\right|=\left|\alpha\frac{2\pi}{\lambda}\cos\left(\frac{2\pi v}{\lambda}\right)\right|\leq\frac{\pi}{2C_{\text{lower}}}, 
\end{align*}
which implies that Assumption~\ref{ass1}-(i) holds with $C_{\mathcal{A}}=\max\{\pi/(2C_{\text{lower}}),1\}$. Under the given conditions, for any $A$, $|\det(J_A(u,v))|\geq1,$ for all $(u,v)\in\mathbb{R}^2,$ implying $C_{\mathcal{J}}=1$ in Assumption~\ref{ass1}-(ii).

For any $A=(a_1,a_2)\in\mathcal{A}$, let $\lambda_{*}$ and $\alpha_{*}$ be the true parameters, satisfying $|\lambda_{*}|\geq C_{\text{lower}}$ and $|\alpha_{*}|\leq1/4$. Taking $\widetilde\alpha$ and $\widetilde\lambda$ such that 
\begin{equation*}
\left|\widetilde\alpha-\alpha_{*}\right|\leq\frac{C_{\alpha}}{d}\quad\mbox{and}\quad|\widetilde\alpha|\leq|\alpha_{*}|,
\end{equation*}
\begin{equation*}
\left|\widetilde\lambda-\lambda_{*}\right|\leq\frac{C_{\lambda}}{d}\quad\mbox{and}\quad|\widetilde\lambda|\geq|\lambda_{*}|,
\end{equation*}
for some constants $C_{\alpha},C_{\lambda}>0$, we can derive for $A^{-1}(u,v)=(\widetilde a_{1}^{-1}(u,v),\widetilde a_{2}^{-1}(u,v))=(u-\widetilde \alpha\sin(2\pi v/\widetilde\lambda),v)$ with $(u,v)\in D_{\mathcal{A}}=[-2C_{\mathcal{A}}-1,2C_{\mathcal{A}}+1]^2$,
\begin{align*}
\left|\widetilde a_{1}^{-1}(u,v)-a_{1}^{-1}(u,v)\right|&=\left|u-\widetilde \alpha\sin\left(\frac{2\pi v}{\widetilde\lambda}\right)-u+\alpha_{*}\sin\left(\frac{2\pi v}{\lambda_{*}}\right)\right|\\
&\leq\left|\alpha_{*}\sin\left(\frac{2\pi v}{\lambda_{*}}\right)-\alpha_{*}\sin\left(\frac{2\pi v}{\widetilde\lambda}\right)\right|+\left|\alpha_{*}\sin\left(\frac{2\pi v}{\widetilde\lambda}\right)-\widetilde\alpha\sin\left(\frac{2\pi v}{\widetilde\lambda}\right)\right|\\
&\leq|\alpha_{*}|\left|\sin\left(\frac{2\pi v}{\lambda_{*}}\right)-\sin\left(\frac{2\pi v}{\widetilde\lambda}\right)\right|+|\alpha_{*}-\widetilde\alpha|\\
&\leq2\pi (2C_{\mathcal{A}}+1)|\alpha_{*}|\left|\frac{1}{\lambda_{*}}-\frac{1}{\widetilde\lambda}\right|+|\alpha_{*}-\widetilde\alpha|\\
&\leq\frac{C_{\mathcal{A}}C_{\lambda}(2C_{\mathcal{A}}+1)}{C_{\text{lower}}\cdot d}+\frac{C_{\alpha}}{d}\\
&=\frac{C(C_{\text{lower}},C_{\lambda},C_{\alpha})}{d}.
\end{align*}
This implies that $\mathcal{A}_{d}^{-1}$ can be constructed by discretizing the parameters $\lambda$ and $\alpha$, namely $|\mathcal{A}_{d}^{-1}|\asymp d^2$. 
\end{proof}

\begin{proof}[Proof of Lemma~\ref{composite-lemma}]
For any $A\in\mathcal{A}_{2}\circ\mathcal{A}_{1}$, according to the definition of $\mathcal{A}_{2}\circ\mathcal{A}_{1}$, there exist $A_{1}\in\mathcal{A}_1$ and $A_{2}\in\mathcal{A}_2$ such that $A=A_{2}\circ A_{1}$. The class $\mathcal{A}_{2}\circ\mathcal{A}_{1}$ contains the identity, as both $\mathcal{A}_{2}$ and $\mathcal{A}_{1}$ include the identity.

We now verify fully visibility, namely, that $[\beta_{\text{left}}, \beta_{\text{right}}] \times [\beta_{\text{down}},\beta_{\text{up}}]\subseteq A([0,1]^2)$ for any $A\in\mathcal{A}_{2}\circ\mathcal{A}_{1}$. Since $\mathcal{A}_2$ satisfies the full visibility condition, for any $A_2\in\mathcal{A}_2$, we have $[\beta_{\text{left}}, \beta_{\text{right}}] \times [\beta_{\text{down}},\beta_{\text{up}}]\subseteq A_2([0,1]^2)$. Given the condition $[0,1]^2\subseteq A_1([0,1]^2)$, it follows that $$[\beta_{\text{left}}, \beta_{\text{right}}] \times [\beta_{\text{down}},\beta_{\text{up}}]\subseteq A_2([0,1]^2)\subseteq A_2\left(A_1([0,1]^2)\right)=A([0,1]^2).$$

For any real numbers $u,v$, denote  $A_{1}(u,v)=(a_1(u,v),b_1(u,v))$ and $A_{2}(u,v)=(a_2(u,v),b_2(u,v))$. Consequently, by writing $A(u,v)=(a(u,v),b(u,v))$, we have $a(u,v)=a_2(a_1(u,v),b_1(u,v))$ and $b(u,v)=b_2(a_1(u,v),b_1(u,v))$. If both $\mathcal{A}_{1}$ and $\mathcal{A}_{2}$ satisfy Assumption~\ref{ass1}, we can differentiate the composite function at any $(u,v)\in\mathbb{R}^2$ by the chain rule and derive
\begin{align*}
\left|\partial_u a(u,v)\right|&=\left|\frac{\partial a_2}{\partial x} \Big|_{(x,y) = (a_1, b_1)} \cdot \partial_u a_1(u,v) 
+ \frac{\partial a_2}{\partial y} \Big|_{(x,y) = (a_1, b_1)} \cdot \partial_u b_1(u,v)\right|\leq2C_{\mathcal{A}_{1}}C_{\mathcal{A}_{2}}.
\end{align*}
Similarly, we can show 
$\left|\partial_v a(u,v)\right|\leq2C_{\mathcal{A}_{1}}C_{\mathcal{A}_{2}},$ $\left|\partial_u b(u,v)\right|\leq2C_{\mathcal{A}_{1}}C_{\mathcal{A}_{2}},$ and $\left|\partial_v b(u,v)\right|\leq2C_{\mathcal{A}_{1}}C_{\mathcal{A}_{2}}.$ Moreover, for all $(u,v)\in\mathbb{R}^2$,
$$\left|\det(J_A(u,v))\right| = \left|\det( J_{A_2}(A_1(u,v)))\right| \cdot \left|\det (J_{A_1}(u,v))\right|\geq C^2_{\mathcal{J}_1}C^2_{\mathcal{J}_2},$$
which completes the proof.
\end{proof}

\subsection{Proofs for classification via image alignment}\label{proofs-image-align}
We will frequently use the following notation. For any given function $f:\mathbb{R}^2\rightarrow\mathbb{R}$, denote 
$$\alpha_{f}^{-}:=\sup\big\{u:\forall\ t\leq u, v\in\mathbb{R},\ f(t,v)=0\big\}, \quad \alpha_{f}^{+}:=\inf\big\{u:\forall\ t\geq u, v\in\mathbb{R},\ f(t,v)=0\big\},$$
    $$\beta_{f}^{-}:=\sup\big\{v:\forall\ t\leq v, u\in\mathbb{R},\ f(u,t)=0\big\}, \quad 
    \beta_{f}^{+}:=\inf\big\{v:\forall\ t\geq v, u\in\mathbb{R},\ f(u,t)=0\big\}.
$$
The rectangular support of function $f$ is then given by $[\alpha_{f}^{-},\alpha_{f}^{+}]\times [\beta_{f}^{-},\beta_{f}^{+}].$
\begin{lemma}\label{lem.support_error}
Let $j_{\pm},\ell_{\pm}$ be as defined in \eqref{eq.j_pm_def} and \eqref{eq.l_pm_def}. If $f$ is a continuous function, then for any $\xi,\xi'>0$ and $\tau,\tau'\in\mathbb{R}$,
\begin{align*}
    \alpha^{-}_{f}<\xi\frac{j_-}{d}-\tau \leq \alpha^{-}_{f}+ \frac{\xi}{d}, \quad  \alpha^{+}_{f}\leq\xi\frac{j_+}{d}-\tau< \alpha^{+}_{f} + \frac{\xi}{d}
\end{align*}
and
\begin{align*}
    \beta^{-}_{f}<\xi'\frac{\ell_{-}}{d}-\tau'\leq \beta^-_{f} + \frac{\xi'}{d}, \quad   \beta^+_{f}\leq\xi'\frac{\ell_+}{d}-\tau' <\beta^+_{f}+\frac{\xi'}{d}.
\end{align*}
\end{lemma}
\begin{proof}
We only prove the inequalities $\alpha^{-}_{f}<\xi j_-/d-\tau\leq \alpha^{-}_{f}+\xi/d$. All the remaining inequalities will follow using the same arguments. 

Fix $\tau,\tau',\xi,\xi'$ and set $\alpha^-:=\alpha^-_{f(\xi\cdot-\tau,\xi'\cdot-\tau')}.$ First, we show that 
\begin{equation}
\alpha^{-}<\frac{j_-}{d} \leq \alpha^- + \frac{1}{d}.\label{discrete-dis}
\end{equation}
Based on the definitions of $j_{-}$ and $\alpha^{-}$, we observe that $j_{-}/d>\alpha^{-}$. We now assume that $j_{-}/d>\alpha^{-}+1/d$. Let $(\alpha^{-}, v(\alpha^{-}))$ be one of the points located on the boundary of the support of $f(\xi\cdot-\tau,\xi'\cdot-\tau')$. Due to the continuity of $f$ and the assumption that $(j_{-}-1)/d>\alpha^{-}$, there exists a $j_{0}\leq j_{-}$ and a small neighborhood $U_{\alpha^{-}}$ of the point $(\alpha^{-}, v(\alpha^{-}))$ satisfying $U_{\alpha^{-}}\subseteq[(j_{0}-2)/d,(j_{0}-1)/d)\times[(\ell-1)/d,\ell/d)$ for some $\ell$ such that 
\begin{align}
X_{j_{0}-1,\ell}\geq\eta\int_{U_{\alpha^{-}}}d^{2}f(\xi u-\tau,\xi'v-\tau') \, dudv>0.
\end{align}
This contradicts that by definition $j_{-}$ is the smallest integer $j$ satisfying $X_{j,\ell}>0$, proving \eqref{discrete-dis}.

The definition of $\alpha^{-}$, and $\xi>0$ yield
\begin{equation*}
\xi\alpha^{-}-\tau=\alpha^{-}_{f},
\end{equation*}
which together with \eqref{discrete-dis} completes the proof.
\end{proof}

Set
\begin{equation*}
\Delta_{f}:=\alpha^+_{f}-\alpha^-_{f}\quad\mbox{and}\quad\Delta'_{f}:=\beta^+_{f}-\beta^-_{f},
\end{equation*}
for the width and the height of the rectangular support of $f$.
\begin{prop}\label{rotate-rescale-L2} 
Let $f \colon \mathbb{R}^2 \to \mathbb{R}$ be a measurable function with $\operatorname{supp}(f) \subseteq [0,1]^2$. If for some $\Delta, \Delta' \neq 0$ and $\delta, \delta' \in \mathbb{R}$, the rescaled and translated function $f(\Delta \cdot + \delta, \Delta' \cdot + \delta')$ also satisfies $\operatorname{supp}(f(\Delta \cdot + \delta, \Delta' \cdot + \delta')) \subseteq [0,1]^2$, then for $p = 1, 2$,
\[
\|f(\Delta \cdot + \delta, \Delta' \cdot + \delta')\|_p^p = \frac{\|f\|_p^p}{|\Delta \Delta'|}.
\]
\end{prop}
\begin{proof}
This follows by a change of variables,
\begin{align*}
\|f(\Delta\cdot+\delta,\Delta'\cdot+\delta')\|_{p}^{p}&=\int_{\cro{0,1}^2}|f(\Delta u+\delta,\Delta' v+\delta')|^{p} \, dudv\\
&=\int_{\mathbb{R}^2}|f(\Delta u+\delta,\Delta' v+\delta')|^{p}  \, 
 dudv\\
&=\frac{1}{|\Delta\Delta'|}\int_{\mathbb{R}^2}|f(x,y)|^{p} \, dxdy\\
&=\frac{1}{|\Delta\Delta'|}\int_{\cro{0,1}^2}|f(x,y)|^{p} \, dxdy\\
&=\frac{\|f\|_{p}^{p}}{|\Delta\Delta'|}.
\end{align*}
\end{proof}

\begin{lemma}
\label{le1}
Consider a generic image of the form \eqref{eq.X_f_model}. Assume that the support of $f$ is contained in $\cro{1/4,3/4}^{2}$, and satisfies the Lipschitz property \eqref{eq.Lip_cond-gen} for some constant $C_{L}$. Let $T_{\bX}$ be as defined in \eqref{eq.T_def} and $h$ be the function $(t,t^{\prime}) \mapsto h(t,t^{\prime}) := f(\Delta_{f}t+\alpha^-_{f},\Delta'_{f}t'+\beta^-_{f}).$ Then, there exists a universal constant $K>0$, such that 
\begin{align*}
    \left\|T_{\bX}- \frac{\sqrt{\Delta_{f} \Delta^{\prime}_{f}}h}{\|f\|_2}\right\|_2 \leq K(C_L\vee C_L^2 ) (\xi\vee\xi^{\prime}\vee1)^2 \frac{1}{d}.
\end{align*}
\end{lemma}

\begin{proof}
In a first step of the proof, we show that
\begin{align}
\label{le1eq1}
    \Big|Z_{\bX}(t,t^{\prime}) - \eta h(t,t')\Big| \leq 10\eta C_{L}\|f\|_1 (\xi\vee\xi^{\prime})\frac{1}{d}, \quad \text{for all} \ t, t^{\prime} \in [0,1],
\end{align}
where $Z_{\bX}(t,t^{\prime})$ is as defined in \eqref{eq.Z_def}.

Fix $t,t^{\prime}\in [0,1]$ and recall that $j_{-}, j_{+}, \ell_{-}$, and $\ell_{+}$ are defined as in \eqref{eq.j_pm_def} and \eqref{eq.l_pm_def}. Define $j_*:=\lfloor j_- + t(j_+ - j_-)\rfloor$ and $\ell_*:=\lfloor \ell_- + t'(\ell_+ - \ell_-)\rfloor$, where the dependence of $j_*$ on $j_{-}, j_{+}$ and $\ell_*$ on $\ell_{-}, \ell_{+}$ has been suppressed. For any $t, t^{\prime} \in [0,1]$,
\begin{align}
\left|Z_{\bX}(t,t^{\prime}) - \eta h(t,t^{\prime})\right|
\leq&\left|X_{j_*,\ell_*} - \eta h(t, t^{\prime})\right|\nonumber\\
\leq&\eta\left|\int_{I_{j_*,\ell_*}}d^{2}f(\xi u-\tau,\xi'v-\tau') \, dudv-f(\Delta_{f}t+\alpha^-_{f},\Delta'_{f}t'+\beta^-_{f})\right|.\label{inequ-1}
\end{align}
Since $f$ satisfies the Lipschitz condition \eqref{eq.Lip_cond-gen}, we can further bound this by noting that
\begin{align}
&\left|f(\xi u-\tau,\xi'v-\tau')-f(\Delta_{f}t+\alpha^-_{f},\Delta'_{f}t'+\beta^-_{f})\right|\nonumber\\
\leq&C_{L}\|f\|_{1}\cro{\left|(\xi u-\tau)-\left(\Delta_{f} t+\alpha^{-}_{f}\right)\right|+\left|(\xi'v-\tau')-\left(\Delta'_{f} t'+\beta^{-}_{f}\right)\right|}
.\label{part-Ia-bound-1}
\end{align}
For any $u\in[(j_*-1)/d,j_*/d)=[(\lfloor j_- + t(j_+ - j_-)\rfloor-1)/d,\lfloor j_- + t(j_+ - j_-)\rfloor/d)$, we have $$\frac{j_- + t(j_+ - j_-)-2}{d}\leq\frac{\lfloor j_- + t(j_+ - j_-)\rfloor-1}{d}\leq u<\frac{\lfloor j_- + t(j_+ - j_-)\rfloor}{d}\leq\frac{j_- + t(j_+ - j_-)}{d},$$ hence, together with Lemma~\ref{lem.support_error}, we obtain
\begin{align}
\left|\left(\xi u-\tau\right)-\left(\Delta_{f} t+\alpha^{-}_{f}\right)\right|&\leq\left|\left(\xi\frac{j_- + t(j_+ - j_-)}{d}-\tau\right)-\left(\Delta_{f} t+\alpha^{-}_{f}\right)\right|+\frac{2\xi}{d}\nonumber\\
&\leq\left|\left(\xi\frac{j_{-}}{d}-\tau\right)-\alpha^{-}_{f}\right|+\left|\xi\frac{j_+ - j_-}{d}-\Delta_{f}\right|+\frac{2\xi}{d}\nonumber\\
&\leq2\left|\left(\xi\frac{j_{-}}{d}-\tau\right)-\alpha^{-}_{f}\right|+\left|\left(\xi\frac{j_{+}}{d}-\tau\right)-\alpha^{+}_{f}\right|+\frac{2\xi}{d}\nonumber\\
&\leq\frac{5\xi}{d}.\label{part-Ia-boundary-1}
\end{align}
Similarly, for any $v\in[(\ell_*-1)/d,\ell_*/d)$,
\begin{equation}\label{part-Ia-boundary-2}
\left|\left(\xi'v-\tau'\right)-\left(\Delta'_{f} t'+\beta^{-}_{f}\right)\right|\leq\frac{5\xi'}{d}.
\end{equation}
Plugging (\ref{part-Ia-boundary-1}) and (\ref{part-Ia-boundary-2}) into (\ref{part-Ia-bound-1}) yields that for any $(u,v)\in I_{j_*,\ell_*}$,
\begin{align*}
\left|f(\xi u-\tau,\xi'v-\tau')-f(\Delta_{f}t+\alpha^-_{f},\Delta'_{f}t'+\beta^-_{f})\right|
\leq10C_{L}\|f\|_{1}(\xi\vee\xi')\frac{1}{d}.
\end{align*}
Combined with \eqref{inequ-1} and the fact that $t,t'\in [0,1]$ was arbitrary, this implies \eqref{le1eq1}.

In the next step, we show that for some universal constant $C_{1}>0,$
\begin{align}
\label{le1eq2}
    \left|\|Z_{\bX}\|_2 - \frac{\|f\|_2 \eta}{\sqrt{\Delta_{f} \Delta^{\prime}_{f}}}\right| \leq C_{1}\eta(C_{L}\vee C_L^{2}) (\xi\vee\xi^{\prime}\vee1)^2\frac{\|f\|_1}{\sqrt{\Delta_{f} \Delta^{\prime}_{f}}} \frac{1}{d}.
\end{align}
Using that for real numbers $a,b\not=0,$ $a-b=(a^2-b^2)/(a+b),$ we can rewrite
\begin{align}
\label{le2eq1}
\left|\|Z_{\bX}\|_2 -  \frac{\|f\|_2 \eta}{\sqrt{\Delta_{f} \Delta^{\prime}_{f}}}\right|&= \left|\|Z_{\bX}\|_2^2- \frac{\|f\|_2^2 \eta^2}{\Delta_{f} \Delta^{\prime}_{f}}  \right| \frac{1}{\|Z_{\bX}\|_2+\frac{\|f\|_2 \eta}{\sqrt{\Delta_{f} \Delta^{\prime}_{f}}}}\nonumber\\ &\leq \left|\|Z_{\bX}\|_2^2 - \frac{\|f\|_2^2 \eta^2}{\Delta_{f} \Delta^{\prime}_{f}} \right| \frac{\sqrt{{\Delta_{f} \Delta^{\prime}_{f}} }}{\|f\|_2 \eta}.
\end{align}
Since $h(t,t^{\prime}) = f(\Delta_{f}t+\alpha^-_{f},\Delta'_{f}t'+\beta^-_{f})$ and the support of $h$ is contained in $[0,1]^{2},$ we have, according to Proposition~\ref{rotate-rescale-L2}, for $p=1,2$, $\|h\|_p^p=\|f\|_p^p/(\Delta_{f} \Delta'_{f}).$ Also, employing \eqref{le1eq1}, we bound the first term on the right-hand side of \eqref{le2eq1} by
\begin{align*}
    &\left|\|Z_{\bX}\|_2^2- \frac{\|f\|_2^2 \eta^2}{\Delta_{f} \Delta^{\prime}_{f}}\right|= \left|\int_0^1 \int_0^1 \left(Z_{\bX}(t,t') - \eta h(t, t^{\prime})+ \eta h(t,t^{\prime})\right)^2dtdt^{\prime}
    - \frac{\|f\|_2^2 \eta^2}{\Delta_{f} \Delta^{\prime}_{f}} \right|\\
    & = \left| \int_0^1 \int_0^1 \left(Z_{\bX}(t,t') - \eta h(t,t^{\prime})\right)^2 dtdt^{\prime}\right.
    + 2 \eta\int_0^1 \int_0^1 \left(Z_{\bX}(t,t') - \eta h(t,t^{\prime}) \right)h(t, t^{\prime}) dtdt^{\prime}
    \\
    & \left. \quad 
    +
    \int_0^1 \int_0^1 \eta^2 h^2(t,t^{\prime})dtdt^{\prime} - \frac{\|f\|_2^2 \eta^2}{\Delta_{f} \Delta^{\prime}_{f}} \right|\\
    & \leq \int_0^1 \int_0^1 \left|Z_{\bX}(t,t') - \eta h(t,t^{\prime})\right|^2 dtdt^{\prime} + 2 \eta \int_0^1 \int_0^1 \left|Z_{\bX}(t,t') - \eta h(t, t^{\prime}) \right|\left|h(t,t^{\prime})\right|dtdt^{\prime}\\
    & \leq \int_0^1 \int_0^1  100\eta^2C_{L}^2\|f\|_1^2 (\xi\vee\xi^{\prime})^2\frac{1}{d^2}  dtdt^{\prime} + 2 \int_0^1 \int_0^1  10\eta^2 C_{L}\|f\|_1 (\xi\vee\xi^{\prime})\frac{1}{d} \left|h(t,t^{\prime})\right|dtdt^{\prime}\\
    & = 100\eta^2C_{L}^2\|f\|_1^2 (\xi\vee\xi^{\prime})^2\frac{1}{d^2} + 20\eta^2 C_{L}\frac{\|f\|_1^2}{\Delta_{f} \Delta^{\prime}_{f}} (\xi\vee\xi^{\prime})\frac{1}{d}\\
    & \leq C_{1} \eta^2(C_{L}\vee C_L^{2})(\xi\vee\xi^{\prime}\vee1)^2\frac{\|f\|_1^2}{\Delta_{f}\Delta^{\prime}_{f}} \frac{1}{d},
\end{align*}
where $C_{1} > 0$ is a sufficiently large universal constant. By the Cauchy-Schwarz inequality, $\|f\|_1\leq \|f\|_2.$ Summarizing,  \eqref{le2eq1} is bounded by 
\begin{align*}
\left|\|Z_{\bX}\|_2 -  \frac{\|f\|_2 \eta}{\sqrt{\Delta_{f} \Delta^{\prime}_{f}}}\right| &\leq C_{1} \eta^2 (C_{L}\vee C_L^{2}) (\xi\vee\xi^{\prime}\vee1)^2 \frac{\|f\|_1^2}{\Delta_{f} \Delta^{\prime}_{f}} \frac{1}{d} \frac{\sqrt{\Delta_{f} \Delta^{\prime}_{f}}}{\|f\|_2 \eta}\\
&\leq C_{1} \eta(C_{L}\vee C_L^{2})(\xi\vee\xi^{\prime}\vee1)^2\frac{\|f\|_1}{\sqrt{\Delta_{f} \Delta^{\prime}_{f}}} \frac{1}{d},
\end{align*}
proving \eqref{le1eq2}.

We now finish the proof. Using $T_{\bX}=Z_{\bX}/\|Z_{\bX}\|_2$ and that $(a+b)^2 \leq 2 a^2 + 2b^2$ for arbitrary real numbers $a,b$, we bound
\begin{align*}
    \left\|T_{\bX} - \frac{\sqrt{\Delta_{f} \Delta^{\prime}_{f}} h}{\|f\|_2}\right\|_2^2 &\leq 
    2 \left\|T_{\bX} - \frac{\sqrt{\Delta_{f} \Delta^{\prime}_{f}} Z_{\bX}}{\|f\|_2 \eta}\right\|_2^2 + 2 \left\|\frac{\sqrt{\Delta_{f} \Delta^{\prime}_{f}} Z_{\bX}}{\|f\|_2 \eta} - \frac{\sqrt{\Delta_{f} \Delta^{\prime}_{f}} \eta h}{ \|f\|_2  \eta}\right\|_2^2\\
    &\leq 2\frac{\Delta_{f} \Delta^{\prime}_{f}}{\|f\|_2^2 \eta^2} \left|\frac{\|f\|_2 \eta}{\sqrt{\Delta_{f} \Delta^{\prime}_{f}}} - \|Z_{\bX}\|_2 \right|^2 + \frac{2 \Delta_{f} \Delta^{\prime}_{f}}{\|f\|_2^2 \eta^2} \left\|Z_{\bX} - \eta h\right\|_2^2.
\end{align*}
Applying \eqref{le1eq2} to the first and \eqref{le1eq1} to the second term and using again $\|f\|_1\leq \|f\|_2$, as well as $\Delta_{f},\Delta'_{f}\leq1,$ it follows
\begin{align*}
   \left\|T_{\bX} - \frac{\sqrt{\Delta_{f} \Delta^{\prime}_{f}} h}{\|f\|_2}\right\|_2^2 &\leq \frac{2 \Delta_{f} \Delta^{\prime}_{f}}{\|f\|_2^2 \eta^2}(C_{1})^2 \eta^2 (C_L^4 \vee C_L^2) (\xi\vee\xi^{\prime}\vee1)^4\frac{\|f\|_1^2}{\Delta_{f} \Delta^{\prime}_{f}} \frac{1}{d^2}\\
   & \quad + \frac{2 \Delta_{f} \Delta^{\prime}_{f}}{\|f\|_2^2 \eta^2} 100\eta^{2} C_{L}^2\|f\|_1^2 (\xi\vee\xi^{\prime})^2\frac{1}{d^{2}}\\
   & \leq K^2(C_L^4 \vee C_L^2) (\xi\vee\xi^{\prime}\vee1)^4\frac{1}{d^2},
\end{align*}
for a universal constant $K>0$.
\end{proof}

\begin{lemma}
\label{le33}
For any measurable functions $h,g: \mathbb{R}^2 \to [0,\infty)$,
\begin{align*}
    &\inf_{\eta, \xi, \xi', t,t',\tilde{t}, \tilde{t}' \in \mathbb{R}, \, 
    \tilde{\eta},\tilde{\xi}, \tilde{\xi}'\in\mathbb{R}\backslash\{0\}} 
\frac{\sqrt{|\tilde{\xi}\tilde{\xi}'|}}{|\tilde \eta|}
    \Big\|\eta g\big(\xi \cdot -t, \xi' \cdot - t'\big) - \tilde{\eta} h\big(\tilde{\xi} \cdot - \tilde{t}, \tilde{\xi}' \cdot - \tilde{t'}\big)\Big\|_{L^2(\mathbb{R}^2)}\\
    &\geq 
    \inf_{a,b,c,b',c' \in \mathbb{R}} \Big\|a g\big(b\cdot-c, b' \cdot -c'\big) -h\Big\|_{L^2(\mathbb{R}^2)}.
\end{align*}
\end{lemma}

\begin{proof}
For arbitrary $\eta, \xi, \xi', t,t',\tilde{t}, \tilde{t}' \in \mathbb{R},
\tilde{\eta},\tilde{\xi}, \tilde{\xi}'\in\mathbb{R}\backslash\{0\}$, substitution gives
\begin{align*}
&\int_{\mathbb{R}^2} \Big(\eta g\big(\xi u - t, \xi' v-t'\big)-\tilde{\eta} h\big(\tilde{\xi} u - \tilde{t}, \tilde{\xi}'v-\tilde{t}'\big)\Big)^2 \, dudv\\
=&\int_{\mathbb{R}^2} \left(\eta g\left(\frac{\xi}{\tilde{\xi}} (x+\tilde{t})-t, \frac{\xi'}{\tilde{\xi}'}(y+\tilde{t}')-t'\right) - \tilde{\eta} h(x,y)\right)^2 \frac{1}{|\tilde{\xi} \tilde{\xi}'|} \, dxdy\\
\geq&\tilde{\eta}^2 \int_{\mathbb{R}^2}\left(\frac{\eta}{\tilde{\eta}} g\left(\frac{\xi}{\tilde{\xi}} (x+\tilde{t})-t, \frac{\xi'}{\tilde{\xi}'}(y+\tilde{t}')-t'\right) -  h(x,y)\right)^2 \frac{1}{|\tilde{\xi} \tilde{\xi}'|} \, dxdy\\
\geq&\frac{\tilde{\eta}^2}{|\tilde{\xi}\tilde{\xi}'|} \inf_{a,b,c,b',c' \in \mathbb{R}} \Big\|a g\big(b\cdot-c, b' \cdot -c'\big) -h\Big\|_{L^2(\mathbb{R}^2)}^2.
\end{align*}
\end{proof}

\begin{proof}[Proof of Theorem \ref{thm.main_part1}]
Without loss of generality, for a test image $\mathbf{X}$, we assume its label $k$ is $0$. The analysis for $k = 1$ follows analogously due to the symmetry of $f_0$ and $f_1$.

Set $\Delta_{f_k}:=\alpha^{+}_{f_k} - \alpha^{-}_{f_k},$ $\Delta'_{f_k}:=\beta^{+}_{f_k} - \beta^{-}_{f_k},$ and
\begin{align*}
    h_k := f_k\big(\Delta_{f_k} \cdot+\alpha^{-}_{f_k}, \Delta'_{f_k} \cdot+\beta^{-}_{f_k}\big).
\end{align*}
Since $k=0,$ the entries of $\bX$ and $Z_{\bX}$ are described by the template function $f_0: \mathbb{R}^2 \to [0,\infty)$. Each transformed training image $Z_{\bX_i}$, with $i \in \{1, \dots, n\}$, corresponds to a template function $f_{k_i}:\mathbb{R}^2 \to [0,\infty)$. If $k_i=0,$ it follows from Lemma \ref{le1} and the triangle inequality that
\begin{align}
\left\|T_{\bX_i} - T_{\bX}\right\|_2 &= \left\|T_{\bX_i} - \frac{\sqrt{\Delta_{f_0} \Delta_{f_0}^{\prime}}}{\|f_0\|_2} h_{f_0} + \frac{\sqrt{\Delta_{f_0} \Delta_{f_0}^{\prime}}}{\|f_0\|_2} h_{f_0} - T_{\bX}\right\|_2\nonumber\\
    & \leq \left\|T_{\bX_i} - \frac{\sqrt{\Delta_{f_0} \Delta_{f_0}^{\prime}}}{\|f_0\|_2} h_{f_0}\right\|_2 + \left\|\frac{\sqrt{\Delta_{f_0} \Delta_{f_0}^{\prime}}}{\|f_0\|_2} h_{f_0} - T_{\bX}\right\|_2\nonumber\\
    & \leq K  (C_L\vee C_L^{2})  \Xi_n^2  \frac{1}{d} + K(C_L\vee C_L^{2})  \Xi_n^2 \frac{1}{d}\nonumber\\
    &= 2  K  (C_L\vee C_L^{2})  \Xi_n^2 \frac{1}{d}.\label{eq.TX_ub}
\end{align}
The support of the function $h_{f_{k}}$ is contained in $[0,1]^2.$ Recall that the separation quantity $D$ is defined in \eqref{eq.D_part2}. Applying Lemma \ref{le33} twice by assigning to $(h,\tilde \eta, \tilde \xi, \tilde \xi')$ the values $(f_0,\sqrt{\Delta_{f_0}\Delta'_{f_0}}/\|f_0\|_2, \Delta_{f_0},\Delta'_{f_0})$ and $(f_1,\sqrt{\Delta_{f_1}\Delta'_{f_1}}/\|f_1\|_2, \Delta_{f_1},\Delta'_{f_1})$ yields
\begin{align}
&\left\|\frac{\sqrt{\Delta_{f_1} \Delta_{f_1}^{\prime}}}{\|f_1\|_2} h_{f_1} - \frac{\sqrt{\Delta_{f_0}\Delta_{f_0}^{\prime}}}{\|f_0\|_2} h_{f_0}\right\|_2 \nonumber\\
    =& \left\|\frac{\sqrt{\Delta_{f_1} \Delta_{f_1}^{\prime}}}{\|f_1\|_2} h_{f_1} - \frac{\sqrt{\Delta_{f_0} \Delta_{f_0}^{\prime}}}{\|f_0\|_2} h_{f_0}\right\|_{L^2(\mathbb{R}^2)} \nonumber\\
    \geq& \frac{\inf_{a,b,b',c,c' \in \mathbb{R}} \|af_1(b\cdot  - c, b'\cdot-c') - f_0\|_{L^2(\mathbb{R}^2)}}{\|f_0\|_2} \vee \frac{\inf_{a,b,b',c,c' \in \mathbb{R}} \|af_0(b\cdot  - c, b'\cdot-c') - f_1\|_{L^2(\mathbb{R}^2)}}{\|f_1\|_2}\nonumber\\
    >& \frac{4 K (C_L\vee C_L^2)\Xi_n^2}{d},\label{eq.delta_variation_bd}
\end{align}
where we used the assumption $D>4 K (C_L\vee C_L^2)\Xi_n^2/d$ for the last step. For an $i$ with $k_i=1,$ we use the reverse triangle inequality
\begin{align*}
    |a^{\prime} - b^{\prime}| \geq |a-b| - |a^{\prime} - a| - |b^{\prime} - b|, \quad \text{for all} \ a,b,a^{\prime}, b^{\prime} \in \mathbb{R},
\end{align*}
inequality \eqref{eq.delta_variation_bd} and Lemma \ref{le1} to bound
\begin{align*}
    \left\|T_{\bX_i} - T_{\bX}\right\|_2 &=  \left\|T_{\bX_i} - \frac{\sqrt{\Delta_{f_1} \Delta_{f_1}^{\prime}}}{\|f_1\|_2} h_{f_1} + \frac{\sqrt{\Delta_{f_1} \Delta_{f_1}^{\prime}}}{\|f_1\|_2} h_{f_1} - \frac{\sqrt{\Delta_{f_0} \Delta_{f_0}^{\prime}}}{\|f_0\|_2} h_{f_0} +
    \frac{\sqrt{\Delta_{f_0} \Delta_{f_0}^{\prime}}}{\|f_0\|_2} h_{f_0} - T_{\bX}\right\|_2\\
    & \geq \left\|\frac{\sqrt{\Delta_{f_1} \Delta_{f_1}^{\prime}}}{\|f_1\|_2} h_{f_1} - \frac{\sqrt{\Delta_{f_0} \Delta_{f_0}^{\prime}}}{\|f_0\|_2} h_{f_0}\right\|_2 - \left\|T_{\bX_i} - \frac{\sqrt{\Delta_{f_1} \Delta_{f_1}^{\prime}}}{\|f_1\|_2} h_{f_1}\right\|_2 - \left\|T_{\bX} - \frac{\sqrt{\Delta_{f_0} \Delta_{f_0}^{\prime}}}{\|f_0\|_2} h_{f_0}\right\|_2\\
    &>  \frac{4 K (C_L\vee C_L^2)\Xi_n^2}{d}- \frac{K (C_L\vee C_L^2)\Xi_n^2}{d}
    - \frac{K (C_L\vee C_L^2)\Xi_n^2}{d} \\
    &=\frac{2 K (C_L\vee C_L^2)\Xi_n^2}{d}.
\end{align*}
Combining this with \eqref{eq.TX_ub}, we conclude that
\begin{align*}
    \wh{i} \in \argmin_{i\in\{1, \dots, n\}} \left\|T_{\bX_i}- T_{\bX}\right\|_2
\end{align*}
holds for some $i$ with $k_i=0$ implying $\wh{k} =0.$ Since $k=0$, this shows the assertion $\wh k=k$.
\end{proof}

\begin{proof}[Proof of Theorem \ref{thm27}]
We consider $f_0(x,y)=(1/4-|1/2-x|-|1/2-y|)_+$, whose support is contained in $\left[1/4,3/4\right]^2$, and its rectangular support exactly matches $\left[1/4,3/4\right]^2$. We now show that $f_0$ satisfies \eqref{eq.Lip_cond-gen} with Lipschitz constant $C_{f_{0}}=96.$ To verify this, observe that $|f_0(x,y)-f_0(x',y')|\leq (|x-x'|+|y-y'|).$ Thus, \eqref{eq.Lip_cond-gen} holds for any $C_{f_{0}}\geq 1/\|f_0\|_1.$ Using the definition of $f_0,$ we compute
\begin{align}
\|f_0\|_1=\int_{[0,1]^2} |f_0(x,y)| \, dxdy =\int_{[0,1]^2} f_0(x,y) \, dxdy =\frac{1}{96}.\label{eq.f0_L1_lb}
\end{align}
Hence the Lipschitz condition is satisfied with $C_{f_{0}}=96.$ Consider the template function $f_{0}$ has been deformed as $f_{0,\tau,\tau',\xi,\xi'}(\cdot,\cdot):=f_{0}(\xi\cdot-\tau,\xi'\cdot-\tau')$, where $\tau,\tau',\xi,\xi'$ satisfy Assumption~\ref{ass1-prime}. Then, the support of $f_{0,\tau,\tau',\xi,\xi'}(\cdot,\cdot)$ is contained within $[0,1]^2$. A generic image $\bX=(X_{j,\ell})_{j,\ell=1,\ldots,d}$ based on $f_{0,\tau,\tau',\xi,\xi'}$ is described as 
$$X_{j,\ell}=\eta\int_{I_{j,\ell}}d^2f_{0,\tau,\tau',\xi,\xi'}(x,y) \, dxdy.$$ 

Next, for any random deformation parameters $\tau,\tau'$ and $\xi,\xi'$ satisfying Assumption~\ref{ass1-prime}, we construct a local perturbation function $g$ on $f_{0,\tau,\tau',\xi,\xi'}$.  Note that the square $\cro{13/32,19/32}^{2}$ is contained in the support of $f_0$ and for all $(x,y)\in\cro{13/32,19/32}^{2}$, $f_0(x,y)\geq1/16$. Taking into account the random re-scaling and shifting, the square $$I_{c}=\cro{(13/32+\tau)/\xi,(19/32+\tau)/\xi}\times\cro{(13/32+\tau')/\xi',(19/32+\tau')/\xi'}$$ is contained in the support of $f_{0,\tau,\tau',\xi,\xi'}.$ We shall build the perturbation of $f_{0,\tau,\tau',\xi,\xi'}$ on $I_{c}$. More precisely, let
$$j_{*}^-:=\left\lceil\left(\frac{13}{32\xi}+\frac{\tau}{\xi}\right)d\right\rceil,\quad\quad j_{*}^+:=\left\lfloor\left(\frac{19}{32\xi}+\frac{\tau}{\xi}\right)d\right\rfloor,$$ and 
$$\ell_{*}^-:=\left\lceil\left(\frac{13}{32\xi'}+\frac{\tau'}{\xi'}\right)d\right\rceil,\quad\quad \ell_{*}^+:=\left\lfloor\left(\frac{19}{32\xi'}+\frac{\tau'}{\xi'}\right)d\right\rfloor,$$ which are the approximated grid location for $I_{c}$. Observe that $\cro{j_{*}^-/d,j_{*}^+/d}\times\cro{\ell_{*}^-/d,\ell_{*}^+/d}\subseteq I_{c}$. Moreover, provided $d\geq32(\xi\vee\xi')$, one can derive 
\begin{equation}\label{bound-jstar}
   j_{*}^+-j_{*}^- \geq\left(\frac{19}{32\xi}+\frac{\tau}{\xi}\right)d-\left(\frac{13}{32\xi}+\frac{\tau}{\xi}\right)d-2=\frac{3d}{16\xi}-2\geq\frac{d}{8\xi},
\end{equation}
and
\begin{equation}\label{bound-lstar}
 \ell_{*}^+-\ell_{*}^- \geq\left(\frac{19}{32\xi'}+\frac{\tau'}{\xi'}\right)d-\left(\frac{13}{32\xi'}+\frac{\tau'}{\xi'}\right)d-2=\frac{3d}{16\xi'}-2\geq\frac{d}{8\xi'}
\end{equation}
and thus, $\cro{j_{*}^-/d,j_{*}^+/d}\times\cro{\ell_{*}^-/d,\ell_{*}^+/d}$ is not empty. Set $\mathcal{I}:=\{j_{*}^-+1,\ldots,j_{*}^+\}$ and $\mathcal{I}':=\{\ell_{*}^-+1,\ldots,\ell_{*}^+\}$. For any $i\in\mathbb{N}$, let $a_{i}^{-}:=(i-3/4)/d$, $a_{i}^{+}:=(i-1/4)/d$. For any $j\in\mathcal{I}$, $\ell\in\mathcal{I}'$, define the following functions
$$S^{--}_{j\ell}(x,y):=\Big(\frac{1}{4d}-|x-a_j^-|-|y-a_\ell^-|\Big)_+,\quad S^{-+}_{j\ell}(x,y):=\Big(\frac{1}{4d}-|x-a_j^-|-|y-a_\ell^+|\Big)_+,$$ and
$$S^{+-}_{j\ell}(x,y):=\Big(\frac{1}{4d}-|x-a_j^+|-|y-a_\ell^-|\Big)_+,\quad S^{++}_{j\ell}(x,y):=\Big(\frac{1}{4d}-|x-a_j^+|-|y-a_\ell^+|\Big)_+.$$
Realizations of these functions are shown in Figure~\ref{fig-S-per}.
\begin{figure}[ht!]
 \centering
 \subfigure[$S_{j\ell}^{--}$]{\includegraphics[width=0.24\textwidth]{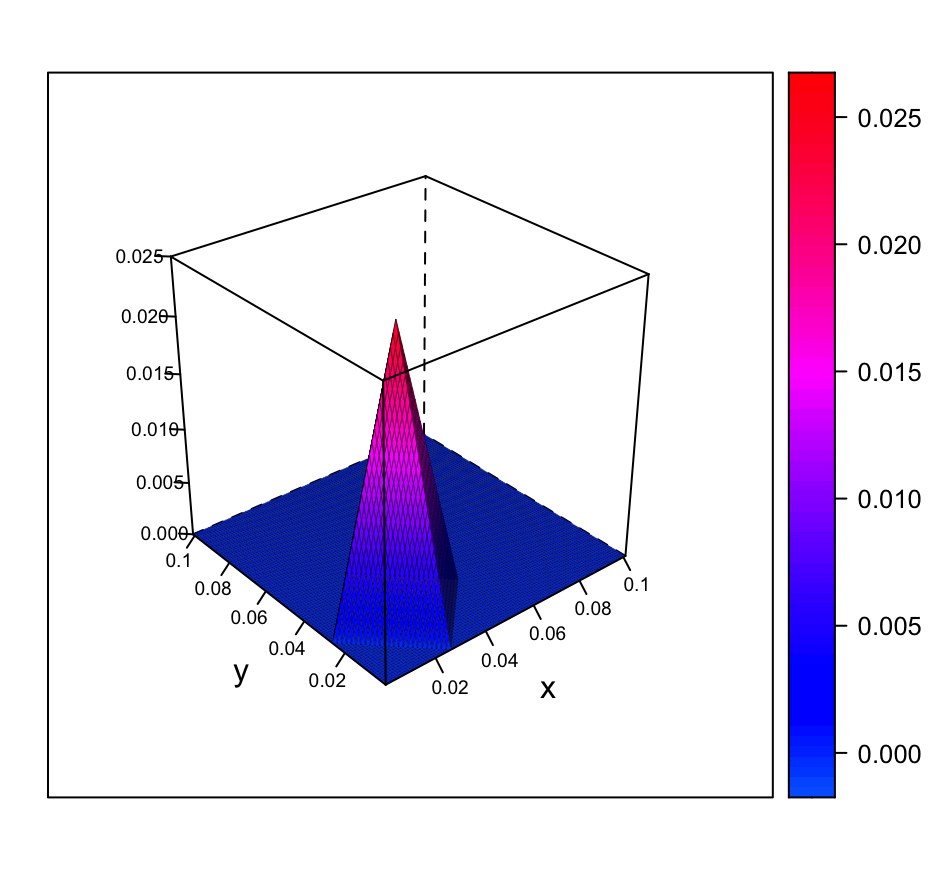}}
  \subfigure[$S_{j\ell}^{-+}$]{\includegraphics[width=0.24\textwidth]{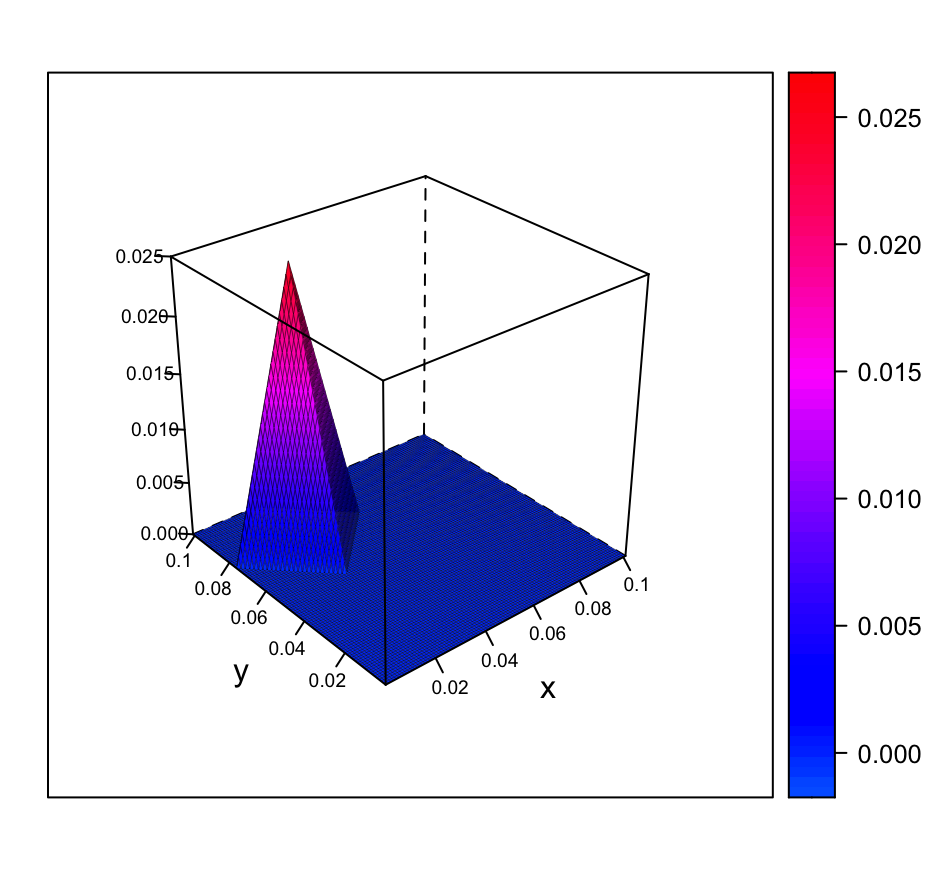}}
    \subfigure[$S_{j\ell}^{+-}$]{\includegraphics[width=0.24\textwidth]{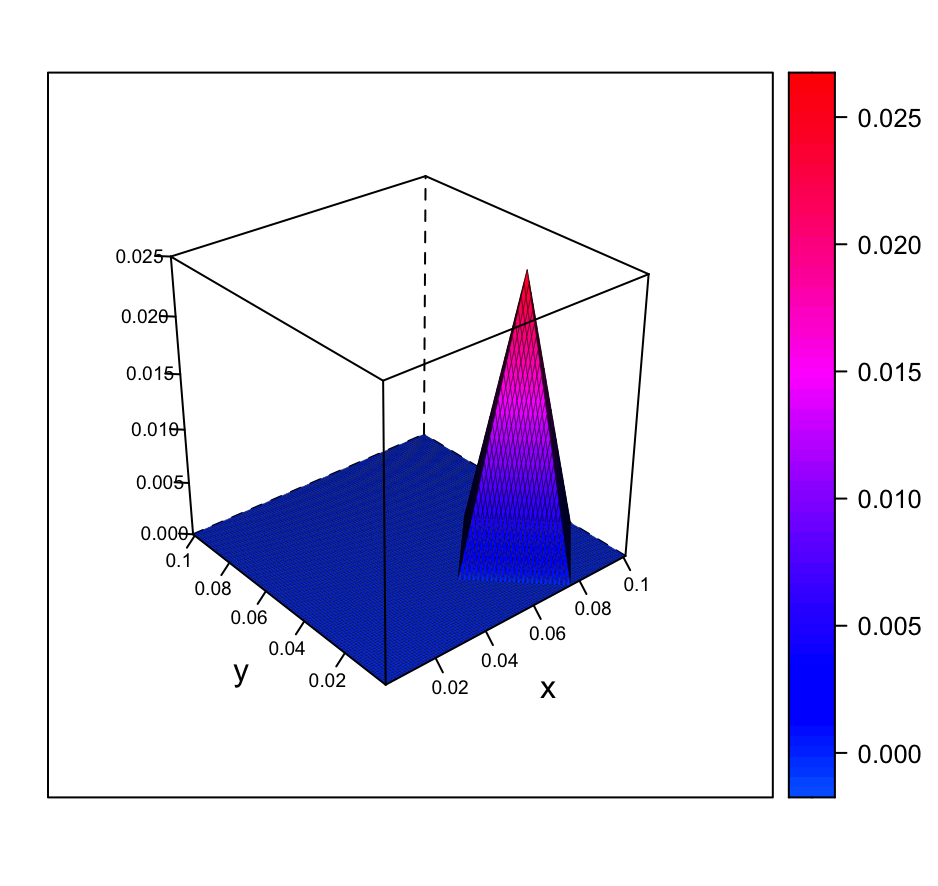}}
    \subfigure[$S_{j\ell}^{++}$]{\includegraphics[width=0.24\textwidth]{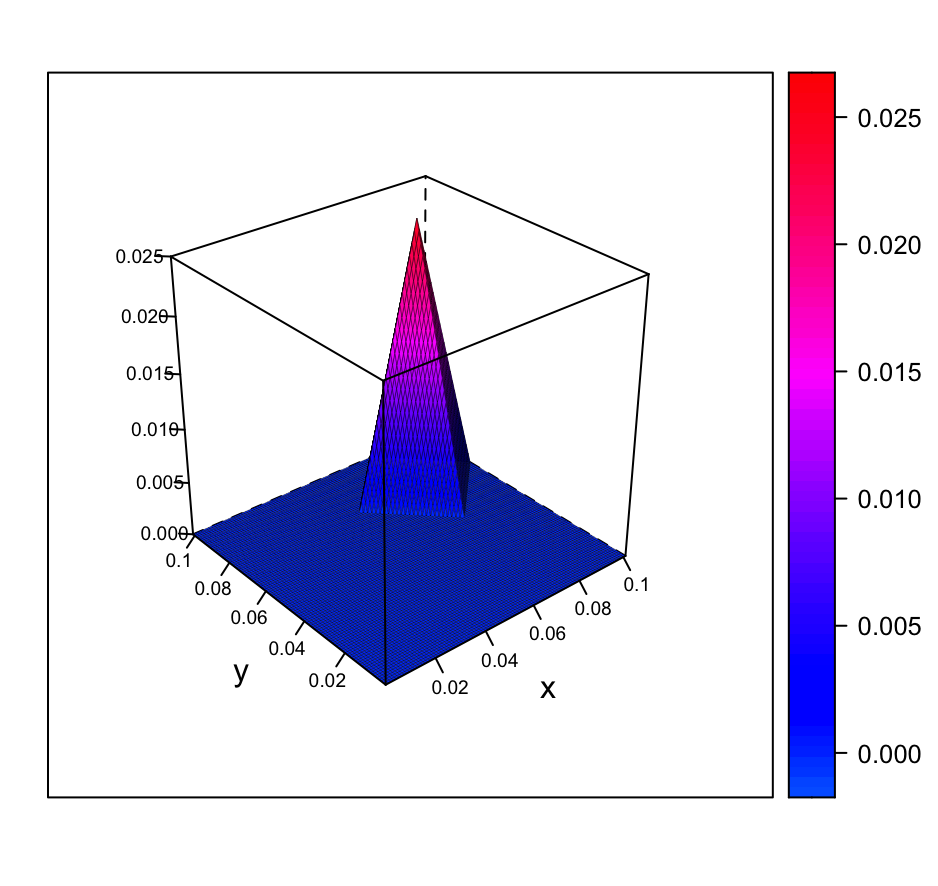}}
  \caption{Examples of the functions $S_{j\ell}^{--}$,$S_{j\ell}^{-+}$,$S_{j\ell}^{+-}$ and $S_{j\ell}^{++}$ on one pixel square when $d=10$. }
  \label{fig-S-per}
\end{figure}
The support of $S^{--}_{j\ell}$, $S^{-+}_{j\ell}$, $S^{+-}_{j\ell},$ and $S^{++}_{j\ell}$ is contained in $\cro{(j-1)/d,(j-1/2)/d}\times\cro{(\ell-1)/d,(\ell-1/2)/d}$, $\cro{(j-1)/d,(j-1/2)/d}\times\cro{(\ell-1/2)/d,\ell/d}$, $\cro{(j-1/2)/d,j/d}\times\cro{(\ell-1)/d,(\ell-1/2)/d}$, and $\cro{(j-1/2)/d,j/d}\times\cro{(\ell-1/2)/d,\ell/d}$ respectively. The supports of any two functions among $S^{--}_{j\ell}$, $S^{-+}_{j\ell}$, $S^{+-}_{j\ell}$ and $S^{++}_{j\ell}$ are disjoint. For any $j\in\mathcal{I}$, $\ell\in\mathcal{I}'$, set $$S_{j\ell}(x,y):=S^{--}_{j\ell}(x,y)-S^{-+}_{j\ell}(x,y)-S^{+-}_{j\ell}(x,y)+S^{++}_{j\ell}(x,y).$$ The support of the function $S_{j\ell}$ is contained in $[(j-1)/d,j/d]\times [(\ell-1)/d,\ell/d].$ For any $(j,\ell)\neq (j',\ell'),$ $S_{j,\ell}$ and $S_{j'\ell'}$ have disjoint support. An example of the function $S_{j\ell}$ is shown in Figure~\ref{fig-S}.
\begin{figure}[ht!]
 \centering
\includegraphics[width=0.4\textwidth]{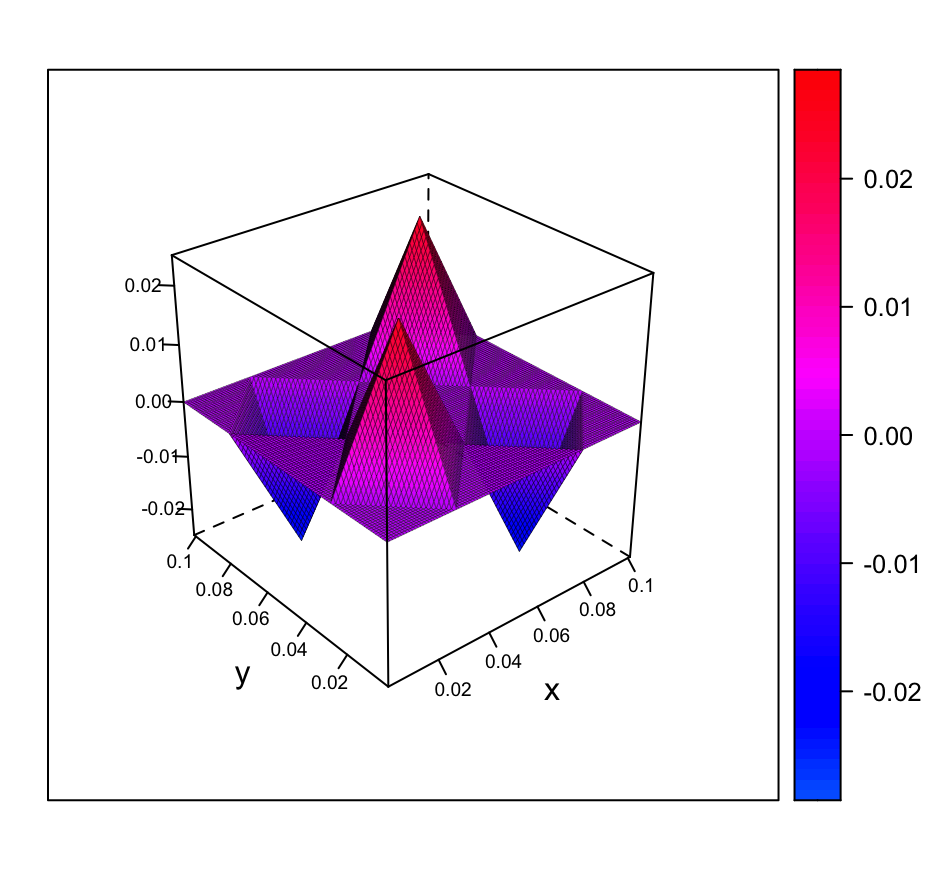}
  \caption{An example of the function $S_{j\ell}$ on one pixel square when $d=10$.}
  \label{fig-S}
\end{figure}

Consider the perturbation
\begin{align*}
g(x,y):=\sum_{j\in\mathcal{I},\ell\in\mathcal{I}'}S_{j\ell}(x,y).
\end{align*}
By construction, the support of $g$ is contained in $I_{c}$. Moreover, on each pixel $I_{j,\ell}\subseteq I_{c}$, the support of $S_{j\ell}$ has been divided into four regions according to the supports of $S^{--}_{j\ell}$, $S^{-+}_{j\ell}$, $S^{+-}_{j\ell}$ and $S^{++}_{j\ell}$ and for any $j\in\mathcal{I}$, $\ell\in\mathcal{I}'$,
\begin{align}
d^{2}\int_{I_{j,\ell}}g(x,y) \, dxdy=d^{2}\int_{I_{j,\ell}}S_{j\ell}(x,y)  \,  dxdy=0.\label{int-zero}
\end{align}
Now we consider the new function
\begin{align*}
f_{1,\tau,\tau',\xi,\xi'}(x,y):= f_{0,\tau,\tau',\xi,\xi'}(x,y)+ g(x,y).
\end{align*}
Recall that the function $f_{0,\tau,\tau',\xi,\xi'}$ takes positive values $\geq1/16$ on $I_{c}$, and the function $g$ is a small perturbation function defined on the interior of $I_{c}$. From Assumption~\ref{ass1-prime}, we know that $\xi,\xi'\geq1/2$. With $d\geq32(\xi\vee\xi')>4$, the function $f_{1,\tau,\tau',\xi,\xi'}$ takes non-negative values on $\cro{0,1}^{2}$, and therefore, so does $f_{1}:=f_{1,\tau,\tau',\xi,\xi'}(1/\xi(\cdot+\tau),1/\xi'(\cdot+\tau'))$. The rectangular support of $f_{1}$ is $\cro{1/4,3/4}^{2}$, as the perturbation $g$ does not does not modify the function values outside $I_c$. 

In the following, we check that $f_{1}$ satisfies the Lipschitz condition in \eqref{eq.Lip_cond-gen}. We first compute for $j\in\mathcal{I}$, $\ell\in\mathcal{I}'$,
\begin{align}
\int_{I_{j,\ell}}\left(S_{j\ell}(x,y)\right)^2 \, dxdy&=\int_{I_{j,\ell}}\left(S_{j\ell}^{--}(x,y)\right)^2+\left(S_{j\ell}^{-+}(x,y)\right)^2+\left(S_{j\ell}^{+-}(x,y)\right)^2+\left(S_{j\ell}^{++}(x,y)\right)^2\, dxdy\nonumber\\
&=4\int_{I_{j,\ell}} \left(S_{j\ell}^{--}(x,y)\right)^2 \, dxdy\nonumber\\
&=4\int_{(j-1)/d}^{(j-1/2)/d}\int_{(\ell-1)/d}^{(\ell-1/2)/d} \Big(\frac{1}{4d}-|x-a_j^-|-|y-a_\ell^-|\Big)_+^2 \ dxdy.\nonumber
\end{align}  
Substituting $u=2d(x-a_{j}^-)+1/2$ and $v=2d(y-a_{\ell}^-)+1/2$, we further obtain
\begin{align}
\int_{I_{j,\ell}}\left(S_{j\ell}(x,y)\right)^2 \, dxdy
&=\frac{4}{4d^{2}}\int_{0}^{1}\int_{0}^{1}\cro{\Big(\frac{1}{4d}-\frac{1}{2d}\left|u-\frac{1}{2}\right|-\frac{1}{2d}\left|v-\frac{1}{2}\right|\Big)_+}^2 \ dudv\nonumber\\
&=\frac{1}{4d^{4}}\int_{0}^{1}\int_{0}^{1}\cro{\Big(\frac{1}{2}-\left|u-\frac{1}{2}\right|-\left|v-\frac{1}{2}\right|\Big)_+}^2 \ dudv\nonumber\\
&= \frac{1}{192d^4}.\nonumber
\end{align}
Thus, for $d\geq32(\xi\vee\xi')$, using \eqref{bound-jstar} and \eqref{bound-lstar}, we have
\begin{align}
\|g\|_{L^2(\mathbb{R}^2)}^2=\|g\|_{2}^2=\sum_{j\in\mathcal{I},\ell\in\mathcal{I}'}\left(\int_{I_{j,\ell}}\left(S_{j\ell}(x,y)\right)^2dxdy\right)=\frac{(j_{*}^{+}-j_{*}^{-})(\ell_{*}^{+}-\ell_{*}^{-})}{192d^4}\geq\frac{1}{8^2\cdot 192\xi\xi'}\frac{1}{d^2}.\label{g-l2-bound}
\end{align}
Similarly, we deduce that 
\begin{align}
\|g\|_{1}&=\sum_{j\in\mathcal{I},\ell\in\mathcal{I}'}\left(\int_{I_{j,\ell}}\left|S_{j\ell}(x,y)\right|dxdy\right)\nonumber\\
&=\sum_{j\in\mathcal{I},\ell\in\mathcal{I}'}\cro{\frac{4}{4d^{2}}\int_{0}^{1}\int_{0}^{1}\Big(\frac{1}{4d}-\frac{1}{2d}\left|u-\frac{1}{2}\right|-\frac{1}{2d}\left|v-\frac{1}{2}\right|\Big)_+ \ dudv}\nonumber\\
&=\sum_{j\in\mathcal{I},\ell\in\mathcal{I}'}\cro{\frac{1}{2d^{3}}\int_{0}^{1}\int_{0}^{1}\Big(\frac{1}{2}-\left|u-\frac{1}{2}\right|-\left|v-\frac{1}{2}\right|\Big)_+ \ dudv}\nonumber\\
&=\frac{(j_{*}^{+}-j_{*}^{-})(\ell_{*}^{+}-\ell_{*}^{-})}{24d^3}\nonumber\\
&\leq\frac{3}{2048\xi\xi'}\frac{1}{d}.\label{g-l1-bound}
\end{align}
The last step follows from the definitions of $j_{*}^{+}, j_{*}^{-}, \ell_{*}^{+}, \ell_{*}^{-}$ which imply that $j_{*}^{+}-j_{*}^{-}\leq 3d/(16\xi)$ and $\ell_{*}^{+}-\ell_{*}^{-}\leq 3d/(16\xi').$
As a consequence of the triangle inequality, \eqref{eq.f0_L1_lb}, and \eqref{g-l1-bound}, we have
\begin{align*}
\|f_{1,\tau,\tau',\xi,\xi'}\|_{1}&\geq \|f_{0,\tau,\tau',\xi,\xi'}\|_{1}-\|g\|_{1}=\frac{\|f_{0}\|_{1}}{\xi\xi'}-\|g\|_{1}\geq\frac{1}{96\xi\xi'}-\frac{3}{2048\xi\xi'}\frac{1}{d}\geq\frac{1}{128\xi\xi'}.
\end{align*}
Moreover, for any $(x,y),(x',y')\in\mathbb{R}^2$, 
\begin{align*}
|f_{1,\tau,\tau',\xi,\xi'}(x,y)-f_{1,\tau,\tau',\xi,\xi'}(x',y')|&\leq\big|f_{0,\tau,\tau',\xi,\xi'}(x,y)-f_{0,\tau,\tau',\xi,\xi'}(x',y')\big|
    +\big|g(x,y)-g(x',y')\big|\\
    &=\big|f_{0}(\xi x-\tau,\xi'y-\tau')-f_{0}(\xi x'-\tau,\xi'y'-\tau')\big|
    +\big|g(x,y)-g(x',y')\big|\\
&\leq(\xi\vee\xi')\big(|x-x'|+|y-y'|\big)
    +2\big(|x-x'|+|y-y'|\big)\\
&=\cro{(\xi\vee \xi')+2}
    \big(|x-x'|+|y-y'|\big),
\end{align*}
which implies that for the constant $C_{L}:=128\xi\xi'\cro{(\xi\vee \xi')+2},$ $f_{1,\tau,\tau',\xi,\xi'}$ satisfies the Lipschitz condition in \eqref{eq.Lip_cond-gen} with $C_{L}$. Hence, $f_{1}$ satisfies the Lipschitz condition with constant $C_{f_{1}}:=C_{L}/(\xi\xi')=128\cro{(\xi\vee \xi')+2}$.

Observe that according to the definition of $f_0$, $$\|f_0\|_{L^2(\mathbb{R}^2)}=\|f_0\|_2\leq\sqrt{\left(\frac{1}{4}\right)^2\left(\frac{3}{4}-\frac{1}{4}\right)^2}=\frac{1}{8}.$$ Meanwhile, according to \eqref{g-l2-bound}, we have for any $d\geq32(\xi\vee\xi')$ and $\xi,\xi'\geq1/2$,
\begin{align*}
\|f_{1}-f_{0}\|_{L^2(\mathbb{R}^2)}=\xi\xi'\|f_{1,\tau,\tau',\xi,\xi'}-f_{0,\tau,\tau',\xi,\xi'}\|_{L^2(\mathbb{R}^2)}=\xi\xi'\|g\|_{L^2(\mathbb{R}^2)}\geq\frac{\sqrt{\xi\xi'}}{111 d}\geq\frac{1}{222 d},
\end{align*}
which implies
\begin{align*}
\frac{\|f_{0}-f_{1}\|_{L^2(\mathbb{R}^2)}}{\|f_0\|_{L^2(\mathbb{R}^2)}}\geq\frac{4}{111d}\geq\frac{1}{28d}.
\end{align*}
Due to \eqref{int-zero}, for any $j,\ell=1,\ldots,d$,
\begin{align*}
X_{j,\ell}&=\eta\int_{I_{j,\ell}}d^2f_{0,\tau,\tau',\xi,\xi'}(x,y) \, dxdy=\eta\int_{I_{j,\ell}}d^2f_{1,\tau,\tau',\xi,\xi'}(x,y) \, dxdy.
\end{align*}
This implies that both template functions $f_{0}$, $f_{1}$ generate the same image $\bX=(X_{j,\ell})_{j,\ell}$ under the same (but random) deformation parameters. It is impossible to infer the label from $\bX.$
\end{proof}

\section{Proofs for Section \ref{sec.CNN}}
\label{sec.proofs_CNNs}
Recall that $[ \mathbf{W}]$ denotes the quadratic support of the matrix $ \mathbf{W}.$

\begin{lemma}
\label{le3.1}
If $ \mathbf{W}=(W_{i,j})_{i,j=1,\ldots,d}$ and $ \mathbf{X}=(X_{i,j})_{i,j=1,\ldots,d}$ are matrices with non-negative entries, then,
\begin{align*}
    |\sigma([\mathbf{W}] \star \bX)|_{\infty}
    =\max_{r,s\in \mathbb{Z}} \ \sum_{i,j=1}^d W_{i+r,j+s} X_{i,j},
\end{align*}
where $W_{k,m}:=0$ whenever $k\wedge m\leq 0$ or $k\vee m >d.$
\end{lemma}

\begin{proof} Due to the fact that all entries of $ \mathbf{W}$ and $ \mathbf{X}$ are non-negative, it follows $\sigma([\mathbf{W}] \star  \mathbf{X})= [ \mathbf{W}] \star  \mathbf{X}.$ Assume that $[ \mathbf{W}]$ is of size $\ell$. As $|\cdot|_{\infty}$ extracts the largest value of $[ \mathbf{W}] \star  \mathbf{X}$ and each entry is the entrywise sum of the Hadamard product of $[ \mathbf{W}]$ and a $\ell \times \ell$ sub-matrix of $ \mathbf{X}'$, we rewrite
\begin{align*}
|[ \mathbf{W}] \star  \mathbf{X}|_{\infty} = \max_{u,v \in \mathbb{Z}} \, \sum_{i,j=1}^\ell [\mathbf{W}]_{i,j} \cdot X'_{i+u,j+v},
\end{align*}
where $X'_{k,m}=X_{k,m}$ for $k,m\in \{1, \dots, d\}$ and $X'_{k,m} := 0$ whenever $k \wedge m \leq 0$ or $k \vee m > d$. By definition of the quadratic support, there exist $R, S \in \{0, \dots, d-\ell\}$ such that $[\mathbf{W}]_{a,b} = W_{a+R, b+S},$ for all $a,b \in \{1, \dots, \ell\}$. Using that $[\mathbf{W}]_{i, j}:=0$ whenever $i \wedge j \leq 0$ or $i \vee j > \ell$, we rewrite
\begin{align*}
    \max_{r,s\in \mathbb{Z}} \, \sum_{i,j=1}^d W_{i+r,j+s} X_{i,j} &= \max_{r,s \in \mathbb{Z}} \,  \sum_{i,j \in \mathbb{Z}} [\mathbf{W}]_{i+r-R,j+s-S} X_{i,j}'\\
    & = \max_{r,s \in \mathbb{Z}} \,  \sum_{i',j' \in \mathbb{Z}} [\mathbf{W}]_{i',j'} X'_{i'-r+R, j'-s+S}\\
    & = \max_{u,v \in \mathbb{Z}} \,  \sum_{i,j \in \mathbb{Z}} [\mathbf{W}]_{i,j} X'_{i+u,j+v}\\
    & = \max_{u,v \in \mathbb{Z}} \, 
  \sum_{i,j=1}^{\ell} [\mathbf{W}]_{i,j} X'_{i+u,j+v}\\
    & = |\sigma([ \mathbf{W}] \star \bX)|_{\infty},
\end{align*}
proving the assertion.
\end{proof}

To prove Theorem \ref{th2-gen}, we need to establish some auxiliary results. It is convenient to first define the discrete $L^2$-inner product for non-negative functions $g,h:\mathbb{R}^{2}\to[0,\infty)$ by
\begin{align}
    \langle h,g\rangle_{2,d}:=\frac{1}{d^2} \sum_{j,\ell=1}^d \ol h_{j,\ell}\ol g_{j,\ell},
\label{eq.in_prod_2d_def2}
\end{align}
with $\overline h_{j,\ell}$ and $\overline g_{j,\ell}$ the pixel values defined in \eqref{ave-intensity}, for $j,\ell\in\{1,\ldots,d\}$. The corresponding norm is then defined as
\begin{align}
\|g\|_{2,d}:=\sqrt{\langle g,g\rangle_{2,d}}.
\label{eq.norm_2d_def}
\end{align}
The original $(j,\ell)$-th pixel cell is $I_{j,\ell} = \left[(j-1)/d,\, j/d\right) \times \left[(\ell-1)/d,\, \ell/d\right).$ For any $\alpha \in (0, 1)$, each combined block $\widetilde{I}_{j', \ell'}$ under the $(\alpha,d)$-block structure is given by
\[
\widetilde{I}_{j', \ell'} = \bigcup_{j \in \mathcal{K}(j'),\; \ell \in \mathcal{K}(\ell')} I_{j, \ell},
\]
with index set $\mathcal{K}(i)= \left\{ \lfloor d^{1-\alpha} + 1 \rfloor (i - 1) + 1,\; \ldots,\; \big( \lfloor d^{1-\alpha} + 1 \rfloor\, i \big) \wedge d \right\}$, for $i\in\{1,\ldots,d_{\alpha}\}$, where, $$d_{\alpha}=\left\lceil \frac{d}{\lfloor d^{1-\alpha}+1\rfloor}\right\rceil$$ denotes the number of subintervals along each axis. An illustration of the $(\alpha,d)$-block structure is shown in Figure~\ref{superpixel}. Under this construction, each side of $\widetilde{I}_{j', \ell'}$ has length at most $\lfloor d^{1 - \alpha} + 1 \rfloor / d$, and the total number of such combined blocks satisfies $(d_{\alpha})^2\leq\lceil d^{\alpha} \rceil^2$. 
\begin{figure}[htbp]
\centering
\includegraphics[width=0.6\textwidth]{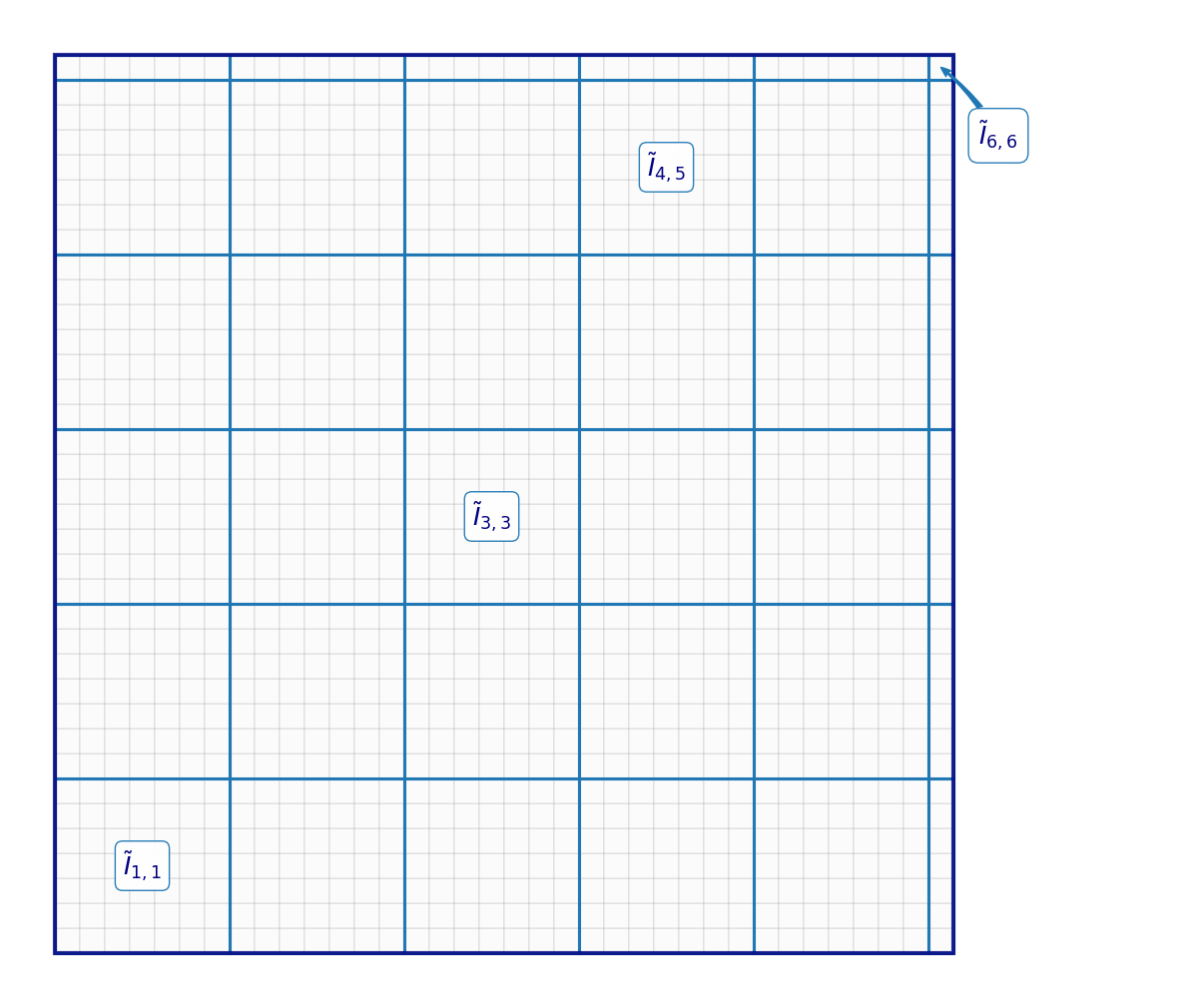}
\caption{Combined $(\alpha,d)$-blocks $\widetilde{I}_{j', \ell'}$ when $d=36$ and $\alpha=0.5$.}
\label{superpixel}
\end{figure}

Let $\lambda$ denote the Lebesgue measure on $\mathbb{R}^2$. Analogous to the definition of $\overline{h}_{j,\ell}$, for any continuous function $h:\mathbb{R}^{2} \to [0,\infty)$ and any $\alpha\in(0,1)$, we set
\begin{align}
\label{htilde}
\widetilde{h}^{\alpha}_{j',\ell'} := \frac{1}{\lambda\big(\widetilde{I}_{j',\ell'}\big)} \int_{\widetilde{I}_{j',\ell'}} h(u,v)\,dudv
\end{align}
for the average intensity of $h$ over the $(\alpha,d)$-combined block $\widetilde{I}_{j',\ell'}$. When $\alpha\in(0,1)$, we set
\begin{align}
h_{\alpha}(u,v) := \sum_{j',\ell'=1}^{d_{\alpha}} \widetilde{h}^{\alpha}_{j',\ell'} \, \1\big((u,v)\in \widetilde{I}_{j',\ell'}\big),\label{halpha-0-1}
\end{align}
which is constant on each $(\alpha,d)$-combined block obtained by merging pixel cells. It follows from the definition that the average intensity of $h_{\alpha}$ on the pixel $I_{j,\ell}$ is given by $$\overline{(h_{\alpha})}_{j,\ell}=d^2\int_{I_{j,\ell}}h_{\alpha}(u,v)dudv=\widetilde{h}^{\alpha}_{j',\ell'},$$ where $j', \ell' \in \{1, \ldots, d_{\alpha}\}$ denote the indices of the $(\alpha,d)$-combined block that contains $I_{j,\ell}$. When $\alpha=1$, we set 
\begin{equation}\label{halpha-1}
h_{\alpha}(u,v):=\sum_{j,\ell=1}^{d}\ol{h}_{j,\ell} \, \1\big((u,v)\in I_{j,\ell}\big).    
\end{equation}
\begin{lemma}\label{order-shift-bound}
Let $h: \mathbb{R}^2 \to [0,\infty)$ be a measurable function satisfying the Lipschitz condition \eqref{eq.Lip_cond-gen} with constant $C_L$. Define $h_\alpha$ as in \eqref{htilde}, \eqref{halpha-0-1}, and \eqref{halpha-1} for $\alpha \in (0,1]$. Then, for any integers $r, s$, and any $j, \ell \in \{1, \ldots, d\}$,
$$\left|\ol{\cro{h_{\alpha}\left(\cdot-\frac{r}{d},\cdot-\frac{s}{d}\right)}}_{j,\ell}-\ol{\cro{\left(h\left(\cdot-\frac{r}{d},\cdot-\frac{s}{d}\right)\right)_{\alpha}}}_{j,\ell}\right|\leq\frac{8C_L\|h\|_1}{d^{\alpha}}.$$
\end{lemma}
\begin{proof}
In the case $\alpha = 1$, no combining is applied, and thus
$h_{\alpha}\left(\cdot - r/d, \cdot - s/d\right)=\left(h\left(\cdot - r/d, \cdot - s/d\right)\right)_{\alpha}$. Therefore, it suffices to consider the case $\alpha \in (0,1)$.

By extending the definition of the $(\alpha,d)$-block structure to $\mathbb{R}^2$, we obtain 
\begin{align}
\ol{\cro{h_{\alpha}\left(\cdot-\frac{r}{d},\cdot-\frac{s}{d}\right)}}_{j,\ell}&=\ol{(h_{\alpha})}_{j-r,\ell-s}=\frac{1}{\lambda\big(\widetilde I_{i,i'}\big)}\int_{\widetilde I_{i,i'}}h(u,v)dudv,\label{a-shift}    
\end{align}
where $i=\lfloor(j-r)/\lfloor d^{1-\alpha}+1\rfloor\rfloor+1$ and $i'=\lfloor(\ell-s)/\lfloor d^{1-\alpha}+1\rfloor\rfloor+1$. For any integer $k\in\{1,\ldots,d_{\alpha}\}$ and any integer $r$, let $$\mathcal{K}(k)-r=\{m-r:\;\ m\in\mathcal{K}(k)\}=\{\lfloor d^{1 - \alpha} + 1 \rfloor(k-1)+1-r,\ldots,\lfloor d^{1 - \alpha} + 1 \rfloor k\wedge d-r\}$$ and define $\widetilde{\mathcal{I}}(k,k',r,s):=\cup_{m\in\mathcal{K}(k)-r,m'\in\mathcal{K}(k')-s}I_{m,m'}.$ Denoting $k=\lfloor j/\lfloor d^{1-\alpha}+1\rfloor\rfloor+1$ and $k'=\lfloor\ell/\lfloor d^{1-\alpha}+1\rfloor\rfloor+1$, it then follows that
\begin{align}
\ol{\cro{\left(h\left(\cdot-\frac{r}{d},\cdot-\frac{s}{d}\right)\right)_{\alpha}}}_{j,\ell}&=\frac{1}{\lambda\big(\widetilde I_{k,k'}\big)}\int_{\widetilde I_{k,k'}}h\left(u-\frac{r}{d},v-\frac{s}{d}\right)dudv\nonumber\\
&=\frac{1}{\lambda\big(\widetilde{\mathcal{I}}(k,k',r,s)\big)}\int_{\widetilde{\mathcal{I}}(k,k',r,s)}h\left(u,v\right)dudv.\label{shift-a}    
\end{align}
We next derive a bound for the distance between the blocks $\widetilde I_{i,i'}$ and $\widetilde{\mathcal{I}}(k,k',r,s)$ along both the
$x$- and $y$-axes. Observe that
\begin{align*}
\left|\left(\lfloor d^{1 - \alpha} + 1 \rfloor k-r\right)-\lfloor d^{1 - \alpha} + 1 \rfloor i\right|&=\left|\lfloor d^{1 - \alpha} + 1 \rfloor\left(\left\lfloor \frac{j}{\lfloor d^{1-\alpha}+1\rfloor}\right\rfloor-\left\lfloor\frac{j-r}{\lfloor d^{1-\alpha}+1\rfloor}\right\rfloor\right)-r\right|\\
&\leq\left|\lfloor d^{1 - \alpha} + 1 \rfloor\left( \frac{j}{\lfloor d^{1-\alpha}+1\rfloor}-\frac{j-r-1}{\lfloor d^{1-\alpha}+1\rfloor}\right)-r\right|\\
&\leq\lfloor d^{1 - \alpha} + 1 \rfloor,
\end{align*}
and similarly we can derive $\left|\left(\lfloor d^{1 - \alpha} + 1 \rfloor k'-s\right)-\lfloor d^{1 - \alpha} + 1 \rfloor i'\right|\leq\lfloor d^{1 - \alpha} + 1 \rfloor.$ This implies that for any $(u,v)\in\widetilde I_{i,i'}$ and any $(u',v')\in\widetilde{\mathcal{I}}(k,k',r,s)$, 
$|u-u'|\vee|v-v'|\leq2\lfloor d^{1 - \alpha} + 1 \rfloor/d$. Consequently, with $h$ satisfying the Lipschitz condition \eqref{eq.Lip_cond-gen}, we obtain that for any $(u,v)\in\widetilde I_{i,i'}$ and any $(u',v')\in\widetilde{\mathcal{I}}(k,k',r,s)$, 
\begin{align*}
|h(u,v)-h(u',v')|\leq C_L\|h\|_1(|u-u'|+|v-v'|)\leq4C_L\|h\|_1\frac{\lfloor d^{1 - \alpha} + 1 \rfloor}{d}\leq\frac{8C_L\|h\|_1}{d^{\alpha}}.
\end{align*}
Combining the above bound with \eqref{a-shift} and \eqref{shift-a} completes the proof.
\end{proof}

The next lemma provides a bound for the approximation error of Riemann sums.
\begin{lemma}
\label{le3.3Integrate-version}
For measurable functions $h, g: \mathbb{R}^2 \to [0,\infty)$ satisfying the Lipschitz condition \eqref{eq.Lip_cond-gen} with constant $C_{L},$ and $h_{\alpha}$ defined as in \eqref{htilde}, \eqref{halpha-0-1}, and  \eqref{halpha-1} for $\alpha\in(0,1],$ we have
\begin{itemize}
    \item[(i)]
     \begin{align*}
    &\left| \langle h_{\alpha},g\rangle_{2,d} - \int_{[0,1]^{2}}h(u,v)g(u,v) dudv\right|\leq\frac{8}{d^{\alpha}} \|g\|_1 \|h\|_1 \left(C_L+C_{L}^2\frac{1}{d}\right),
\end{align*}
\item [(ii)]
    \begin{align*}
    \left|\frac{1}{\|h_{\alpha}\|_{2,d}} - \frac{1}{\|h\|_2}\right| \leq \frac{8(C_L + 2C_L^2/d^{\alpha})}{d^{\alpha}\|h_{\alpha}\|_{2,d}},\quad\quad\left|\frac{1}{\|g\|_{2,d}} - \frac{1}{\|g\|_2}\right| \leq \frac{4(C_L + C_L^2/d)}{d\|g\|_{2,d}},
\end{align*}
\item [(iii)] \begin{align*}
    \left|\frac{1}{\|h_{\alpha}\|_{2,d}\|g\|_{2,d}} - \frac{1}{\|h\|_2\|g\|_2}\right| \leq\frac{4}{\|h_{\alpha}\|_{2,d} \|g\|_{2,d}}\left(\frac{3C_L}{ d^{\alpha}}+\frac{13C_{L}^{2}}{d^{2\alpha}}+\frac{24C_{L}^{3}}{d^{3\alpha}}+\frac{16C_{L}^{4}}{d^{4\alpha}}\right),
\end{align*}
\item[(iv)]if $\supp h\subseteq\cro{0,1}^{2}$, then, for any integers $r,s$,
\begin{align*}
\left|\frac{\langle h_{\alpha}(\cdot-r/d,\cdot-s/d),g\rangle_{2,d}}{\|h_{\alpha}\|_{2,d}\| g\|_{2,d}} - \frac{\int_{\cro{0,1}^{2}}h(u-r/d,v-s/d)g(u,v) \, dudv}{\|h\|_2\| g\|_2}\right|\leq\frac{28C_{L}}{d^{\alpha}}\Big(1+\frac{2C_{L}}{d^{\alpha}}\Big)^3. 
\end{align*}
\end{itemize}
\end{lemma} 
\begin{proof}
\textit{(i):} Fix $j,\ell\in\{1,\ldots,d\}$. By the Lipschitz property, we obtain for any $(u,v)\in I_{j,\ell}=[(j-1)/d, j/d) \times [(\ell-1)/d, \ell/d),$ that 
\[\left|\ol g_{j,\ell} - g(u,v)\right|\leq\frac{2C_{L}\|g\|_1}{d}\]
and for any $\alpha \in (0,1]$,
\begin{align}
\left|\ol{(h_{\alpha})}_{j,\ell} -h\left(u,v\right)\right|&\leq\frac{2\lfloor d^{1 - \alpha} + 1 \rfloor C_{L}\|h\|_1}{d}\leq\frac{4C_{L}\|h\|_1}{d^{\alpha}},
\end{align}
where we use $d^{\alpha}/2\leq d/\lfloor d^{1 - \alpha} + 1 \rfloor$ in the last step. Using this repeatedly, the triangle inequality gives 
\begin{align*}
&\left|\ol{(h_{\alpha})}_{j,\ell}\ol g_{j,\ell}-h\left(u,v\right)g(u,v)\right|\\
\leq&\left|\ol{(h_{\alpha})}_{j,\ell}\ol g_{j,\ell}-\ol{(h_{\alpha})}_{j,\ell}g(u,v)\right|+\left|\ol{(h_{\alpha})}_{j,\ell}g(u,v) - h\left(u,v\right)g(u,v)\right|\\
\leq&\left|\ol{(h_{\alpha})}_{j,\ell}\right|C_L\|g\|_1 \frac{2}{d} + | g(u,v)| C_{L} \|h\|_1 \frac{4}{d^{\alpha}}\\
\leq&\left|h\left(u,v\right)\right|C_L\|g\|_1 \frac{2}{d} + C^2_{L}\|g\|_1 \|h\|_1\frac{8}{d^{1+\alpha}}+ |g(u,v)| C_{L} \|h\|_1 \frac{4}{d^{\alpha}},
\end{align*}
which yields
\begin{align*}
&\int_{I_{j,\ell}}\left|\ol{(h_{\alpha})}_{j,\ell}\ol g_{j,\ell}-h\left(u,v\right)g(u,v)\right| dudv\\
\leq& C_{L}\|g\|_1 \frac{2}{d} \int_{I_{j,\ell}}h\left(u,v\right) \, dudv + \frac{8}{d^{3+\alpha}}C^2_{L} \|g\|_1 \|h\|_1 + C_{L} \|h\|_1 \frac{4}{d^{\alpha}} \int_{I_{j,\ell}}g(u,v) \, dudv.
\end{align*}
Rewriting
\begin{align*}
    \frac{1}{d^2} \sum_{j,\ell=1}^d\ol{(h_{\alpha})}_{j,\ell}\ol g_{j,\ell} = \sum_{j,\ell=1}^d \int_{I_{j,\ell}}\ol{(h_{\alpha})}_{j,\ell}\ol g_{j,\ell}dudv
\end{align*}
implies
\begin{align*}
&\left|\frac{1}{d^2} \sum_{j,\ell=1}^d\ol{(h_{\alpha})}_{j,\ell}\ol g_{j,\ell} - \int_{[0,1]^2}h\left(u,v\right)g(u,v) dudv\right|\\
\leq& C_{L}\|g\|_1 \frac{2}{d} \sum_{j,\ell=1}^d \int_{I_{j,\ell}} h\left(u,v\right) dudv + \frac{8}{d^{1+\alpha}}C^2_{L} \|g\|_1 \|h\|_1 + C_{L} \|h\|_1 \frac{4}{d^{\alpha}} \sum_{j,\ell=1}^d \int_{I_{j,\ell}}g(u,v)dudv\\
\leq& C_{L}\|g\|_1 \|h\|_1\frac{2}{d} + \frac{8}{d^{1+\alpha}} C^2_{L} \|g\|_1 \|h\|_1 + C_{L} \|h\|_1\|g\|_1 \frac{4}{d^{\alpha}}\\
\leq&\frac{8}{d^{\alpha}} \|g\|_1 \|h\|_1 \left(C_L+C_{L}^2\frac{1}{d}\right).
\end{align*}

\textit{(ii):} For positive real numbers $a,b,$ we have
\begin{align}
    \Big|\frac 1a-\frac 1b\Big|
    =\frac{|b-a|}{ab}=\frac{|a^2-b^2|}{(a+b)ab}\leq \frac{|a^2-b^2|}{a b^2}.\label{inequality-tool}
\end{align}
Now, set $a=\|h_{\alpha}\|_{2,d}$ and $b=\|h\|_2.$ Using an argument similar to the one used in the proof of (i), with $g = h_{\alpha}$, one can obtain
\begin{align*}
|a^2-b^2|=\left|\|h_{\alpha}\|^2_{2,d}-\|h\|_2^2\right|\leq \frac{8}{d^{\alpha}} \|h\|_1^2 \left(C_{L} + \frac{2C_{L}^2}{d^{\alpha}}\right).
\end{align*}
Since $\|h\|_1\leq \|h\|_2,$ the result follows.

Next, set $a=\|g\|_{2,d}$ and $b=\|g\|_2.$ Using an argument similar to the one used in the proof of (i), with $h_{\alpha}=g$, it follows that
\begin{align*}
|a^2-b^2|=\left|\|g\|^2_{2,d}-\|g\|_2^2\right|\leq \frac{4}{d} \|g\|_1^2 \left(C_{L} + \frac{C_{L}^2}{d}\right).
\end{align*}
Since $\|g\|_1\leq \|g\|_2,$ the result follows.

\textit{(iii):} Let $a,b,c,d$ be positive real numbers. Applying the triangle inequality repeatedly yields
\begin{align*}
    \Big|\frac{1}{ac}-\frac{1}{bd}\Big|
    &\leq \frac{1}{a}\Big|\frac{1}{c}-\frac{1}{d}\Big|
    +\frac 1d\Big|\frac 1a-\frac 1b\Big|
    \leq 
    \frac{1}{a}\Big|\frac{1}{c}-\frac{1}{d}\Big|
    +\frac 1c\Big|\frac 1a-\frac 1b\Big|+\Big|\frac 1d-\frac 1c\Big|\Big|\frac 1a-\frac 1b\Big|.
\end{align*}
With $a=\|h_{\alpha}\|_{2,d}, b=\|h\|_2, c=\|g\|_{2,d}, d=\|g\|_2,$ and using (ii), we have
\begin{align*}
&\left|\frac{1}{\|h_{\alpha}\|_{2,d}\|g\|_{2,d}} - \frac{1}{\|h\|_2\|g\|_2}\right|\\
&\leq\frac{4(C_{L}/d + C_{L}^2/d^2)+8(C_{L}/d^{\alpha} + 2C_{L}^2/d^{2\alpha})+32(C_{L}/d+C_{L}^2/d^2)(C_{L}/d^{\alpha}+2C_{L}^2/d^{2\alpha})}{\|h_{\alpha}\|_{2,d} \|g\|_{2,d}}\\
&\leq\frac{12C_{L}/d^{\alpha} + 20C_{L}^2/d^{2\alpha}+32(C_{L}^2/d^{2\alpha}+3C_{L}^3/d^{3\alpha}+2C_{L}^4/d^{4\alpha})}{\|h_{\alpha}\|_{2,d} \|g\|_{2,d}}\\
&\leq\frac{4}{\|h_{\alpha}\|_{2,d} \|g\|_{2,d}}\left(\frac{3C_L}{ d^{\alpha}}+\frac{13C_{L}^{2}}{d^{2\alpha}}+\frac{24C_{L}^{3}}{d^{3\alpha}}+\frac{16C_{L}^{4}}{d^{4\alpha}}\right).
\end{align*}

\textit{(iv):} Using Lemma~\ref{order-shift-bound}, we obtain that
\begin{align}
&\langle\left| h_{\alpha}(\cdot-r/d,\cdot-s/d)-\cro{h(\cdot-r/d,\cdot-s/d)}_{\alpha}\right|,g\rangle_{2,d}\nonumber\\
&=\frac{1}{d^2}\sum_{j,\ell=1}^{d}\ol{g}_{j,\ell} \ol{\left| h_{\alpha}(\cdot-r/d,\cdot-s/d)-\cro{h(\cdot-r/d,\cdot-s/d)}_{\alpha}\right|}_{j,\ell}\nonumber\\
&\leq\frac{8C_L\|h\|_1}{d^{\alpha}}\frac{1}{d^2}\sum_{j,\ell=1}^{d}\ol{g}_{j,\ell}\nonumber\\
&=\frac{8C_L\|h\|_1\|g\|_1}{d^{\alpha}}.\label{order-change-differ}
\end{align}
The Cauchy-Schwarz inequality and the fact that $\supp h\subseteq\cro{0,1}^{2}$ give that $$\langle h_{\alpha}(\cdot-r/d,\cdot-s/d),g\rangle_{2,d}\leq \|h_{\alpha}\|_{2,d}\|g\|_{2,d}.$$ Combining the triangle inequality with (i), (iii), \eqref{order-change-differ} and $\|g\|_1\leq \|g\|_2$ yields
\begin{align*}
&\left|\frac{\langle h_{\alpha}(\cdot-r/d,\cdot-s/d),g\rangle_{2,d}}{\| h_{\alpha}\|_{2,d}\|g\|_{2,d}} - \frac{\int_{\cro{0,1}^{2}}h(u-r/d,v-s/d)g(u,v) \, dudv}{\|h\|_2 \|g\|_2}\right|\\
&\leq\big| \langle h_{\alpha}(\cdot-r/d,\cdot-s/d),g\rangle_{2,d}\big|\left|\frac{1}{\|h_{\alpha}\|_{2,d}\|g\|_{2,d}} - \frac{1}{\|h\|_2\|g\|_2}\right| \\
&\quad+\frac{1}{\|h\|_2\|g\|_2} \left|\langle h_{\alpha}(\cdot-r/d,\cdot-s/d),g\rangle_{2,d} - \int_{\cro{0,1}^{2}} h(u-r/d,v-s/d)g(u,v) \, dudv\right|\\
&\leq\big| \langle h_{\alpha}(\cdot-r/d,\cdot-s/d),g\rangle_{2,d}\big|\left|\frac{1}{\|h_{\alpha}\|_{2,d}\|g\|_{2,d}} - \frac{1}{\|h\|_2\|g\|_2}\right| \\
&\quad+\frac{1}{\|h\|_2\|g\|_2} \left|\langle\cro{h(\cdot-r/d,\cdot-s/d)}_{\alpha},g\rangle_{2,d} - \int_{\cro{0,1}^{2}} h(u-r/d,v-s/d)g(u,v) \, dudv\right|\\ 
&\quad+\frac{1}{\|h\|_2\|g\|_2} \langle\left| h_{\alpha}(\cdot-r/d,\cdot-s/d)-\cro{h(\cdot-r/d,\cdot-s/d)}_{\alpha}\right|,g\rangle_{2,d}\\
&\leq\frac{28C_{L}}{d^{\alpha}}\Big(1+\frac{2C_{L}}{d^{\alpha}}\Big)^3,
\end{align*}
where the last inequality can be verified by expanding $(1+2C_{L}/d^{\alpha})^3$ into powers of $C_{L}/d^{\alpha}.$
\end{proof}

\begin{prop}\label{l1-l2-connect}
For any function $f$ that satisfies Assumption~\ref{ass-1} and any deformation class $\mathcal{A}$ that satisfies Assumptions~\ref{ass1}-(i), we have for all $A\in\mathcal{A}$,
\begin{equation*}
   \|f\circ A\|_{1}\leq2C_{\mathcal{A}}C_{L}\|f\|_{1}\quad\mbox{and}\quad \|f\circ A\|_{2}\leq2\sqrt{2}C_{\mathcal{A}}C_{L}\|f\|_{1}.
\end{equation*}
\end{prop}
\begin{proof}
Under Assumptions~\ref{ass-1} and \ref{ass1}-(i), the support of $f\circ A$ is contained in $\cro{0,1}^2$, which implies that $f\big(A(0,0)\big)=0$ due to the continuity of $f$ and $A$. Given that Assumptions~\ref{ass-1} and \ref{ass1}-(i) hold, applying Lemma~\ref{Lip-composite}, we obtain
\begin{align}
\|f\circ A\|_{1}&=\int_{\cro{0,1}^2}\big|f\big( A(u,v)\big)\big| \, dudv\nonumber\\
&=\int_{\cro{0,1}^2}\left|f\big(A(u,v)\big)-f\big(A(0,0)\big)\right| \, dudv\nonumber\\
&\leq\int_{\cro{0,1}^2}2C_{\mathcal{A}}C_{L}\|f\|_{1}(u+v) \, dudv\nonumber\\
&=2C_{\mathcal{A}}C_{L}\|f\|_{1}.\label{l1-bound-con}
\end{align}
Now, we consider the bound for $\|f\circ A\|_2$. Again, we use the property that under Assumptions~\ref{ass-1} and \ref{ass1}-(i),  $f\big(A(0,0)\big)=0$. Applying Lemma~\ref{Lip-composite} yields
\begin{align}
\|f\circ A\|_{2}^{2}&=\int_{\cro{0,1}^2}\cro{f\big(A(u,v)\big)}^2 \, dudv\nonumber\\
&=\int_{\cro{0,1}^2} \Big| f\big( A(u,v)\big)\left(f\big(A(u,v)\big)-f\big(A(0,0)\big)\right)\Big| \, dudv\nonumber\\
&\leq\int_{\cro{0,1}^2} \big|f\big(A(u,v)\big)\big| 2C_{\mathcal{A}}C_{L}\|f\|_{1}(u+v) \, dudv\nonumber\\
&\leq4C_{\mathcal{A}}C_{L}\|f\|_{1}\int_{\cro{0,1}^2}\big|f\big(A(u,v)\big)\big|\, dudv\nonumber\\
&=4C_{\mathcal{A}}C_{L}\|f\|_{1}\|f\circ A\|_{1}\nonumber\\
&\leq\left(2\sqrt{2}C_{\mathcal{A}}C_{L}\|f\|_{1}\right)^2,\nonumber
\end{align}
where the last inequality follows with \eqref{l1-bound-con}. Taking square roots on both sides, we conclude that
$\|f\circ A\|_{2}\leq2\sqrt{2}C_{\mathcal{A}}C_{L}\|f\|_{1}.$
\end{proof}

In the next result, we demonstrate that for any image generated through deformation of the template function $g$, with a suitably designed filter function based on $g$, the output provided by the CNN layers is always greater than some quantity of order $1-O(1/d^{\alpha})$. In contrast, for any filter derived from the other template function $f$, the CNN layer output is always smaller than some quantity depending on the separation quantity between $f$ and $g$. 

For any deformation $A$, let $\overline{ \mathbf{X}}_{g\circ A} := (\overline{X}_{g\circ A,j,\ell})_{j,\ell = 1, \dots, d}$ with
\begin{align*}
    \overline{X}_{g\circ A,j,\ell}:=\frac{\overline {g\circ A}_{j,\ell}}{d\|g\circ A\|_{2,d}}
\end{align*}
and $\overline {g\circ A}_{j,\ell}$ be the average intensity of $g\circ A$ on $I_{j,\ell}$, as defined in \eqref{ave-intensity}. For any $\alpha\in(0,1]$, $ \mathbf{w}_{(f\circ A)_{\alpha}}:=(w_{(f\circ A)_{\alpha},j,\ell})_{j,\ell=1, \dots, d}$, where 
\begin{align*}
    w_{(f\circ A)_{\alpha}, j,\ell}:= \frac{\overline {\big((f\circ A)_{\alpha}\big)}_{j,\ell}}{d\|(f\circ A)_{\alpha}\|_{2,d}}.
\end{align*}

\begin{prop}
\label{prop.cnns}
Let $f, g$ be two non-negative functions satisfying Assumption \ref{ass-1} with Lipschitz constant $C_{L}$ and suppose that Assumptions \ref{ass1}-(i) and \ref{ass-cover} hold for the deformation set $\mathcal{A}$. Let $D(f,g)$ be the separation quantity defined as in \eqref{eq.D_part-cnn}.
Then, there are constants $C_{1}(C_{L},C_{\mathcal{A}})$ and $C_{2}(C_{L},C_{\mathcal{A}})$ such that for any $\alpha\in(0,1]$ and $\mathcal{A}_{d_{\alpha}} \subseteq \mathcal{A}$ as defined in Assumption~\ref{ass-cover}, we have
\begin{itemize}
\item[(i)]
\begin{align}
\max_{A'\in\mathcal{A}_{d_{\alpha}}}\big|\sigma([ \mathbf{w}_{(g\circ A')_{\alpha}}] \star \overline{ \mathbf{X}}_{g\circ A})\big|_{\infty}
   \geq 1-\frac{C_{1}(C_{L},C_{\mathcal{A}})}{d^{\alpha}},\label{filter-lower}
\end{align}
\item[(ii)]
\begin{align}
&\max_{A'\in\mathcal{A}_{d_{\alpha}}} \, \big|\sigma([ \mathbf{w}_{(f\circ A')_{\alpha}}] \star \overline{ \mathbf{X}}_{g\circ A})\big|_{\infty}
   \leq 1-\frac{D^2(f,g)\vee D^2(g,f)}{16C_{\mathcal{A}}^2C_{L}^2} + \frac{C_{2}(C_{L},C_{\mathcal{A}})}{d^{\alpha}}.\label{filter-upper}
\end{align}
\end{itemize}
\end{prop}

\begin{proof}
We first prove (i). Under Assumption~\ref{ass-cover}, for any $A\in\mathcal{A}$, there exists a deformation $A'\in \mathcal{A}_{d_{\alpha}}\subseteq\mathcal{A}$ and indices $j_0,\ell_0\in\{1,\ldots,d\}$ such that $A'(\cdot+j_0/d,\cdot+\ell_0/d)$ satisfies Assumption~\ref{ass1}-(i), and $$\left\|A'\left(\cdot+\frac{j_0}{d},\cdot+\frac{\ell_0}{d}\right)-A\right\|_{\infty}\leq d^{-\alpha}.$$
According to Lemma~\ref{le3.1}, we deduce that
\begin{align}
\big|\sigma([ \mathbf{w}_{(g\circ A')_{\alpha}}] \star \overline{ \mathbf{X}}_{g\circ A})\big|_{\infty}&=\max_{r,r'\in \mathbb{Z}} \sum_{j,\ell=1}^d \1_{1\leq j+r\leq d,1\leq \ell+r'\leq d} \, \frac{\overline {\big((g\circ A')_{\alpha}\big)}_{j+r,\ell+r'}}{\|(g\circ A')_{\alpha}\|_{2,d}\cdot d} \, \frac{\overline {g\circ A}_{j,\ell}}{\|g\circ A\|_{2,d}\cdot d}\nonumber\\
&\geq\sum_{j,\ell=1}^d \1_{1\leq j+j_{0}\leq d,1\leq \ell+\ell_{0}\leq d} \, \frac{\overline {\big((g\circ A')_{\alpha}\big)}_{j+j_{0},\ell+\ell_{0}}}{\|(g\circ A')_{\alpha}\|_{2,d}\cdot d} \, \frac{\overline {g\circ A}_{j,\ell}}{\|g\circ A\|_{2,d}\cdot d}\nonumber\\
&=\frac{\left\langle (g\circ A')_{\alpha}\left(\cdot+\frac{j_{0}}{d},\cdot+\frac{\ell_{0}}{d}\right),g\circ A\right\rangle_{2,d}}{\|(g\circ A')_{\alpha}\|_{2,d}\|g\circ A\|_{2,d}},\label{suffi-filter}
\end{align}
where the second equality follows from the fact that the support of $g \circ A'$ is contained in $\cro{0,1}^2$, provided that $A'\in\mathcal{A}_{d_{\alpha}}\subseteq\mathcal{A}$. To derive \eqref{filter-lower}, it is sufficient to show
\begin{equation*}
\frac{\left\langle (g\circ A')_{\alpha}\left(\cdot+\frac{j_{0}}{d},\cdot+\frac{\ell_{0}}{d}\right),g\circ A\right\rangle_{2,d}}{\|(g\circ A')_{\alpha}\|_{2,d}\|g\circ A\|_{2,d}}\geq1-C_{1}(C_{L},C_{\mathcal{A}})\cdot d^{-\alpha}.
\end{equation*}
In the next step, we show that under the provided conditions, the functions $g \circ A$ and $g \circ A'$ satisfy \eqref{eq.Lip_cond-gen} with Lipschitz constant $4C^3_{\mathcal{A}}C_{L}$. Observe that for any $g$ satisfying Assumption~\ref{ass-1}, and any $A$ fulfilling Assumption~\ref{ass1}-(i), we can derive
\begin{align}
\|g\|_{1}=\int_{[0,1]^2}g(x,y)  \, dxdy=\int_{[0,1]^2}g\circ A(u,v)\left|\det(J_{A}(u,v))\right|  \,  dudv\leq 2C_{\mathcal{A}}^2\|g\circ A\|_{1}.\label{joc-cal}
\end{align}
Using \eqref{joc-cal}, applying Lemma~\ref{Lip-composite} yields for any real numbers $u,u',v,v'$,
\begin{align*}
    \left|g\circ A(u,v)-g\circ A(u',v')\right|&\leq2C_{\mathcal{A}}C_{L}\|g\|_{1}(|u-u'|+|v-v'|)\\
&\leq4C_{\mathcal{A}}^3C_{L}\|g\circ A\|_{1}(|u-u'|+|v-v'|),
\end{align*}
which validates the claim. Applying Lemma \ref{le3.3Integrate-version} (iv) with $h=g\circ A'$ and $g=g\circ A$ allows us to bound 
\begin{align}
&\frac{\left\langle (g\circ A')_{\alpha}\left(\cdot+\frac{j_{0}}{d},\cdot+\frac{\ell_{0}}{d}\right),g\circ A\right\rangle_{2,d}}{\|(g\circ A')_{\alpha}\|_{2,d}\|g\circ A\|_{2,d}}\nonumber\\
\geq&\frac{\int_{\cro{0,1}^{2}} g\circ A'\left(u+\frac{j_{0}}{d},v+\frac{\ell_{0}}{d}\right)g\circ A(u,v) \, dudv}{\|g\circ A'\|_{2}\|g\circ A\|_{2}}- \frac{112C_{L}C^3_{\mathcal{A}}}{d^{\alpha}}\left(1+\frac{8C_{L}C^3_{\mathcal{A}}}{d^{\alpha}}\right)^3.\label{int-bound-gen}
\end{align}
Now we derive a lower bound for the first summand in \eqref{int-bound-gen}. Under Assumptions \ref{ass-1} and \ref{ass-cover}, we have
\begin{equation}\label{inf-bound-tool}
    \left\|g\circ A'\left(\cdot+\frac{j_{0}}{d},\cdot+\frac{\ell_{0}}{d}\right)-g\circ A\right\|_{\infty}\leq2C_L\|g\|_{1}\cdot d^{-\alpha}.
\end{equation}
Let $h_{1},h_{2}:\cro{0,1}^2\rightarrow \mathbb{R}$ be two non-negative bounded functions. Using the triangle inequality, 
\begin{align}
    \| h_1 h_2\|_1 \geq \|h_1\|_2^2-\| h_1 (h_1-h_2)\|_1\geq \|h_1\|_2^2-\|h_1-h_2\|_\infty \|h_1\|_1.\label{eq.h1h2}
\end{align}
Applying this inequality with $h_{1}=g\circ A$ and $h_{2}=g\circ A'\left(\cdot+j_{0}/d,\cdot+\ell_{0}/d\right)$ yields
\begin{align}
\int_{\cro{0,1}^{2}}g\circ A'\left(u+\frac{j_{0}}{d},v+\frac{\ell_{0}}{d}\right)g\circ A(u,v) \, dudv&\geq\|g\circ A\|_{2}^{2}-\frac{2C_{L}\|g\|_1}{d^{\alpha}}\|g\circ A\|_{1}\nonumber\\
&\geq\|g\circ A\|_{2}^{2}-\frac{4C_{L}C_{\mathcal{A}}^2}{d^{\alpha}}\|g\circ A'\|_{1}\|g\circ A\|_{1},\label{summand-inequa-1}
\end{align}
where the first inequality follows from \eqref{inf-bound-tool}, and the second inequality uses \eqref{joc-cal}. By the Cauchy-Schwarz inequality, $\|g\circ A\|_{1}\leq\|g\circ A\|_{2}$, for all $A\in\mathcal{A}$, which, together with \eqref{summand-inequa-1}, implies that
\begin{align}
\frac{\int_{\cro{0,1}^{2}} g\circ A'\left(u+\frac{j_{0}}{d},v+\frac{\ell_{0}}{d}\right)g\circ A(u,v) \, dudv}{\|g\circ A'\|_{2}\|g\circ A\|_{2}}&\geq\frac{\|g\circ A\|_{2}^{2}}{\|g\circ A'\|_{2}\|g\circ A\|_{2}}-\frac{4C_{L}C_{\mathcal{A}}^2}{d^{\alpha}}\frac{\|g\circ A\|_{1}\|g\circ A'\|_{1}}{\|g\circ A'\|_{2}\|g\circ A\|_{2}}\nonumber\\
&\geq\frac{\|g\circ A\|_{2}}{\|g\circ A'\|_{2}}-\frac{4C_{L}C_{\mathcal{A}}^2}{d^{\alpha}}.\label{summand-inequa-1-precise}
\end{align}
Next, we proceed by bounding the term $\|g\circ A\|_{2}/\|g\circ A'\|_{2}$. Let $h_{1},h_{2}:\cro{0,1}^2\rightarrow \mathbb{R}$ be two non-negative bounded functions. Interchanging the role of $h_1$ and $h_2$ in \eqref{eq.h1h2} gives $\| h_1 h_2\|_1 \geq \|h_2\|_2^2-\|h_1-h_2\|_\infty \|h_2\|_1.$ Using triangle inequality, we also obtain $\| h_1\|_2^2=\| h_1^2\|_1\geq \|h_1h_2\|_1-\|h_1(h_2-h_1)\|_1\geq \|h_1h_2\|_1-\|h_1\|_1\|h_2-h_1\|_\infty.$ Combining the previous two inequalities gives $\| h_1\|_2^2\geq \|h_2\|_2^2-(\|h_1\|_1+\|h_2\|_1) \|h_1-h_2\|_\infty $ and dividing by $\|h_2\|_2^2$ yields $\| h_1\|_2^2/\|h_2\|_2^2 \geq 1-(\|h_1\|_1+\|h_2\|_1)\|h_1-h_2\|_\infty /\|h_2\|_2^2$. Applying this inequality with $h_{1}=g\circ A$ and $h_{2}=g\circ A'\left(\cdot+j_{0}/d,\cdot+\ell_{0}/d\right)$ as well as using \eqref{inf-bound-tool} yields
\begin{align}
    \frac{\|g\circ A\|_{2}^{2}}{\|g\circ A'\|_{2}^{2}}&=\frac{\|g\circ A\|_{2}^{2}}{\left\|g\circ A'\left(\cdot+\frac{j_0}{d},\cdot+\frac{\ell_0}{d}\right)\right\|_{2}^{2}}\nonumber\\
    &\geq1-\frac{\|g\circ A\|_{1}+\left\|g\circ A'\left(\cdot+\frac{j_0}{d},\cdot+\frac{\ell_0}{d}\right)\right\|_{1}}{\left\|g\circ A'\left(\cdot+\frac{j_0}{d},\cdot+\frac{\ell_0}{d}\right)\right\|_{2}^{2}}\cdot\frac{2C_L\|g\|_{1}}{d^{\alpha}}\nonumber\\
    &=1-\frac{\|g\circ A\|_{1}+\left\|g\circ A'\right\|_{1}}{\left\|g\circ A'\right\|_{2}^{2}}\cdot\frac{2C_L\|g\|_{1}}{d^{\alpha}}\nonumber\\
    &\geq1-\frac{\|g\circ A\|_{1}+\left\|g\circ A'\right\|_{1}}{\left\|g\circ A'\right\|_{2}^{2}}\cdot\frac{4C_LC^2_{\mathcal{A}}\|g\circ A'\|_{1}}{d^{\alpha}}\nonumber\\
    &\geq1-\frac{4C_LC^2_{\mathcal{A}}}{d^{\alpha}}\left(\frac{\|g\circ A\|_{1}}{\|g\circ A'\|_{1}}+1\right),\label{ratio-l2-square}
\end{align}
where the second equality follows from Assumptions~\ref{ass1}-(i) and \ref{ass-cover}, the second-to-last inequality comes from \eqref{joc-cal}, and the last inequality is obtained using $\|g\circ A'\|_{1}\leq\|g\circ A'\|_{2}$.
Moreover,
\begin{align}
\frac{\|g\circ A\|_{1}}{\|g\circ A'\|_{1}}&=1+\frac{\int_{\cro{0,1}^{2}}g\circ A(u,v)-g\circ A'\left(u+\frac{j_0}{d},v+\frac{\ell_0}{d}\right)dudv}{\|g\circ A'\|_{1}}\nonumber\\
&\leq1+\frac{\left\|g\circ A'\left(\cdot+\frac{j_{0}}{d},\cdot+\frac{\ell_{0}}{d}\right)-g\circ A\right\|_{\infty}}{\|g\circ A'\|_{1}}\nonumber\\
&\leq 1+\frac{2C_{L}\|g\|_{1}}{d^{\alpha}}\frac{1}{\|g\circ A'\|_{1}}\nonumber\\
&\leq1+\frac{4C_{L}C_{\mathcal{A}}^2}{d^{\alpha}},\label{ratio-l1}
\end{align}
where the second-to-last inequality comes from \eqref{inf-bound-tool}, and the last inequality is due to \eqref{joc-cal}.
By plugging \eqref{ratio-l1} into \eqref{ratio-l2-square}, we deduce that
\begin{align*}
\frac{\|g\circ A\|_{2}^{2}}{\|g\circ A'\|_{2}^{2}}\geq1-\frac{4C_{L}C_{\mathcal{A}}^2}{d^{\alpha}}\left(2+\frac{4C_{L}C_{\mathcal{A}}^2}{d^{\alpha}}\right)=1-\frac{8C_{L}C_{\mathcal{A}}^2}{d^{\alpha}}\left(1+\frac{2C_{L}C_{\mathcal{A}}^2}{d^{\alpha}}\right),
\end{align*}
which implies that
\begin{align}
\frac{\|g\circ A\|_{2}}{\|g\circ A'\|_{2}}\geq1-\frac{8C_{L}C_{\mathcal{A}}^2}{d^{\alpha}}\left(1+\frac{2C_{L}C_{\mathcal{A}}^2}{d^{\alpha}}\right).\label{ratio-l2}
\end{align}
Combining \eqref{int-bound-gen}, \eqref{summand-inequa-1-precise} and \eqref{ratio-l2}, we obtain
\begin{align}
&\frac{\left\langle (g\circ A')_{\alpha}, g\circ A\left(\frac{j_{0}}{d}+\cdot,\frac{\ell_{0}}{d}+\cdot\right)\right\rangle_{2,d}}{\|(g\circ A')_{\alpha}\|_{2,d}\|g\circ A\|_{2,d}}\nonumber\\
&\geq\frac{\|g\circ A\|_{2}}{\|g\circ A'\|_{2}}-\frac{4C_{L}C_{\mathcal{A}}^2}{d^{\alpha}}- \frac{112C_{L}C^3_{\mathcal{A}}}{d^{\alpha}}\left(1+\frac{8C_{L}C^3_{\mathcal{A}}}{d^{\alpha}}\right)^3\nonumber\\
&\geq1-\frac{16C^2_{L}C_{\mathcal{A}}^4}{d^{2{\alpha}}}
-\frac{12C_{L}C_{\mathcal{A}}^2}{d^{\alpha}}- \frac{112C_{L}C^3_{\mathcal{A}}}{d^{\alpha}}\left(1+\frac{8C_{L}C^3_{\mathcal{A}}}{d^{\alpha}}\right)^3\nonumber\\
&\geq1-\frac{C_1(C_{L},C_{\mathcal {A}})}{d^{\alpha}},
\end{align}
where $C_1(C_{L},C_{\mathcal {A}})$ is a universal constant depending only on $C_{L}$ and $C_{\mathcal{A}}$. This proves (i).

Next we proceed to prove (ii). With Lemma~\ref{le3.1} and the non-negativity of $f$, we rewrite 
\begin{align*}
    \big|\sigma([ \mathbf{w}_{(f\circ A')_{\alpha}}] \star \overline{ \mathbf{X}}_{g\circ A})\big|_{\infty} 
    &= \max_{r,r' \in \{-d, \dots, d\}} \sum_{j, \ell=1}^d\1_{1\leq j+r\leq d,1\leq \ell+r'\leq d}\frac{\overline {\big((f\circ A')_{\alpha}\big)}_{j+r,\ell+r'}}{d\|(f\circ A')_{\alpha}\|_{2,d}}\cdot\frac{\overline{g\circ A}_{j,\ell}}{d\|g\circ A\|_{2,d}}\\
    &\leq\max_{r,r' \in \{-d, \dots, d\}}
    \frac{\left\langle (f\circ A')_{\alpha}\left(\cdot +\frac rd,\cdot+\frac {r'}{d}\right),g\circ A\right\rangle_{2,d}}{\|(f\circ A')_{\alpha}\|_{2,d}\|g\circ A\|_{2,d}}.
\end{align*}
Under Assumptions~\ref{ass-1}, \ref{ass1}-(i) and \ref{ass-cover}, we can deduce similarly through Lemma~\ref{Lip-composite} that the functions $f \circ A'$ and $g \circ A$ satisfy \eqref{eq.Lip_cond-gen} with a Lipschitz constant $4C^3_{\mathcal{A}}C_{L}$. For any $\alpha\in(0,1]$ and any $A'\in\mathcal{A}_{d_{\alpha}}$, Lemma \ref{le3.3Integrate-version} (iv) allows us to bound
\begin{align}
    &\max_{r,r' \in \{-d, \dots, d\}}
    \frac{\left\langle (f\circ A')_{\alpha}\left(\cdot+\frac rd,\cdot+\frac{r'}{d}\right),g\circ A\right\rangle_{2,d}}{\|(f\circ A')_{\alpha}\|_{2,d}\|g\circ A\|_{2,d}}\nonumber\\
    \leq&\max_{r,r' \in \{-d, \dots, d\}}\frac{\int_{\cro{0,1}^{2}}f\circ A'\left(u+\frac rd,v+\frac{r'}{d}\right)g\circ A(u,v)dudv}{\|f\circ A'\|_2\|g\circ A\|_2}+\frac{112C_{L}C^3_{\mathcal{A}}}{d^{\alpha}}\left(1+\frac{8C_{L}C^3_{\mathcal{A}}}{d^{\alpha}}\right)^3.\label{hg-gen}
\end{align}
By Proposition~\ref{l1-l2-connect} and the fact that $\|g\|_{2}\geq\|g\|_{1}$, we have
\begin{align}
&\frac{\inf_{a,s,s' \in \mathbb{R}, A,A'\in\mathcal{A}} \|af\circ A'(\cdot+s,\cdot+s') -g\circ A\|_{L^2(\mathbb{R}^2)}^2}{\|g\|_2^2} \nonumber\\
\leq&8C_{\mathcal{A}}^2C_{L}^2\frac{\inf_{a,s,s' \in \mathbb{R}, A,A'\in\mathcal{A}} \|af\circ A'(\cdot+s,\cdot+s') -g\circ A\|_{L^2(\mathbb{R}^2)}^2}{\|g\circ A\|_2^2}\nonumber\\
\leq&8C_{\mathcal{A}}^2C_{L}^2\left\|\frac{1}{\|f\circ A'\|_{2}}f\circ A'\left(\cdot+\frac{r}{d},\cdot+\frac{r'}{d}\right)-\frac{1}{\|g\circ A\|_{2}}g\circ A\right\|_{L^2(\mathbb{R}^2)}^2\nonumber\\
=&8C_{\mathcal{A}}^2C_{L}^2\left(\int_{\mathbb{R}^2} \frac{\left(g\circ A(u,v)\right)^2}{\|g\circ A\|_2^2}+\frac{\left(f\circ A'\left(u+\frac{r}{d},v+\frac{r'}{d}\right)\right)^2}{\|f\circ A'\|_2^2}-\frac{2\cro{g\circ A(u,v)}\cro{f\circ A'\left(u+\frac{r}{d},v+\frac{r'}{d}\right)}}{\|g\circ A\|_2\|f\circ A'\|_2}dudv\right)\nonumber\\
\leq&16C_{\mathcal{A}}^2C_{L}^2\left(1-\frac{\int_{\cro{0,1}^{2}}\cro{f\circ A'\left(u+\frac rd,v+\frac{r'}{d}\right)}\cro{g\circ A(u,v)}dudv}{\|f\circ A'\|_2\|g\circ A\|_2}\right).\label{lower-connect-int}
\end{align}
The integral in the last step can be restricted to $[0,1]^2$ because by the assumptions, the support of the function $g\circ A$ is contained in $[0,1]^2.$ Rewriting the previous inequality gives  
\begin{align*}
\frac{\int_{\cro{0,1}^{2}}\cro{f\circ A'\left(u+\frac rd,v+\frac{r'}{d}\right)}\cro{g\circ A(u,v)}dudv}{\|f\circ A'\|_2\|g\circ A\|_2}&\leq 1-\frac{\inf_{a,s,s' \in \mathbb{R}, A,A'\in\mathcal{A}} \|af\circ A'(\cdot+s,\cdot+s') -g\circ A\|_{L^2(\mathbb{R}^2)}^2}{16C_{\mathcal{A}}^2C_{L}^2\|g\|_2^2}\\
&=1-\frac{D^2(f,g)}{16C_{\mathcal{A}}^2C_{L}^2},
\end{align*}
with $D(f,g)$ as in \eqref{eq.D_part-cnn}.
By interchanging the role of $g$ and $f$ in \eqref{lower-connect-int}, we finally get the upper bound $1-(D^2(f,g)\vee D^2(g,f))/(16C_{\mathcal{A}}^2C_{L}^2)$ in the previous inequality. Together with \eqref{hg-gen}, the asserted inequality in (ii) follows.
\end{proof}

The next lemma shows how one can compute the maximum of a $r$-dimensional vector with a fully connected neural network.

\begin{lemma}
\label{lemma_max}
There exist networks $\mathsf{Max}^r, \mathsf{Max}_r \in \mathcal{F}_{\mathsf{id}}(1+2\lceil \log_2 r\rceil, (r, 2r, \dots, 2r, 1)),$ such that 
\begin{align*}
    \mathsf{Max}^r( \mathbf{x}) = \max\{x_1, \dots, x_r\} \quad \text{and} \ \ 
    \mathsf{Max}_r( \mathbf{x}) = r\cdot\max\{x_1, \dots, x_r\},
    \quad \text{for all} \quad \mathbf{x} = (x_1, \dots, x_r) \in[0,\infty)^r.
\end{align*}
In both networks all network parameters are bounded in absolute value by $1.$
\end{lemma}

\begin{proof}
Due to the identity $\max\{y,z\} = ((y-z)_++z)_+$ that holds for all $y,z\geq0,$ one can compute $\max\{y,z\}$ by a network $\mathsf{Max}(y,z)$ with two hidden layers and width vector $(2,2,1,1)$. This network construction involves \textit{five} non-zero weights, all bounded in absolute value by 1.

In a second step we now describe the construction of the $\mathsf{Max}^r$ network.
Let $q=\lceil \log_2 r \rceil$. In the first hidden layer the network computes the padded vector
\begin{align}
    (x_1, \dots, x_r) \mapsto (x_1, \dots, x_r, \underbrace{0, \dots, 0}_{2^q-r}).
    \label{eq.9631}
\end{align}
This requires $r$ non-zero network parameters, corresponding to the identity mapping for the first $r$ inputs; the remaining $2^q-r$ outputs are constant zeros and require no non-zero parameters.

Next we apply the network $\mathsf{Max}(y,z)$ from above to the pairs $(x_1, x_2),$ $(x_3, x_4), \dots, (0,0)$
in order to compute
\begin{align*}
    \big(\mathsf{Max}(x_1, x_2), \mathsf{Max}(x_3,x_4), \dots, \mathsf{Max}(0,0)\big) \in [0,\infty)^{2^{q-1}}.
\end{align*}
This reduces the length of the vector by a factor two. By consecutively pairing neighboring entries and applying the network $\mathsf{Max},$ the procedure is continued until there is only one entry left. Together with the layer \eqref{eq.9631}, the resulting network $\mathsf{Max}^r$ has $2q+1$ hidden layers. It can be realized by taking width vector $(r,2r,2r, \dots, 2r,1)$. We have $\mathsf{Max}(y,z)=\max\{y,z\}$ and thus also $\mathsf{Max}^r(x_1,\ldots,x_r)=\max\{x_1, \dots, x_r\},$ proving the assertion.

To obtain $\mathsf{Max}_r(\bx)=r \cdot \max\{x_1, \dots, x_r\}$, observe that
\[
r \cdot \max\{x_1, \dots, x_r\} = \sum_{j=1}^r \max\{x_1, \dots, x_r\}.
\]
Therefore, it suffices to follow the same construction as for $\mathsf{Max}^r$ up to the last hidden layer. In the final hidden layer, we replace the single computation of $\mathsf{Max}$ with $r$ parallel copies, and in the output layer, we sum their outputs.
\end{proof}

One can reduce the number of required layers to $1+\lceil \log_2 r\rceil$ on the cost of a more involved proof. 

\begin{proof}[Proof of Theorem \ref{th2-gen}]
In the first step of the proof, we explain the construction of the CNN. Under Assumption~\ref{ass-cover}, for any $\alpha \in (0,1]$, there exists a finite subset $\mathcal{A}_{d_{\alpha}}$ that forms a $d^{-\alpha}$-covering of $\mathcal{A}$. We define for every $A\in\mathcal{A}_{d_{\alpha}}$ and for any of the classes $k \in \{0,1\}$, a matrix $ \mathbf{w}_{(f_k\circ A)_{\alpha}}=(w_{(f_k\circ A)_{\alpha}, j, \ell})_{j,\ell}$ with entries 
\begin{align*}
     w_{(f_k\circ A)_{\alpha}, j, \ell}=  \frac{\overline{\big((f_{k}\circ A)_{\alpha}\big)}_{j,\ell}}{d\|(f_{k}\circ A)_{\alpha}\|_{2,d}}.
\end{align*}
The corresponding filter is then defined as the quadratic support $[ \mathbf{w}_{(f_k\circ A)_{\alpha}}]$ of the matrix $ \mathbf{w}_{(f_k\circ A)_{\alpha}}$. Since we have $|\mathcal{A}_{d_{\alpha}}|$ possible choices for the deformations and two different template functions $f_0$ and $f_1$, this results in at most $2|\mathcal{A}_{d_{\alpha}}|$ different filters. Since each filter corresponds to a feature map $\sigma([ \mathbf{w}_{(f_k\circ A)_{\alpha}}] \star \overline\bX),$ we also have at most $2|\mathcal{A}_{d_{\alpha}}|$ feature maps. Among those, half of them correspond to class zero and the other half to class one. 

Now, a global max-pooling layer is applied to the output of each filter map. As explained before, in our framework the max-pooling layer extracts the signal with the largest absolute value. Application of the max-pooling layer thus yields a network with outputs
\begin{align*}
 \mathbf{O}_{k,A}(\ol\bX)= \big| \sigma([ \mathbf{w}_{(f_k\circ A)_{\alpha}}] \star\overline\bX)\big|_{\infty}.
\end{align*}
In the last step of the network construction, we take several fully connected layers, that extract, on the one hand, the largest value of $ \mathbf{O}_{0,A}(\ol\bX)$ and, on the other hand, the largest value of $ \mathbf{O}_{1,A}(\ol\bX)$, where $A\in\mathcal{A}_{d_{\alpha}}$. Applying two networks $\mathsf{Max}^r$ from Lemma \ref{lemma_max} and with $r=|\mathcal{A}_{d_{\alpha}}|$ in parallel leads to a network with two outputs
\begin{align}
\label{z1z2}
    \Big(\max\Big\{ \mathbf{O}_{0,A}(\ol\bX): A\in\mathcal{A}_{d_{\alpha}}\Big\}, \max\Big\{ \mathbf{O}_{1,A}(\ol\bX): A\in\mathcal{A}_{d_{\alpha}}\Big\}\Big).
\end{align}
By Lemma \ref{lemma_max}, the two parallelized $\mathsf{Max}^r$ networks are in the network class $\mathcal{F}_{\mathsf{id}}(1+2\lceil \log_2 r\rceil, (2r, 4r, \dots, 4r, 2))$ with $r=|\mathcal{A}_{d_{\alpha}}|.$ 

In the last step, the softmax function $\Phi(x_1,x_2) = (e^{ x_1}/(e^{  x_1}+e^{ x_2}), e^{ x_2}/(e^{ x_1}+e^{ x_2}))$ is applied. This guarantees that the output of the network is a probability vector over the two classes $0$ and $1.$ The whole network construction is contained in the CNN class $\mathcal{G}(\alpha,|\mathcal{A}_{d_{\alpha}}|)$ that has been introduced in \eqref{CNN}.

We now derive a bound on the approximation error of this CNN. Denote by \(A_{*} \in \mathcal{A}\) the true deformation for the generic image \(\bX = (X_{j,\ell})_{j,\ell=1,\ldots,d}\), namely,
\[ X_{j,\ell} = d^{2}\eta \int_{I_{j,\ell}} f_{k}\big(A_{*}(u,v)\big) \, dudv. \]
Without loss of generality, we assume its label is $k=0$, so that $f_0$ is the corresponding template function. The case $k=1$ follows analogously.
By assumption, the conditions of Proposition \ref{prop.cnns} are satisfied and we conclude that there exist $A'\in\mathcal{A}_{d_{\alpha}}$ and a corresponding filter $ \mathbf{w}_{(f_{0}\circ A')_{\alpha}}$ such that 
\begin{align*}
\big|\sigma\left([ \mathbf{w}_{(f_0\circ A')_{\alpha}}] \star \overline \bX\right)\big|_{\infty}
    \geq 1-\frac{C_1(C_L,C_{\mathcal{A}})}{d^{\alpha}}.
\end{align*}
Proposition \ref{prop.cnns} (ii) further shows that all feature maps based on the template function $f_1$ are bounded by
\begin{align*}
\max_{A'\in\mathcal{A}_{d_{\alpha}}} \big|\sigma([ \mathbf{w}_{(f_1\circ A')_{\alpha}}] \star \overline\bX)\big|_{\infty} &\leq 1-\frac{D^2}{16C_{\mathcal{A}}^2C_{L}^2} + \frac{C_2(C_{L},C_{\mathcal{A}})}{d^{\alpha}},
\end{align*}
with $D$ as in \eqref{eq.D_part-cnn}. This, in turn, means that the two outputs $(z_0, z_1)$ of the network \eqref{z1z2} can be bounded by
\begin{align*}
z_0 \geq 1-\frac{C_1(C_{L},C_{\mathcal{A}})}{d^{\alpha}} \quad \text{and} \quad z_1 \leq 1-\frac{D^2}{16C_{\mathcal{A}}^2C_{L}^2} + \frac{C_2(C_{L},C_{\mathcal{A}})}{d^{\alpha}}.
\end{align*}
As the softmax function $\Phi$ is applied to the network output, there exists $\mathbf{p}=(p_1,p_2)\in\mathcal{G}(\alpha,|\mathcal{A}_{d_{\alpha}}|)$ such that
$$p_{1}(\overline\bX)=\frac{e^{ z_0}}{e^{ z_0} + e^{ z_1}},\quad p_{2}(\overline\bX)=\frac{e^{ z_1}}{e^{ z_0} + e^{ z_1}}.$$ Set $\kappa:=16C^2_{\mathcal{A}}C^2_{L}\cro{C_1(C_{L},C_{\mathcal{A}})+C_2(C_{L},C_{\mathcal{A}})+1}$. Provided $D^{2}\geq\kappa/d^{\alpha}$, we deduce that, $p_{1}(\overline\bX)>p_{2}(\overline\bX)$ hence $\1(p_{2}(\overline\bX)>1/2)=0$. This proves the assertion.
\end{proof}
\begin{proof}[Proof of Lemma \ref{lem.perfect_recovery-gen}]
For such values of $D$, Theorem 4.2 shows that there exists a function $ \mathbf{p}=(p_1,p_2)$ belonging to $\mathcal{G}(\alpha,|\mathcal{A}_{d_{\alpha}}|)$ such that $\1(p_2(\overline 
\bX)>1/2)=k(\bX)$. This shows that $k$ can be written as a deterministic function evaluated at $\bX.$ To see that $k(\bX)$ equals the conditional class probability $p(\bX),$ observe that 
\begin{align*}
    p(\bX)= \mathbf{P}(k=1| \mathbf{X})
= \mathbf{E}\big[\1(k=1)\big| \mathbf{X}\big]
    = \1(k(\bX)=1)=k(\bX).
\end{align*}
\end{proof}

To facilitate our later proofs, we first introduce the VC-classes (Vapnik-Chervonenkis-Classes) of sets as follows.
\begin{defi}
Let $\mathcal{A}$ be a class of subsets of $\mathcal{X}$ with $\mathcal{A} \neq \emptyset$. Then the VC-dimension (Vapnik-Chervonenkis-Dimension) $V_{\mathcal{A}}$ of $\mathcal{A}$ is 
\begin{align*}
V_{\mathcal{A}} := \sup\{m \in \mathbb{N}: S(\mathcal{A},m) = 2^m\},
\end{align*}
where $S(\mathcal{A},m) := \max_{\{\bx_1, \dots, \bx_m\} \subseteq \mathcal{X}} |\{A \cap \{\bx_1, \dots, \bx_m\}: A \in \mathcal{A}\}|$ denotes the $m$-th shatter coefficient.
\end{defi}
To extend the notion of VC-classes from sets to functions, we use the following definition.
\begin{defi}
Recall that an (open) subgraph of a function $f:\mathcal{X}\rightarrow\mathbb{R}$ in $\mathcal{F}$ is the subset of $\mathcal{X}\times\mathbb{R}$ given by
\begin{equation*}
\mathcal{S}_{f}=\left\{(x,t)\in\mathcal{X}\times\mathbb{R}:\;t<f(x)\right\}.
\end{equation*}
A collection $\mathcal{F}$ of measurable functions on $\mathcal{X}$ is called a VC-subgraph class, or just a VC-class, with dimension not larger than $V$ if, $\mathcal{F}^{+}:=\{\mathcal{S}_{f}:\;f\in\mathcal{F}\}$ is a VC-class of sets in $\mathcal{X}\times\mathbb{R}$ with dimension $V_{\mathcal{F}^{+}}$ not larger than $V$.
\end{defi}

The following result states a fundamental property of VC-subgraph classes.
\begin{lemma}
\label{B7}
Let $\mathcal{F}$ be a family of real-valued functions on $\mathcal{X}$, and let $g: \mathbb{R} \to \mathbb{R}$ be a fixed monotone function. Define the class $\mathcal{G}=\{g \circ f: f \in \mathcal{F}\}$. If $\mathcal{F}$ is a VC-class with VC-dimension $V_{\mathcal{F}^+}$, then $\mathcal{G}$ is also a VC-class, and its VC-dimension $V_{\mathcal{G}^+}$ satisfies
\begin{align*}
    V_{\mathcal{G}^+} \leq V_{\mathcal{F}^+}.
\end{align*}
\end{lemma}
\begin{proof}
See Lemma 2.6.18-(viii) in \cite{MR1385671}.
\end{proof}

The following oracle inequality decomposes the excess misclassification probability of the estimator into two terms, namely a term measuring the complexity of the function class and a term measuring the approximation power. The complexity is measured via the VC-dimension introduced above. This decomposition will serve as the foundation for the proof of Theorem~\ref{thm4}. 

\begin{lemma}[Corollary 5.3 in \cite{B05}]
\label{le5}
Assume that $(\bX_1,k_1),\ldots,(\bX_n,k_n)$ are i.i.d.\ copies of a random vector $(\bX,k)\in \mathcal{X}\times \{0,1\}.$ Let $\wh{g}_n$ be the classifier
\begin{align*}
    \wh{g}_n \in \argmin_{g \in \mathcal{C}}\frac{1}{n}\sum_{i=1}^{n}\1\left(g({\bX}_i)\not=k_{i}\right),
\end{align*}
based on a function class $\mathcal{C}\subseteq \{f: \mathcal{X} \to\{0,1\}\}$ with finite VC-dimension $V$. Then, there exists a universal constant $C_{1}$ such that, for any positive integer $n,$ with probability at least $1-\delta$,
\begin{align*}
    \PROB\{\wh{g}_n(\bX)\not=k|\mathcal{D}_{n}\}-\inf_{g\in\mathcal{C}}\PROB\{g(\bX)\not=k\}\leq C_{1}\left(\sqrt{\inf_{g\in\mathcal{C}}\PROB\{g(\bX)\not=k\}\frac{V\log n+\log{\frac{1}{\delta}}}{n}}+\frac{V\log n+\log{\frac{1}{\delta}}}{n}\right).
\end{align*}
\end{lemma}

If the $\Phi$ in the class $\mathcal{G}(\alpha,m),$ as defined in \eqref{CNN}, is replaced by the identity map $\id,$ we denote the resulting class by $\mathcal{G}_{\mathrm{id}}(\alpha,m).$ Let 
\begin{align}
    \mathcal{H}_{\rho}(\alpha,m) := \left\{\rho \circ (f_1 - f_2) : (f_1,f_2) \in \mathcal{G}_{\mathrm{id}}(\alpha,m)\right\}
    \label{eq.Hbeta_def}
\end{align}
with
\begin{align*}
\rho(z) = \frac{1}{1+e^{z}}.
\end{align*}
In fact, in the minimization problem \eqref{eq.p_LS}, only the second component of $(q_1, q_2)$ is considered, as $q_1$ is determined by $q_2$ via $q_1 = 1 - q_2$. According to definitions \eqref{eq.def_beta_SM} and \eqref{CNN}, $q_2$ is given by
\begin{align*}
    \frac{e^{z_1}}{e^{z_0}+e^{z_1}}=\frac{1}{1+e^{ (z_0-z_1)}}=\rho(z_0-z_1), 
\end{align*}
with $(z_0,z_1) \in \mathcal{G}_{\mathrm{id}}(\alpha,m)$. Hence, there is a one-to-one correspondence between $\mathcal{G}(\alpha,m)$ and $\mathcal{H}_{\rho}(\alpha,m).$ The following result establishes a VC-dimension bound for the function class $\mathcal{H}_{\rho}(\alpha,m)$. 

\begin{lemma}
\label{le6}
Let $\mathcal{H}_{\rho}(\alpha, m)$ be defined as in \eqref{eq.Hbeta_def} for any $\alpha \in (0,1]$ and $m \geq 2$. Then, there exists a universal constant $C_{2}>0$ such that the VC-dimension $V_{\mathcal{H}^+_{\rho}(\alpha,m)}$ of $\mathcal{H}_{\rho}(\alpha,m)$ satisfies
\begin{align*}
V_{\mathcal{H}^+_{\rho}(\alpha,m)} \leq C_{2} (m\lceil d^\alpha\rceil^2+m^2)\log^3(md).
\end{align*}
\end{lemma}

\begin{proof}
Define 
\begin{align*}
    \mathcal{H}'_{\rho}(\alpha,m) := \left\{\rho \circ (h \circ  \mathbf{g}): h \in \mathcal{F}_{\mathrm{id}}(1, (2,2,1)),  \mathbf{g} \in \mathcal{G}_{\mathrm{id}}(\alpha,m)\right\}. 
\end{align*}
The identity $f_1-f_2 = \sigma(f_1-f_2) - \sigma(f_2-f_1) \in \mathcal{F}_{\mathrm{id}}(1,(2,2,1))$ shows that the class $\mathcal{H}_{\rho}(\alpha,m)$ defined in \eqref{eq.Hbeta_def} is a subset of $\mathcal{H}'_{\rho}(\alpha,m).$ It follows that $V_{\mathcal{H}^+_{\rho}(\alpha,m)} \leq  V_{\mathcal{H}'^+_{\rho}(\alpha,m)}.$ Therefore, it is enough to derive a VC-dimension bound for $\mathcal{H}'_{\rho}(\alpha,m).$ Since $\rho$ is a fixed monotone function, Lemma \ref{B7} yields
\begin{align}
V_{\mathcal{H}'^+_{\rho}(\alpha,m)} \leq V_{(\mathcal{F}_{\mathrm{id}}(1,(2,2,1)) \circ \mathcal{G}_{\mathrm{id}}(\alpha,m))^+}.
\label{eq.757352}
\end{align}
With definition \eqref{CNN}, we can rewrite
\begin{align}
\mathcal{F}_{\mathrm{id}}\big(1,(2,2,1)\big) \circ \mathcal{G}_{\mathrm{id}}(\alpha,m) &= \left\{f \circ g: f \in \mathcal{F}_{\mathrm{id}}\big(3+2\lceil \log_2 m\rceil, (2m, 4m, \dots, 4m, 2, 2, 1)\big), g \in \mathcal{F}^C(\alpha,2m)\right\} \notag \\
&=: \mathcal{G}'_{\mathrm{id}}(\alpha,m).
    \label{eq.Gidprime}
\end{align}
In the following, we omit the dependence on $m$ and $\alpha$ in the function class $\mathcal{G}'_{\mathrm{id}} := \mathcal{G}'_{\mathrm{id}}(\alpha,m)$.
To bound $V_{\mathcal{G}'^+_{\mathrm{id}}}$, we apply Lemma 7 in \cite{KKW20}. In their notation, images are of size $d_1\times d_2$. The authors prove on p.26 in the supplement of \cite{KKW20} the bound 
\begin{align}
2(L_1+L_2+2)W\log_2\cro{\big(2et(L_1+L_2+2)k_{\max}d_1d_2\big)^4},\label{ref-vc}
\end{align}
where $L_1$ is the number of convolutional layers, $L_2$ is the number of hidden layers in the fully connected network, $t$ is the input dimension of the fully connected layers, and $k_{\max}$ denotes the maximum of $t$, the maximal width of the fully connected layers, and the maximal number of channels in the convolutional layers. For the considered architecture, this corresponds to $L_1 = 1$,  $k_{\max} = 4m$, $t = 2m$, and $d_1 = d_2 = d$. Moreover, $W$ is the number of weights in the networks in their proof. A careful inspection of the proof shows that, in the case of weight sharing within a single filter, if $\alpha < 1$ and the filter matrix has an $(\alpha, d)$-block structure, all entries of each of the submatrices in the block structure are the same. As there are at most $\lceil d^{\alpha} \rceil^2$ submatrices, each filter contains at most $\lceil d^{\alpha} \rceil^2$ distinct weight values. Treating these distinct weight values as variables, implementing convolution with each filter (i.e., the entry-wise sum of the Hadamard product) at a given position on a fixed image input yields a polynomial of degree at most $1$ in at most $\lceil d^{\alpha} \rceil^2$ variables, rather than $d^2$. Therefore, with a slight modification, one only needs to count the distinct weight values in each filter instead of the total number of weights.

Let $W_1$ denote the sum over the number of distinct weight values in each of the $2m$ filters and $W_2$ the number of weights in the fully connected layers. We derive from \eqref{ref-vc} that
\begin{align}
V_{\mathcal{G'_{\mathrm{id}}}^+}\leq 8(L+3)(W_1+W_2)\log_2\big(2e(2m)(L+3)(4m)d^2\big)\label{eq.4632784}
\end{align}
with $L=3+2\lceil \log_2 m\rceil$ representing the number of hidden layers in the fully connected part.

Next, we derive an upper bound for $W=W_1+W_2$. As shown above, each filter matrix with $(\alpha, d)$-block structure has at most $\lceil d^{\alpha} \rceil^2$ distinct weight values and there are $2m$ filters. Consequently, $W_1\leq 2m\lceil d^{\alpha} \rceil^2$. There are $3+2\lceil \log_2 m\rceil$ hidden layers in the fully connected part and the width vector is $(2m,4m,\ldots,4m,2,2,1).$ The $4+2\lceil \log_2 m\rceil$ weight matrices have all at most $16m^2$ parameters. Moreover, there are at most $4m(4+2\lceil \log_2 m\rceil)$ bias parameters. This implies
$$W_2\leq16m^2(4+2\lceil \log_2 m\rceil)+4m(4+2\lceil \log_2 m\rceil).$$
Putting the two bounds together, it follows that
\begin{align}
    W &\leq 2m\lceil d^{\alpha}\rceil^2+16m^2(4+2\lceil \log_2 m\rceil)+4m(4+2\lceil \log_2 m\rceil)\nonumber\\
    &\leq2m\lceil d^{\alpha}\rceil^2+40m^2(2+\lceil \log_2 m\rceil)\nonumber\\
    &\leq2m\lceil d^{\alpha}\rceil^2+40m^2(3+\log_2 m)\nonumber\\
    &\leq2m\lceil d^{\alpha}\rceil^2+160m^2\log_2 m,\label{total-par-cnn}
\end{align}
using $m\geq2$ for the last step. Since also $L+3\leq 8+2\log_2 m\leq 10\log_2 m,$ \eqref{eq.4632784} can then be bounded as 
\begin{align*}
V_{\mathcal{G'_{\mathrm{id}}}^+} &\leq 80(\log_2 m)(2m\lceil d^{\alpha}\rceil^2+160m^2\log_2 m)\log_2(160em^2d^2(\log_2m))\\
&\leq C_2(m\lceil d^{\alpha}\rceil^2+m^2)\log^3(md),
\end{align*}
where $C_{2}>0$ is a universal constant. Together with \eqref{eq.757352} and \eqref{eq.Gidprime}, the result follows.
\end{proof}

\begin{lemma}
\label{indicator-function-vc}
Let $\mathcal{H}$ be a VC-class of real-valued measurable functions on $\mathcal{X}$ with VC-dimension $V_{\mathcal{H}^{+}}$. Then, the function class $$\mathcal{C}=\left\{x\mapsto \1\left(f(x)>\frac{1}{2}\right):\;f\in\mathcal{H}\right\}$$ is also a VC-class on $\mathcal{X}$ with VC-dimension at most $V_{\mathcal{H}^{+}}$.
\end{lemma}
\begin{proof}
According to Proposition 2.1 of \cite{Baraud2016}, the class $\mathcal{H}$ is weakly VC-major with dimension no larger than $V_{\mathcal{H}^{+}}$. This means that the collection of subsets
$$\mathcal{A}_{\mathcal{H}}=\left\{\left\{x\in\mathcal{X}\quad\mbox{such that}\ \  f(x)>\frac{1}{2}\right\}:\;f\in\mathcal{H}\right\},$$ is a VC-class of subsets of $\mathcal{X}$ with a dimension not larger than $V_{\mathcal{H}^{+}}$. Due to $\mathcal{C} = \left\{ \1(A),\ A \in\mathcal{A}_{\mathcal{H}} \right\}$, the conclusion follows from the property of VC-classes of functions (see, for instance, page 275 of \cite{Vaart_1998}).
\end{proof}

\begin{proof}[Proof of Theorem \ref{thm4}]
The proof is based on the application of Lemma~\ref{le5}. For any $\alpha\in(0,1]$, define $\mathcal{C}(\alpha,m):=\{f:\;f(\bx)=\1(g(\bx)>1/2),\;g\in\mathcal{H}_{\rho}(\alpha,m)\}$, where $\mathcal{H}_{\rho}(\alpha,m)$ is defined as in \eqref{eq.Hbeta_def}. Minimizing the empirical misclassification error in \eqref{eq.p_LS} can thus be reformulated as
$$\wh g\in\argmin_{f \in \mathcal{C}(\alpha,m)} \, \frac{1}{n}\sum_{i=1}^{n}\1\left(f({\ol \bX}_i)\not=k_{i}\right).$$

Let $m:=|\mathcal{A}_{d_{\alpha}}|$. Lemma \ref{le6} applied to $m=|\mathcal{A}_{d_{\alpha}}|$ yields that for any $|\mathcal{A}_{d_{\alpha}}|\geq 2$, there exists a universal constant $C_2>0$ such that
\begin{align*}
V_{\mathcal{H}_{\rho}^+(\alpha,m)} \leq C_{2}\big(\lceil d^{\alpha}\rceil^2 |\mathcal{A}_{d_{\alpha}}|+|\mathcal{A}_{d_{\alpha}}|^2\big)\log^3(d|\mathcal{A}_{d_{\alpha}}|),
\end{align*}
which, together with Lemma~\ref{indicator-function-vc}, implies that the VC-dimension $V$ of $\mathcal{C}(\alpha,|\mathcal{A}_{d_{\alpha}}|)$ is bounded by
\begin{align}
 V\leq C_{2}(\lceil d^{\alpha}\rceil^2 |\mathcal{A}_{d_{\alpha}}|+|\mathcal{A}_{d_{\alpha}}|^2)\log^3(d|\mathcal{A}_{d_{\alpha}}|).\label{eq.fhskj}
\end{align}

For $D^2\geq\kappa/d^{\alpha}$, Theorem \ref{th2-gen} implies existence of a network $\mathbf{p} = (p_1,p_2) \in \mathcal{G}(\alpha,|\mathcal{A}_{d_{\alpha}}|)$, such that the corresponding classifier $\1(p_2(\ol\bX)>1/2)=k,$ almost surely. Thus,
\begin{align}
    \inf_{ \mathbf{q}=(q_1,q_2) \in \mathcal{G}(\alpha,|\mathcal{A}_{d_{\alpha}}|)} \PROB\big(\1(q_2(\ol\bX)>1/2)\not=k\big) =\inf_{ g\in \mathcal{C}(\alpha,|\mathcal{A}_{d_{\alpha}}|)} \PROB\big(g(\ol\bX)\not=k\big)=0.\label{appro-zero}
\end{align}
Applying Lemma \ref{le5} and using \eqref{appro-zero}, we obtain that for any $|\mathcal{A}_{d_{\alpha}}|\geq 2$, with probability at least $1-\delta$,
\begin{align*}
\PROB\big(\wh{k}(\bX)\not=k\big| \mathcal{D}_n\big)=\PROB\Big(\1\big(\wh{p}_2(\ol\bX)>1/2\big)\not=k\big| \mathcal{D}_n\Big)\leq C_{1}V\frac{ \log n+\log{\frac{1}{\delta}}}{n}.
\end{align*}
For a random variable $Z$, writing $Z_{+}=\max\{Z,0\}$ yields $ \mathbf{E}[Z]\leq \mathbf{E}[Z_{+}]=\int_0^{+\infty} \PROB(Z_{+}> t) \, dt=\int_0^{+\infty} \PROB(Z> t) \, dt.$ Thus integration with respect to $\delta$ gives \begin{align*}
\frac{n}{C_{1}V} \mathbf{E}_{\mathcal{D}_{n}}\cro{\PROB\big(\wh{k}(\bX)\not=k\big| \mathcal{D}_n\big)}-\log n&\leq \int_{0}^{+\infty}\PROB\left(\frac{n}{C_{1}V}\PROB\big(\wh{k}(\bX)\not=k\big| \mathcal{D}_n\big)-\log n>t\right)dt\\
&=\int_{0}^{+\infty}\PROB\left(\PROB\big(\wh{k}(\bX)\not=k\big| \mathcal{D}_n\big)>C_{1}V\frac{\log n+t}{n}\right)dt\\
&\leq\int_{0}^{+\infty}e^{-t}dt\\
&=1,
\end{align*}
which together with \eqref{eq.fhskj} implies that for $C=C_{1}C_{2}>0$,
$$ \PROB\big(\wh{k}(\bX)\not=k\big) = \mathbf{E}_{\mathcal{D}_{n}}\cro{\PROB\big(\wh{k}(\bX)\not=k\big| \mathcal{D}_n\big)}\leq\frac{C (\lceil d^{\alpha}\rceil^2 |\mathcal{A}_{d_{\alpha}}|+|\mathcal{A}_{d_{\alpha}}|^2)\log^3(d|\mathcal{A}_{d_{\alpha}}|) (\log n+1)}{n},$$
with $\PROB$ the distribution over all randomness in the data and the new sample $\bX.$ The assertion then follows from $\lceil d^{\alpha}\rceil^2\leq4d^{2\alpha}$.
\end{proof}

We next proceed to prove Theorem~\ref{CE-bound}. Here, it is more convenient to work with labels in $\{-1,1\}$ instead of $\{0,1\}$. Accordingly, for any labeled image $(\bX, k)$ with $k \in \{0,1\}$, we define $$\ol{k} := 2k - 1 \in \{-1,1\}$$ to denote its corresponding $\{-1,1\}$ label. Recall that we use $\overline{\bX}_i$ to denote the normalized image $\bX_i$, as defined in \eqref{eq.97bc}. Similarly, $\overline{\bX}$ denotes the normalized version of $\bX$. The next lemma establishes two auxiliary inequalities that will be used in the proof of Theorem~\ref{CE-bound}.

\begin{lemma}\label{bound-E2-l1}
Let $\tilde{f}$ be a measurable function such that for any random pair $\big(\ol{\bX},\ol{k}\big) \in \mathcal{X} \times \{-1, 1\}$ and some positive constant $K$, we have $\tilde{f}(\ol{\bX}) = K \ol{k}$, almost surely. For $\varphi(x) = \log(1 + e^{-x})$ and any measurable function $f$ with sup-norm bounded by $K$,
\begin{equation}\label{2-to-1-conn}
\mathbf{E}\cro{\Big(\varphi\big(\ol{k}f(\ol\bX)\big)-\varphi\big(\ol{k}\tilde f(\ol\bX)\big)\Big)^2}\leq \log(1+e^K)\left(\mathbf{E}\cro{\varphi\big(\ol{k}f(\ol\bX)\big)}-\mathbf{E}\cro{\varphi\big(\ol{k}\tilde f(\ol\bX)\big)}\right),    
\end{equation}
and
\begin{equation}\label{1-to-l1-conn}
\mathbf{E}\left[\varphi\big(\ol{k} f(\ol{\bX})\big)\right] - \mathbf{E}\left[\varphi\big(\ol{k} \tilde{f}(\ol{\bX})\big)\right] \geq \frac{1}{1 + e^{K}} \mathbf{E}\left[\big|f(\ol{\bX}) - \tilde{f}(\ol{\bX})\big|\right].
\end{equation}
\end{lemma}
\begin{proof}
Fix a point $\ol{\bX} = \bx$ such that $\tilde{f}(\bx) = K\ol{k}$ with $\bar{k} \in \{-1,1\}$.
By the definition of $\tilde{f}$, it follows that $\ol{k} \tilde{f}(\bx) = K$. Since $\varphi(x)$ is monotone decreasing and $\ol{k} f(\bx) \leq K$ by definition, we have
$$\varphi(\ol{k} f(\bx)) \geq \varphi(K) = \varphi(\ol{k} \tilde{f}(\bx))\geq0,$$ which implies
\begin{align*}
\big(\varphi(\ol{k} f(\bx)) - \varphi(\ol{k} \tilde{f}(\bx))\big)^2
&=\big(\varphi(\ol{k} f(\bx)) - \varphi(\ol{k} \tilde{f}(\bx))\big)\big(\varphi(\ol{k} f(\bx)) - \varphi(\ol{k} \tilde{f}(\bx))\big)\\
&\leq\varphi(\ol{k} f(\bx))\big(\varphi(\ol{k} f(\bx)) - \varphi(\ol{k} \tilde{f}(\bx))\big)\\
&\leq\log(1 + e^K)\big(\varphi(\ol{k} f(\bx)) - \varphi(\ol{k} \tilde{f}(\bx))\big),
\end{align*}
where the last inequality follows from the fact that $\ol{k} f(\bx) \geq -K$. Taking expectation with respect to $\ol{\bX}$ yields \eqref{2-to-1-conn}.

To show \eqref{1-to-l1-conn}, we again fix a point $\ol\bX=\bx$ such that $\tilde f(\bx)=K\ol{k}$. It follows from Taylor expansion that
\begin{align*}
\varphi(\ol{k}f(\bx))-\varphi(\ol{k}\tilde f(\bx))&=\frac{e^{-K}}{1+e^{-K}}\cro{\tilde f(\bx)\ol{k}-f(\bx)\ol{k}}+\frac{e^{-\ol{k}\gamma}}{2(1+e^{-\ol{k}\gamma})^2}\cro{\tilde f(\bx)\ol{k}-f(\bx)\ol{k}}^2,
\end{align*}
where $\gamma$ takes values between $f(\bx)$ and $\tilde f(\bx)$. Thus, we have
\begin{align*}
\varphi(\ol{k}f(\bx))-\varphi(\ol{k}\tilde f(\bx))&\geq\frac{e^{-K}}{1 + e^{-K}}\cro{\tilde f(\bx)\ol{k} - f(\bx) \ol{k}}=\frac{e^{-K}}{1 + e^{-K}}\big|f(\bx)- \tilde f(\bx)\big|.
\end{align*}
Taking expectation with respect to $\ol{\bX}$ completes the proof.
\end{proof}

The following lemma bounds the metric entropy (i.e., the logarithm of the covering number) with respect to the sup-norm for a function class needed in the proof of Theorem~\ref{CE-bound}.
\begin{lemma}\label{entropy-sup}
Given any $\alpha\in(0,1]$ and $m\geq2$, define the function class $$\mathcal{H}(\alpha,m):= \left\{\gamma_2-\gamma_1:\; (\gamma_1,\gamma_2) \in \mathcal{G}_{\mathrm{id}}(\alpha,m)\right\}$$ on the domain $[0,1]^{d\times d}$,  where $\mathcal{G}_{\mathrm{id}}(\alpha,m)$ denotes the class obtained by replacing $\Phi$ with the identity map $\mathrm{id}$ in \eqref{CNN}. For any $\delta>0,$ denoting $L=1+2\lceil\log_2m\rceil$, we have
$$\log\mathcal{N}(\delta,\mathcal{H}(\alpha,m),\|\cdot\|_{\infty})\leq\big(2m\lceil d^{\alpha}\rceil^2 + 160m^2 \log_2 m\big)\log\left(\frac{4d^2(4m+1)^{L+1}(L+1)}{\delta}\right).$$ 
\end{lemma}
\begin{proof}
Let $g_1, g_2 \in \mathcal{F}^{C}(\alpha,2m)$, where $\mathcal{F}^{C}(\alpha,2m)$ is defined as in \eqref{cnn-layer-def}, be two networks from the convolutional layer with filter sets $\{\mathbf{W}_{g_1,s}\}_{s=1}^{2m}$ and $\{\mathbf{W}_{g_2,s}\}_{s=1}^{2m}$, respectively. Suppose that the parameters of $g_1$ and $g_2$ differ by at most $\varepsilon$ in each corresponding entry of their filter matrices. By the definition, the filter parameters lie in $[-1, 1]$, and the input image $\bx$ is a $d \times d$ matrix with entries in $[0, 1]$. Then, for any $i, j$, implementing the convolution, i.e., the entry-wise sum of the Hadamard product, we have for all $s \in \{1, \ldots, 2m\}$,
\begin{align*}
\big|([\mathbf{W}_{g_1,s}] \star \bx)_{i,j}-([\mathbf{W}_{g_2,s}] \star \bx)_{i,j}\big|\leq d^2\varepsilon,
\end{align*}
which implies that after applying the activation function $\sigma$ pointwise and performing global max-pooling, the output satisfies
\begin{align*}
\big|\mathbf{O}_{g_1,s}(\bx)-\mathbf{O}_{g_2,s}(\bx)\big|&=\big||\sigma([\mathbf{W}_{g_1,s}] \star \bx)|_{\infty}-|\sigma([\mathbf{W}_{g_2,s}] \star \bx)|_{\infty}\big|\\
&\leq\max_{i,j}\big|([\mathbf{W}_{g_1,s}] \star \bx)_{i,j}-([\mathbf{W}_{g_2,s}] \star \bx)_{i,j}\big|\\
&\leq d^2\varepsilon.
\end{align*}
Let $h_1, h_2 \in \mathcal{F}_{\mathrm{id}}\big(1 + 2\lceil \log_2 m \rceil, (2m, 4m, \ldots, 4m, 2)\big)$ be two fully connected networks, as defined in \eqref{fully-layers-def}, whose parameters differ by at most $\varepsilon$ at each corresponding entry of the weight matrices and bias vectors. Given an input vector $\by \in \mathbb{R}^{2m}$, we denote by $(h(\by))_i$, the $i$-th output ($i = 1, 2$) of a network $h \in \mathcal{F}_{\mathrm{id}}\big(L, \mathbf{p}\big)$ with $L = 1 + 2\lceil \log_2 m \rceil$ hidden layers and width vector $\bp = (p_0, p_1, \ldots, p_L, p_{L+1}) = (2m, 4m, \ldots, 4m, 2)$. Since, by definition, all network parameters of $h$ are bounded in absolute value by $1$, it follows from the proof of Lemma~5 in~\cite{Sch17} that $h$ is Lipschitz in the sense that
\begin{equation}\label{cover-sum-1}
|h(\bx) - h(\by)|_{\infty} \leq \left( \prod_{\ell=0}^{L} p_{\ell} \right) |\bx - \by|_{\infty} = 2^{2L+1} m^{L+1} |\bx - \by|_{\infty}.
\end{equation}
Moreover, observe that for any $\bx$ with entries in $[0, 1]$, any $g \in \mathcal{F}^{C}(\alpha,2m)$, we have $0\leq |g(\bx)|_{\infty}\leq d^2$. With a similar argument as in the proof of Lemma~5 in~\cite{Sch17}, we obtain that
\begin{equation}\label{cover-sum-2}
\big|(h_1(g_2(\bx)))_i-(h_2(g_2(\bx)))_i\big|\leq d^2(L+1)\left(\prod_{\ell=0}^{L}(p_{\ell}+1)\right)\varepsilon.
\end{equation}
Applying \eqref{cover-sum-1} and \eqref{cover-sum-2}, we can further bound, for $i = 1, 2$,
\begin{align}
\big|(h_1(g_1(\bx)))_i-(h_2(g_2(\bx)))_i\big|&\leq\big|(h_1(g_1(\bx)))_i-(h_1(g_2(\bx)))_i+(h_1(g_2(\bx)))_i-(h_2(g_2(\bx)))_i\big| \nonumber\\
&\leq\big|(h_1(g_1(\bx)))_i-(h_1(g_2(\bx)))_i\big|+\big|(h_1(g_2(\bx)))_i-(h_2(g_2(\bx)))_i\big| \nonumber\\
&\leq2^{2L+1}m^{L+1}|g_1(\bx)-g_2(\bx)|_{\infty}+d^2(L+1)\left(\prod_{\ell=0}^{L}(p_{\ell}+1)\right)\varepsilon\nonumber\\
&\leq2^{2L+1}m^{L+1}d^2\varepsilon+(2m+1)(4m+1)^{L}(L+1)d^2\varepsilon,\nonumber\\
&\leq(4m+1)^{L+1}(L+1)d^2\varepsilon.\label{cover-difference}
\end{align}
Define $f_1(\bx):=(h_1(g_1(\bx)))_2-h_1(g_1(\bx)))_1$ and $f_2(\bx):=(h_2(g_2(\bx)))_2-h_2(g_2(\bx)))_1$. It then follows from \eqref{cover-difference} that
\begin{align}
|f_1(\bx)-f_2(\bx)|&=\big|\cro{(h_1(g_1(\bx)))_2-h_1(g_1(\bx)))_1}-\cro{(h_2(g_2(\bx)))_2-h_2(g_2(\bx)))_1}\big|\nonumber\\
&\leq\big|(h_1(g_1(\bx)))_1-(h_2(g_2(\bx)))_1\big|+\big|(h_1(g_1(\bx)))_2-(h_2(g_2(\bx)))_2\big|\nonumber\\
&\leq2(4m+1)^{L+1}(L+1)d^2\varepsilon.\label{final-diff-cover}
\end{align}
Similar to the derivation in~\eqref{total-par-cnn}, the total number of parameters in the considered CNN can be bounded by $2m\lceil d^{\alpha}\rceil^2 + 160m^2 \log_2 m$, assuming $m \geq 2$. Since all parameters are bounded in absolute value by one, \eqref{final-diff-cover} implies that we can discretize them using a grid of size $\varepsilon = \delta/[2d^2(4m + 1)^{L + 1} (L + 1)]$ to obtain a $\delta$-covering set of $\mathcal{H}(\alpha,m)$. This implies that the covering number satisfies
$$\mathcal{N}(\delta,\mathcal{H}(\alpha,m),\|\cdot\|_{\infty})\leq\left(\frac{4d^2(4m+1)^{L+1}(L+1)}{\delta}\right)^{2m\lceil d^{\alpha}\rceil^2 + 160m^2 \log_2 m}.$$
Taking the logarithm on both sides completes the proof.
\end{proof}

\begin{lemma}[Theorem 3 of \cite{ShenWong}]\label{ShenWong}
Let $\mathcal{F}$ be a class of functions bounded above by $F$. Assume that $\mathbf{E}[f(Z)]=0$, for any $f\in\mathcal{F}$ and a constant $v>0$ exists such that $\sup_{f\in\mathcal{F}}\Var(f(Z))\leq v$. For $M>0$ and $\zeta\in(0,1)$, suppose the $L^2$-bracketing entropy of class $\mathcal{F}$ satisfies
\begin{equation}\label{c1}
H^{B}_2(\sqrt{v},\mathcal{F})\leq\frac{\zeta n M^2}{8(4v+MF/3)},
\end{equation}
and
\begin{equation}\label{c2}
M\leq\frac{\zeta v}{4F},\quad \sqrt{v}\leq F,
\end{equation}
and, if $\zeta M/8<\sqrt{v}$,
\begin{equation}\label{c3}
\int_{\zeta M/32}^{\sqrt{v}}\sqrt{H^{B}_2(u,\mathcal{F})}du\leq\frac{\sqrt{n}M\zeta^{3/2}}{2^{10}}.   
\end{equation}
Then
$$\PROB^{*}\left(\sup_{f\in\mathcal{F}}\frac{1}{n}\sum_{i=1}^n\cro{f(Z_i)-\mathbf{E}f(Z_i)}\geq M\right)\leq3\exp\cro{-(1-\zeta)\frac{nM^2}{2(4v+MF/3)}},$$ where $\PROB^*$ denotes the outer probability measure.
\end{lemma}

\begin{lemma}\label{large-devi}
Let $(\ol{\bX}_1, \ol{k}_1),\ldots,(\ol{\bX}_n, \ol{k}_n)$ be $n$ i.i.d.\ copies of the random pair $(\ol\bX,\ol{k})\in\mathcal{X}\times\{-1,1\}$ and suppose $\tilde{g}(\ol{\bX}) = K \ol{k}$ almost surely, for some constant $K > 0$. Let $\mathcal{G}_{K}$ denote a class of measurable real-valued functions defined on $\mathcal{X}$ whose sup-norm is bounded by $K$, and define
\[
\widehat{g}_{n}:=\argmin_{g \in \mathcal{G}_{K}}\frac{1}{n}\sum_{i=1}^n\log\Big(1+\exp\big(-\ol{k}_ig(\ol\bX_i)\big)\Big).
\] If $\tilde g\in\mathcal{G}_{K}$ and there exists a sequence $\{\varepsilon_n\}_{n\in\mathbb{N}}$ such that 
\begin{equation}\label{entropy-bound}
\log\mathcal{N}\Big(\frac{\varepsilon_n^2}{160},\mathcal{G}_K,\|\cdot\|_{\infty}\Big)\leq \frac{n\varepsilon_n^2}{2^{14}\times6250\log(1+e^K)},
\end{equation} 
then, denoting the population risk by $\mathcal{R}(f) := \mathbf{E}[\log(1 + \exp(-\ol{k} f(\ol{\bX})))]$, we have
$$\PROB\big(\mathcal{R}(\widehat g_n)-\mathcal{R}(\tilde{g}) \geq \varepsilon_n^2\big)\leq \frac{3\exp(-\delta_n)}{1-\exp(-\delta_n)},\quad \text{with} \ \ \delta_n=\frac{n\varepsilon_n^2}{2034\log(1+e^K)}.$$ 
\end{lemma}

\begin{proof} Following \cite{park2009, KOK19}, we use chaining. For any $g \in \mathcal{G}_K$, consider the empirical process
$$
Z_n(g) := \frac{1}{n} \sum_{i=1}^n \left[ \varphi\big(\ol{k}_i \tilde{g}(\ol{\bX}_i)\big) - \varphi\big(\ol{k}_i g(\ol{\bX}_i)\big) - \mathbf{E}\left[ \varphi\big(\ol{k} \tilde{g}(\ol{\bX})\big) - \varphi\big(\ol{k} g(\ol{\bX})\big) \right] \right],
$$
where $\varphi(x):=\log(1+e^{-x}).$ For any $g\in\mathcal{G}_K,$ denote $$\mathcal{E}\big(g, \tilde{g} \big):=\mathcal{R}(g)-\mathcal{R}(\tilde{g})=\mathbf{E}[\varphi(\ol{k}g(\ol\bX))]-\mathbf{E}[\varphi(\ol{k}\tilde{g}(\ol\bX))].$$ Since $\widehat{g}_n$ minimizes the empirical risk over the class $\mathcal{G}_K$ and $\tilde{g}\in\mathcal{G}_{K}$, we have
\begin{equation}\label{ce-eq-1}
\PROB\left( \mathcal{E}\big(\widehat{g}_n, \tilde{g} \big) \geq \varepsilon_n^2 \right)
\leq 
\PROB^*\left( \sup_{g \in \mathcal{G}_K :\; \mathcal{E}(g, \tilde{g}) \geq \varepsilon_n^2} 
\frac{1}{n} \sum_{i=1}^n \left[ \varphi(\ol{k}_i \tilde{g}(\ol\bX_i)) - \varphi(\ol{k}_i g(\ol\bX_i)) \right] \geq 0 \right).
\end{equation}

To bound the right-hand side of \eqref{ce-eq-1}, we partition the function class $\{g \in \mathcal{G}_K : \mathcal{E}(g, \tilde{g}) \geq \varepsilon_n^2\}$ into a finite union of subclasses. More precisely, define for $j=1,2,\ldots$ 
$$\mathcal{G}_{j,K}:=\left\{g\in\mathcal{G}_K:\;2^{j-1}\varepsilon_n^2\leq\mathcal{E}(g,\tilde{g})<2^j\varepsilon_n^2\right\}.$$ Since $|\varphi(x_1)-\varphi(x_2)|\leq|x_1-x_2|$, for all $x_1,x_2\in\mathbb{R}$, it follows that for any $g \in \mathcal{G}_K$,
$$
\mathcal{E}(g, \tilde{g}) 
\leq \mathbf{E}\left[ \left| \varphi(\ol{k} g(\ol\bX)) - \varphi(\ol{k} \tilde{g}(\ol\bX)) \right| \right]
\leq \mathbf{E}\left[ \left|\ol{k}g(\ol\bX) - \ol{k}\tilde{g}(\ol\bX) \right| \right]
\leq 2K.
$$
Thus, for those $j$ such that $2^{j-1} \varepsilon_n^2 > 2K$, the set $\mathcal{G}_{j,K}$ is empty. With
\[
j_n^* := \max\left\{ j \in \mathbb{N} :\; 2^{j-2} \varepsilon_n^2 \leq K \right\},
\]
we have
\[
\left\{ g \in \mathcal{G}_K :\; \mathcal{E}(g, \tilde{g}) \geq \varepsilon_n^2 \right\}
\subseteq \bigcup_{j = 1}^{j_n^*} \mathcal{G}_{j,K}.
\]
Writing $T_{n,j}=2^{j-1}\varepsilon_n^2$, it follows from \eqref{ce-eq-1} that
\begin{equation}\label{ce-eq-2}
\PROB\left(\mathcal{E}\big(\widehat g_n,\tilde{g}\big)\geq\varepsilon_n^2\right)\leq\sum_{j=1}^{j_n^*}\PROB^*\left(\sup_{g\in\mathcal{G}_{j,K}}Z_n(g)\geq T_{n,j}\right).
\end{equation}
We next bound 
$\PROB^*\left( \sup_{g \in \mathcal{G}_{j,K}} Z_n(g) \geq T_{n,j} \right).$ For each $j = 1, \ldots, j_n^*$, we apply Lemma~\ref{ShenWong} to the class $$\mathcal{F}_j:=\left\{(\bx,\ol{k})\mapsto\cro{\varphi\big(\ol{k}\tilde{g}(\bx)\big)-\varphi\big(\ol{k}g(\bx)\big)}-\cro{\mathcal{R}(\tilde{g})-\mathcal{R}(g)}:\;g\in\mathcal{G}_{j,K}\right\},$$ with $\zeta=4/5$, $F=10\log(1+e^K)$, $M=T_{n,j}$, and $v=50\log(1+e^K)T_{n,j}.$ With these chosen values, all conditions of Lemma~\ref{ShenWong} are satisfied. We verify them below.

Applying Lemma~\ref{bound-E2-l1}, we can derive that for any $j,$
\begin{align*}
\sup_{f\in\mathcal{F}_j}\Var(f(\ol\bX,\ol{k}))&\leq\sup_{g\in\mathcal{G}_{j,K}}\mathbf{E}\cro{\big(\varphi(\ol{k}g(\ol\bX))-\varphi(\ol{k}\tilde{g}(\ol\bX))\big)^2}\\
&\leq\log(1+e^K)\sup_{g\in\mathcal{G}_{j,K}}\mathcal{E}(g,\tilde{g})\\&<\log(1+e^K)2^j\varepsilon_n^2\\
&=2\log(1+e^K)T_{n,j}.\end{align*}
For any $f \in \mathcal{F}_j$, the sup-norm is bounded by $4 \log(1 + e^{K})<F$, and one can verify that the conditions
$$
M \leq \frac{\zeta v}{4F} \quad \text{and} \quad \sqrt{v} \leq F
$$
are both satisfied. For any $\delta>0$, Lemma~2.1 of \cite{van2000empirical} shows that
$$H^B_2(\delta,\mathcal{F}_j)\leq \log\mathcal{N}(\delta/2,\mathcal{F}_j,\|\cdot\|_{\infty}).$$ Moreover, observe that for any $\|g_1-g_2\|_{\infty}\leq\delta/4$,
\begin{align*}
\left|\cro{\varphi\big(\ol{k}g_1(\bx)\big)-\varphi\big(\ol{k}g_2(\bx)\big)}-\cro{\mathcal{R}(g_1)-\mathcal{R}(g_2)}\right|\leq2\|g_1-g_2\|_{\infty}\leq\frac{\delta}2,    
\end{align*}
which implies
\begin{align*}
H^B_2(\delta,\mathcal{F}_j)\leq \log\mathcal{N}\Big(\frac{\delta}4,\mathcal{G}_{j,K},\|\cdot\|_{\infty}\Big)\leq \log\mathcal{N}\Big(\frac{\delta}4,\mathcal{G}_K,\|\cdot\|_{\infty}\Big), 
\end{align*}
where the last inequality follows from the inclusion $\mathcal{G}_{j,K}\subseteq\mathcal{G}_K$. It then follows that
\begin{align*}
T_{n,j}^{-1}\int_{\zeta T_{n,j}/32}^{\sqrt{50\log(1+e^K)T_{n,j}}}\sqrt{H^B_2(u,\mathcal{F}_j)}du&\leq  T_{n,j}^{-1}\int_{T_{n,j}/40}^{\sqrt{50\log(1+e^K)T_{n,j}}}\sqrt{\log\mathcal{N}\left(\frac{u}{4},\mathcal{G}_K,\|\cdot\|_{\infty}\right)}du\\ 
&\leq\sqrt{\frac{50\log(1+e^K)}{T_{n,j}}\log\mathcal{N}\left(\frac{T_{n,j}}{160},\mathcal{G}_K,\|\cdot\|_{\infty}\right)}\\
&\leq\sqrt{\frac{50\log(1+e^K)}{T_{n,j}}\frac{nT_{n,j}}{2^{14}\times6250\log(1+e^K)}}\\
&\leq\frac{\sqrt{n}\zeta^{3/2}}{2^{10}},
\end{align*}
thereby establishing condition \eqref{c3} in Lemma~\ref{ShenWong}. The above result further implies that, on one hand,
\begin{align*}
H^{B}_2(\sqrt{v},\mathcal{F}_j)&\leq\left(\frac{T_{n,j}}{\sqrt{v}-T_{n,j}/40}T_{n,j}^{-1}\int_{ T_{n,j}/40}^{\sqrt{v}}\sqrt{H^B_2(u,\mathcal{F}_j)}du\right)^2\\
&\leq\left(\frac{T_{n,j}}{\sqrt{v}-T_{n,j}/40}\cdot\frac{\sqrt{n}(4/5)^{3/2}}{2^{10}}\right)^2\\
&\leq\frac{nT_{n,j}^2}{5\times2^{18}v},
\end{align*}
where we use the fact that $8\sqrt{v}\geq T_{n,j}$ in the last inequality. On the other hand, we have
\begin{align*}
\frac{\zeta nT_{n,j}^2}{8(4v+T_{n,j}F/3)}=\frac{nT_{n,j}^2}{10(4v+T_{n,j}F/3)}=\frac{nT_{n,j}^2}{10(4+1/15)v},    
\end{align*}
which confirms that condition \eqref{c1} in Lemma~\ref{ShenWong} also holds.

Applying Lemma~\ref{ShenWong} to each $\mathcal{F}_j$, we finally obtain
\begin{align}
\PROB\left(\mathcal{E}\big(\widehat g_n,\tilde{g}\big)\geq\varepsilon_n^2\right)
\leq\sum_{j=1}^{j_n^*}3\exp\left(-\frac{nT_{n,j}^2}{10(4v+T_{n,j}F/3)}\right)
=\sum_{j=1}^{j_n^*}3\exp\cro{-\frac{n2^{j}\varepsilon_n^2}{1000(4+1/15)\log(1+e^K)}}.\label{bound-prob}
\end{align}
Setting $$\delta_n=\frac{n\varepsilon_n^2}{2034\log(1+e^K)},$$ we can derive from \eqref{bound-prob} that
\begin{align*}
\PROB\left(\mathcal{E}\big(\widehat g_n,\tilde{g}\big)\geq\varepsilon_n^2\right)\leq3  \sum_{j=1}^{\infty}\cro{\exp\left(-\frac{\delta_n}{2}\right)}^{2^j}\leq3  \sum_{j=1}^{\infty}\big(\exp(-\delta_n)\big)^{j}\leq\frac{3\exp(-\delta_n)}{1-\exp(-\delta_n)},
\end{align*}
which completes the proof.
\end{proof}

\begin{proof}[Proof of Theorem~\ref{CE-bound}]
Denote the empirical cross-entropy loss for function $g$ taking values in $(0,1)$ as
$$
\mathcal{R}_n^{\text{CE}}(g) = -\frac{1}{n} \sum_{i=1}^{n} \left[ k_i \log\big(g(\ol{\bX}_i)\big) + (1 - k_i) \log\big(1 - g(\ol{\bX}_i)\big) \right].
$$
Recall that $\ol{k}_i := 2k_i - 1$. Denote the empirical logistic loss for any real-valued function $f$ as
$$
\mathcal{R}_n^{\varphi}(f) = \frac{1}{n} \sum_{i=1}^n \varphi\big( \ol{k}_i f(\ol{\bX}_i) \big),
$$
where $\varphi(x)=\log(1 + e^{-x})$. It is known that minimizing $\mathcal{R}_n^{\text{CE}}(\cdot)$ over a function class $\mathcal{G}$ is equivalent to minimizing $\mathcal{R}_n^{\varphi}(\cdot)$ over the transformed class $\mathcal{F} = \left\{ \log\left(g/(1 - g)\right) :\; g\in \mathcal{G} \right\}$. 

For any $\alpha\in(0,1]$, recall that $\mathcal{\widetilde G}(\alpha,|\mathcal{A}_{d_{\alpha}}|)$ is defined in \eqref{g-beta-B}. Define
\begin{align*}
\mathcal{\widetilde G}^{\varphi}(\alpha,|\mathcal{A}_{d_{\alpha}}|):&=\left\{\left(\log\left(\frac{g_1}{1 - g_1}\right), \log\left(\frac{g_2}{1 - g_2}\right)\right) :\; (g_1,g_2)\in \mathcal{\widetilde G}(\alpha,|\mathcal{A}_{d_{\alpha}}|)\right\}.
\end{align*}
Since $g_1=1-g_2$, we have that for any $\bp=(p_1,p_2)\in\mathcal{\widetilde G}^{\varphi}(\alpha,|\mathcal{A}_{d_{\alpha}}|)$, $p_1=-p_2$ and $$p_2\in\mathcal{\widetilde H}(\alpha,|\mathcal{A}_{d_{\alpha}}|):=\left\{\big((f_2-f_1)\vee-1\big)\wedge 1,(f_1,f_2)\in\mathcal{G}_{\id}(\alpha,|\mathcal{A}_{d_{\alpha}}|)\right\},$$ where $\mathcal{G}_{\id}(\alpha,|\mathcal{A}_{d_{\alpha}}|)$ denotes the class obtained by replacing $\Phi$ with the identity map $\mathrm{id}$ in \eqref{CNN}. It is enough to only focus on the class $\mathcal{\widetilde H}(\alpha,|\mathcal{A}_{d_{\alpha}}|)$. Denote $$\wh p_2^{\varphi}:=\argmin_{g\in\mathcal{\widetilde H}(\alpha,|\mathcal{A}_{d_{\alpha}}|)}\mathcal{R}_n^{\varphi}(g).$$ It then follows that
$\widehat{p}_2^{\varphi} = \log\left( \widehat{p}_2^{\text{CE}}/(1 - \widehat{p}_2^{\text{CE}}) \right)$ and
\begin{equation}\label{excess-eta}
\PROB\left(\,\1(\widehat{p}_2^{\text{CE}}(\ol\bX)>1/2)\not=k\right)=\PROB\left(\,\sgn\big(\ol{k}\cdot\widehat{p}_2^{\varphi}(\ol\bX)\big)<0\right),    
\end{equation}
where $\ol{k}=2k-1.$ Let $\eta(\bx):=\mathbf{P}(k=1|\bX=\bx)$ and $k^*(\bx) = \1(\eta(\bx) > 1/2)$ be the Bayes classifier. Following from Lemma~\ref{lem.perfect_recovery-gen}, we know that under the given conditions, $\PROB(k^*(\bX)\not=k)=0$. Define ${p}_2^{\varphi}:=\argmin_{g\in\mathcal{\widetilde H}(\alpha,|\mathcal{A}_{d_{\alpha}}|)}\mathbf{E}[\varphi\big( \ol{k}g(\ol{\bX}) \big)]$ the population level minimizer. We claim that 
\begin{equation}\label{popula-risk-mini}
\PROB\big(\,\sgn\big(\ol{k}\cdot{p}_2^{\varphi}(\ol\bX)\big)<0\big)=0.
\end{equation}
In fact, following the proof of Theorem~\ref{th2-gen} and using Lemma~\ref{lemma_max}, under the given conditions, if $k=0$, there exists a network $(z_0,z_1)\in\mathcal{G}_{\id}(\alpha,|\mathcal{A}_{d_{\alpha}}|)$ such that
$$z_0 \geq |\mathcal{A}_{d_{\alpha}}|\left(1-\frac{C_1(C_{L},C_{\mathcal{A}})}{d^{\alpha}}\right)\quad \text{and} \quad z_1 \leq |\mathcal{A}_{d_{\alpha}}|\left(1-\frac{D^2}{16C_{\mathcal{A}}^2C_{L}^2} + \frac{C_2(C_{L},C_{\mathcal{A}})}{d^{\alpha}}\right).$$ Since $D^2\geq\kappa/(d^\alpha),$ where $\kappa=16C_{\mathcal{A}}^2C_L^2[C_1(C_L,C_{\mathcal{A}})+C_2(C_L,C_{\mathcal{A}})+1]$, and $|\mathcal{A}_{d_{\alpha}}|\geq d^{\alpha}$, the corresponding output is $\widetilde p_2^{\varphi}(\ol\bX)=((z_1-z_0)\vee-1)\wedge1=-1,$ almost surely. Moreover, if $k=1$, we can similarly obtain that
$\widetilde{p}_2^{\varphi}(\ol\bX)=1,$ almost surely. This implies that the population risk minimizer satisfies $p_2^{\varphi}(\ol\bX)=\ol{k}$ almost surely. Hence, ${p}_2^{\varphi}$ also achieves zero misclassification error.

Combining \eqref{popula-risk-mini} with \eqref{excess-eta}, we obtain that
\begin{align}
\PROB\left(\,\1(\widehat{p}_2^{\text{CE}}(\ol\bX)>1/2)\not=k\right)&=\PROB\left(\,\1(\widehat{p}_2^{\text{CE}}(\ol\bX)>1/2)\not=k\right)-\PROB\left(\,k^*(\bX)\not=k\right)\nonumber\\
&=\PROB\left(\,\sgn\big(\ol{k}\cdot\widehat{p}_2^{\varphi}(\ol\bX)\big)<0\right)-\PROB\left(\,\sgn\big(\ol{k}\cdot{p}_2^{\varphi}(\ol\bX)\big)<0\right). \label{tran-mis-error}
\end{align}
Thus, it suffices to bound the right-hand side of~\eqref{tran-mis-error}. Observe that, by the definition of $p_2^{\varphi}$, for any fixed $p_2\in\mathcal{\widetilde H}(\alpha,|\mathcal{A}_{d_{\alpha}}|)$,
\begin{align}
&\PROB\left(\,\sgn\big(\ol{k}\cdot p_2(\ol\bX)\big)<0\right)-\PROB\left(\,\sgn\big(\ol{k}\cdot{p}_2^{\varphi}(\ol\bX)\big)<0\right)\nonumber\\
&=\PROB\left(\,|p_2(\ol\bX)-p_2^{\varphi}(\ol\bX)|\geq 1\right)\nonumber\\
&\leq\mathbf{E}\left(|p_2(\ol\bX)-p_2^{\varphi}(\ol\bX)|\right)\nonumber\\
&\leq(1+e)\cro{\mathbf{E}\left(\varphi(\ol{k}p_2(\ol\bX))\right) - \mathbf{E}\left(\varphi(\ol{k} p_2^{\varphi}(\ol{\bX}))\right)},\label{diff-var-excess}
\end{align}
where the first inequality is a consequence of Markov's inequality, and the second follows from Lemma~\ref{bound-E2-l1}. With \eqref{diff-var-excess}, we can deduce that
\begin{align}
&\PROB\left(\,\sgn\big(\ol{k}\cdot\wh p_2^{\varphi}(\ol\bX)\big)<0\right)-\PROB\left(\,\sgn\big(\ol{k}\cdot{p}_2^{\varphi}(\ol\bX)\big)<0\right)\nonumber\\
&=\mathbf{E}_{\mathcal{D}_n}\cro{\PROB\left(\,\sgn\big(\ol{k}\cdot\wh p_2^{\varphi}(\ol\bX)\big)<0\big|\mathcal{D}_n\right)-\PROB\left(\,\sgn\big(\ol{k}\cdot{p}_2^{\varphi}(\ol\bX)\big)<0\right)} \nonumber\\
&\leq(1+e)\mathbf{E}_{\mathcal{D}_n}\cro{\mathcal{R}(\wh p_2^{\varphi})- \mathcal{R}(p_2^{\varphi})},\label{relate-to-logistic}
\end{align}
where $\mathcal{R}(f)=\mathbf{E}\left(\varphi(\ol{k}f(\ol\bX))\right)$. 

Next, we bound $\mathbf{E}_{\mathcal{D}_n}\cro{\mathcal{R}(\wh p_2^{\varphi})- \mathcal{R}(p_2^{\varphi})}$ using Lemma~\ref{large-devi} with $\mathcal{G}_K=\mathcal{\widetilde H}(\alpha,|\mathcal{A}_{d_{\alpha}}|)$, $K=1$ and $\varepsilon_n^2=2034V\log(1+e)\log^{1+\gamma}n/n$, where $V=2|\mathcal{A}_{d_{\alpha}}|\cro{\lceil d^{\alpha}\rceil^2 + 80|\mathcal{A}_{d_{\alpha}}|\log_2 (|\mathcal{A}_{d_{\alpha}}|)}.$ Observe that for any real-valued functions $g_1,g_2$ and any $\delta>0$ such that $\|g_1-g_2\|_{\infty}\leq\delta$, $$\left\|\big((g_1\vee-1)\wedge1\big)-\big((g_2\vee-1)\wedge1\big)\right\|_{\infty}\leq\|g_1-g_2\|_{\infty}\leq\delta.$$ Combining this with Lemma~\ref{entropy-sup}, we obtain that
\begin{align*}
\log\mathcal{N}\Big(\frac{\varepsilon_n^2}{160},\mathcal{\widetilde H}(\alpha,|\mathcal{A}_{d_{\alpha}}|),\|\cdot\|_{\infty}\Big)&\leq\log\mathcal{N}\Big(\frac{\varepsilon_n^2}{160},\mathcal{H}(\alpha,|\mathcal{A}_{d_{\alpha}}|),\|\cdot\|_{\infty}\Big)\\
&\leq V\log\left(\frac{d^2(4|\mathcal{A}_{d_{\alpha}}|+1)^{L+1}(L+1)}{4V\log^{1+\gamma}n/n}\right),
\end{align*}
where $L=1+2\lceil\log_2(|\mathcal{A}_{d_{\alpha}}|)\rceil.$ For any $n$ that is sufficiently large compared to $d$ and $|\mathcal{A}_{d_{\alpha}}|$, condition \eqref{entropy-bound} is fulfilled and we obtain that
\begin{equation*}
\PROB\big(\mathcal{R}(\widehat p_2^{\varphi})-\mathcal{R}(p_2^{\varphi}) \geq \varepsilon_n^2\big)\leq 6\exp(-V\log^{1+\gamma}n).
\end{equation*}
Since $0\leq\mathcal{R}(\wh p_2^{\varphi})-\mathcal{R}(p_2^{\varphi})\leq1$, taking the expectation over the training data $\mathcal{D}_n = \{({\bX}_i, {k}_i)\}_{i=1}^{n},$ with $n\geq3$ yields that
\begin{align*}
\mathbf{E}_{\mathcal{D}_n}\cro{\mathcal{R}(\wh p_2^{\varphi})-\mathcal{R}(p_2^{\varphi})}&\leq \varepsilon_n^2+  \PROB\big(\mathcal{R}(\widehat p_2^{\varphi})-\mathcal{R}(p_2^{\varphi}) \geq \varepsilon_n^2\big)\\
&\leq C'|\mathcal{A}_{d_{\alpha}}|(\lceil d^{\alpha}\rceil^2+|\mathcal{A}_{d_{\alpha}}|)\log(|\mathcal{A}_{d_{\alpha}}|)\frac{\log^{1+\gamma}n}{n},
\end{align*}
with $C'>0$ a universal constant. Therefore, together with \eqref{tran-mis-error} and \eqref{relate-to-logistic}, we finally obtain
\begin{equation*}
\PROB\left(\,\1(\widehat{p}_2^{\text{CE}}(\ol\bX)>1/2)\not=k\right)=\PROB\left(\,\sgn\big(\ol{k}\cdot\widehat{p}_2^{\varphi}(\ol\bX)\big)<0\right)\leq C|\mathcal{A}_{d_{\alpha}}|(\lceil d^{\alpha}\rceil^2+|\mathcal{A}_{d_{\alpha}}|)\log(|\mathcal{A}_{d_{\alpha}}|)\frac{\log^{1+\gamma}n}{n},
\end{equation*}
where $C>0$ is a universal constant. The conclusion then follows from $\lceil d^{\alpha}\rceil^2\leq4d^{2\alpha}$.
\end{proof}

\section{Deformed Image Examples from Simulation}\label{sim-add}
\begin{figure}[htpb]
\centering
{\includegraphics[width=0.8\textwidth]{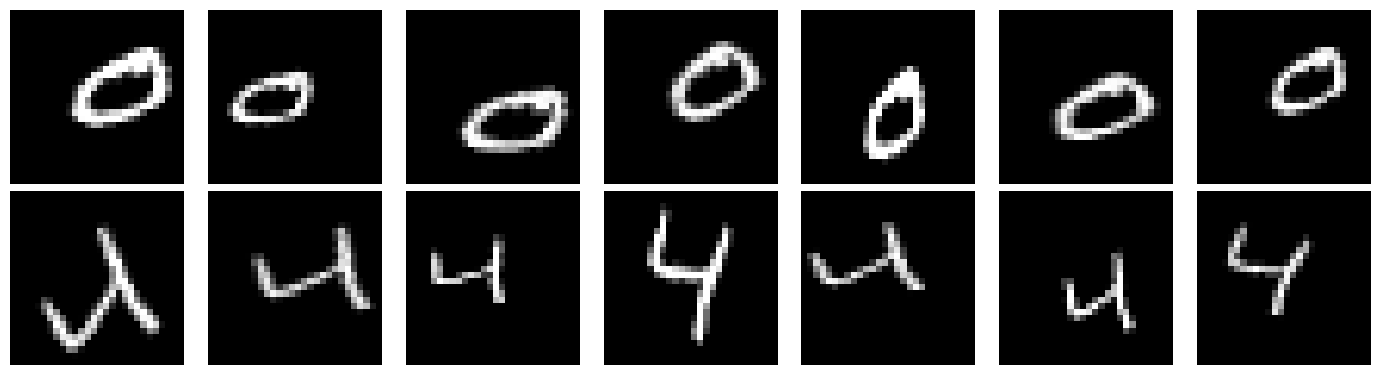}}
{\includegraphics[width=0.8\textwidth]{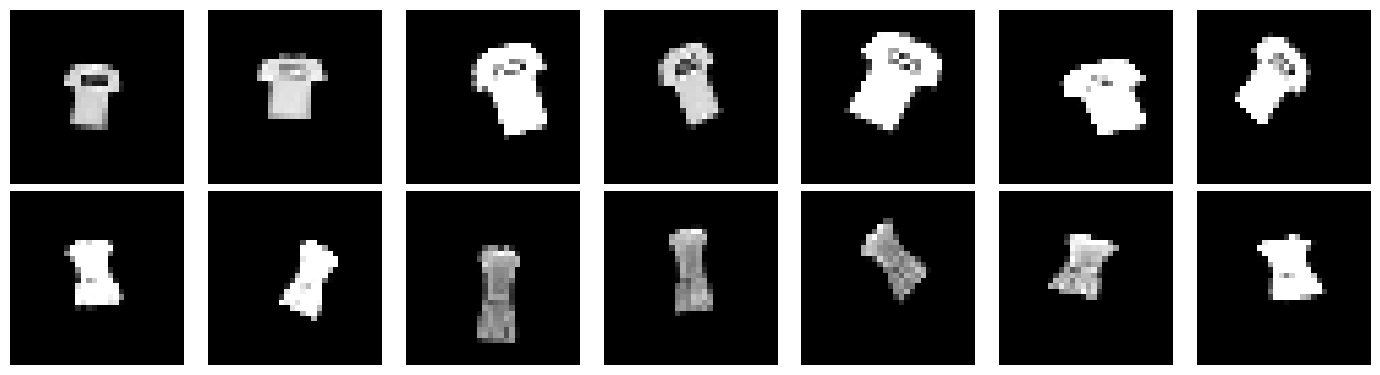}}
{\includegraphics[width=0.8\textwidth]{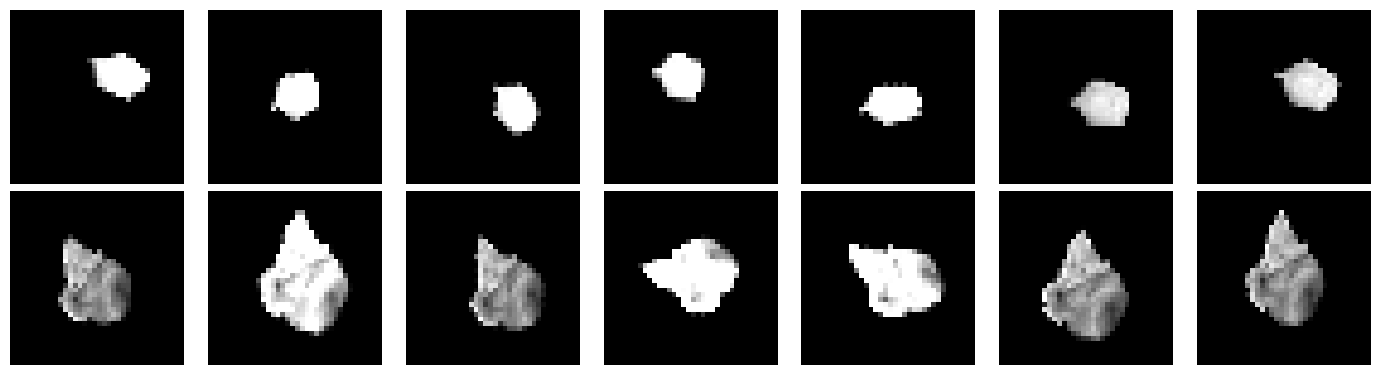}}
\caption{For template images taken from MNIST (rows 1–2), FashionMNIST (rows 3–4), and CIFAR-100 (rows 5–6), randomly generated deformations (including scaling, shifting, rotation, and brightness adjustments) are displayed.}
\end{figure}

\clearpage
\bibliographystyle{siam}
\bibliography{Literatur.bib}
\end{document}